\documentclass[a4paper,12pt,reqno]{amsart}

\usepackage{amsmath,amssymb,amsthm}
\usepackage{latexsym}
\usepackage{color}
\usepackage{graphicx}
\usepackage{mathrsfs}
\usepackage{enumerate}
\usepackage[abbrev]{amsrefs}
\usepackage[T1]{fontenc}
\usepackage{mathtools}
\mathtoolsset{showonlyrefs=true}

\usepackage[colorlinks,
            linkcolor=blue,      
           anchorcolor=blue,  
          citecolor=red,       
          ]{hyperref}
\usepackage{hyperref}

\usepackage{framed}

\allowdisplaybreaks[4]

\theoremstyle{plain}
\newtheorem{thm}{Theorem}[section]
\newtheorem{lemm}[thm]{Lemma}
\newtheorem{prop}[thm]{Proposition}

\newtheorem*{claim}{Claim}

\theoremstyle{definition}
\newtheorem{df}[thm]{Definition}
\newtheorem{rem}[thm]{Remark}

\makeatletter

\@addtoreset{equation}{section}
\makeatother

\newcommand{\dB}{\dot{B}}

\newcommand{\fB}{\dot{\mathfrak{B}}}
\newcommand{\ifB}{\mathfrak{B}}
\newcommand{\dH}{\dot{H}}

\newcommand{\supp}{\operatorname{supp}}

\renewcommand{\leq}{\leqslant}
\renewcommand{\geq}{\geqslant}

\newcommand{\Z}{\mathbb{Z}}

\newcommand{\lr}[1]{{\left\langle#1\right\rangle}}
\newcommand{\n}[1]{{\left\|#1\right\|}}
\newcommand{\abso}[1]{{\left|#1\right|}}
\newcommand{\lp}[1]{\left[#1\right]}
\newcommand{\Mp}[1]{\left\{#1\right\}}
\renewcommand{\sp}[1]{\left(#1\right)}

\newcommand{\p}{\partial}

\newcommand{\phih}{\phi^{\rm h}}
\newcommand{\phiv}{\phi^{\rm v}}

\newcommand{\Deltah}{\Delta_{\rm h}}
\newcommand{\nablah}{\nabla_{\rm h}}
\newcommand{\divh}{\operatorname{div}_{\rm h}}
\renewcommand{\div}{\operatorname{div}}

\newcommand{\xh}{x_{\rm h}}
\newcommand{\xih}{\xi_{\rm h}}
\newcommand{\Deltahh}{\dot{\Delta}^{\rm h}}
\newcommand{\Deltav}{\dot{\Delta}^{\rm v}}
\newcommand{\iDh}{\Delta^{\rm h}}
\newcommand{\iDv}{\Delta^{\rm v}}
\newcommand{\R}{\mathbb{R}}
\newcommand{\f}{\frac}
\newcommand{\betah}{\beta_{\rm h}}
\newcommand{\uh}{u_{\rm h}}

\begin{document}
\title[Analyticity for the $3$D anisotropic Navier--Stokes equations]
{Analyticity in space and time for global solutions to the anisotropic Navier--Stokes equations in the critical $L^p(\mathbb{R}^3)$ framework}
\author[M.~Fujii]{Mikihiro Fujii}
\address[M.~Fujii]{Graduate School of Science, Nagoya City University, Nagoya, 467-8501, Japan}
\email[M.~Fujii]{fujii.mikihiro@nsc.nagoya-cu.ac.jp}
\author[Y.~Li]{Yang Li}
\address[Y.~Li]{School of Mathematical Sciences and Center of Pure Mathematics, Anhui University, Hefei, 230601, People's Republic of China}
\email[Y.~Li]{lynjum@163.com}
\date{\today} 
\keywords{anisotropic Navier--Stokes, analyticity, critical space, Littlewood--Paley theory}
\subjclass[2020]{76D05, 35A20, 42B25}
\begin{abstract}
In the present paper, we consider the real analyticity of the global solutions to the $3$D incompressible anisotropic Navier--Stokes equations.
We show that if only the horizontal component of initial velocity is small and analytic in $x_3$, then there exists a unique global solution which is analytic in $t>0$ and $x\in \mathbb{R}^3$.
Our functional framework lies in some anisotropic Besov spaces based on $L^p(\mathbb{R}^3)$.
To our best knowledge, this paper is the first contribution to the well-posedness of the anisotropic Navier--Stokes equations in function spaces of the Besov type based on the full $L^p(\mathbb{R}^3)$ setting.
\end{abstract}
\maketitle

\tableofcontents

\section{Introduction}\label{sec:intro}
Let us consider the Cauchy problem for the anisotropic Navier--Stokes equations in $\R^3$, where the vertical viscosity vanishes. 
In modeling geophysical fluids and the Ekman layer, the horizontal viscosity often plays a dominant role compared to the vertical viscosity.
We refer to the monograph \cite{Ped} by Pedlosky for more explanations and physical backgrounds.
The governing equations in Eulerian coordinates take the form 
\begin{align}\label{eq:ANS}
    \begin{cases}
        \partial_t u - \Delta_{\rm h} u + (u \cdot \nabla)u + \nabla P = 0, \qquad & t>0,x \in \mathbb{R}^3, \\ 
        \div  u = 0, & t\geq 0,x \in \mathbb{R}^3, \\
        u(0,x) = u_0(x), & x \in \mathbb{R}^3.
    \end{cases}
\end{align}
Here, we denote by $u=u(t,x):[0,\infty)\times \R^3 \to \R^3$ the velocity field and $P=P(t,x):(0,\infty)\times \R^3 \to \R$ the pressure respectively, while $u_0=u_0(x):\mathbb{R}^3 \to \mathbb{R}^3$ is the given initial data.
For the representation of the viscosity, we adopt the horizontal Laplacian $\Deltah:=\nablah \cdot \nablah =\partial_{x_1}^2+\partial_{x_2}^2$ with $\nablah:=(\partial_{x_1},\partial_{x_2})$.
The central interest of the present manuscript is to consider the real analyticity of the global solutions to \eqref{eq:ANS} in the scaling critical anisotropic Besov spaces based on $L^p(\mathbb{R}^3)$. More precisely, we aim to prove that if the given initial data is small and analytic in $x_3$ only for horizontal component $u_{0,\rm h}=(u_{0,1},u_{0,2})$ of the initial velocity, then there exists a unique global solution that is real analytic in $t>0$ and $x \in \mathbb{R}^3$.
\\[3pt]
{\bf Known results related to our study.}
Before stating our main result precisely, we recall the history of the mathematical research on \eqref{eq:ANS}.
In the view point of the mathematical analysis for \eqref{eq:ANS}, the horizontal viscosity $-\Deltah u$ cannot control the first-order derivative with respect to $x_3$ in the nonlinear term.
It was initiated by Chemin et al. \cite{CDGG01} who first gave an answer to this question by using the energy method via $\partial_{x_3}u_3 = -\divh u_{\rm h}$ due to the divergence free condition, where $\divh:=\nablah \cdot$ is the horizontal divergence operator.
In \cites{CDGG01,If02}, they proved the local well-posedness for large data and global well-posedness for small data in the anisotropic Sobolev spaces $H^{0,s}(\mathbb{R}^3):= H^s(\mathbb{R}_{x_3};L^2(\mathbb{R}^2_{x_{\rm h}}))$ with $s > 1/2$.
Here, the condition $s>1/2$ is required since they used the fact that $H^s(\mathbb{R}_{x_3})$ is an algebra, whereas $s=1/2$ is the scaling-critical exponent.
However, since the embedding $H^{\frac{1}{2}}(\mathbb{R}_{x_3})$ is not an algebra, the solvability of \eqref{eq:ANS} in the critical space $H^{0,\frac{1}{2}}(\mathbb{R}^3)$ remains a difficult problem.
Paicu \cite{Pai05} focused on the end point case of the Sobolev embedding $\dB_{2,1}^{\frac{1}{2}}(\mathbb{R}_{x_3})\hookrightarrow L^{\infty}(\mathbb{R}_{x_3})$ and obtained the local well-posedness for large data and global well-posedness for small data of \eqref{eq:ANS} in the scaling critical anisotropic Besov space $\dot{\mathcal{B}}^{0,\frac{1}{2}}(\mathbb{R}^3):=\dB_{2,1}^{\frac{1}{2}}(\mathbb{R}_{x_3};L^2(\mathbb{R}^2_{x_{\rm h}}))$ with the norm
\begin{align}
    \n{f}_{\dot{\mathcal{B}}^{0,\frac{1}{2}}}
    :=
    \sum_{j \in \mathbb{Z}}
    2^{\frac{1}{2}j}
    \n{\Deltav_j f}_{L^2(\mathbb{R}^3)}<\infty
\end{align}
where $\{ \Deltav_j \}_{j \in \mathbb{Z}}$ is the Littlewood--Paley decomposition on the vertical line $\mathbb{R}_{x_3}$.
Later, this result was improved to the general $L^p$ framework by \cites{CZ07,ZF08}.
They considered the new anisotropic Besov norm 
\begin{align}
    \n{f}_{\dot{\mathcal{B}}_p^{s,\sigma}}
    :=
    \sum_{j \in\mathbb{Z}}
    2^{\sigma j}
    \sp{
    \sum_{k \geq j-1}
    2^{2sk}
    \n{\Deltahh_k\Deltav_j f}_{L^p_{\xh}L^2_{x_3}}^2
    }^{\frac{1}{2}}
    +
    \sum_{j \in \mathbb{Z}}
    2^{\sigma j}
    \n{\sum_{k \leq j-2}\Deltahh_k\Deltav_j f}_{L^2}
\end{align}
and proved that for any $2 \leq p \leq 4$ and small solenoidal initial data $a$ that is small in $\dot{\mathcal{B}}_p^{\frac{2}{p}-1,\frac{1}{2}}(\mathbb{R}^3)$, 
\eqref{eq:ANS} possesses a unique global solution.
Here $\{ \Deltahh_k \}_{k \in \mathbb{Z}}$ denotes the Littlewood--Paley decomposition on the horizontal plane $\mathbb{R}^2_{\xh}$.
We notice that their work extended the horizontal functional framework to the general $L^p(\mathbb{R}^2)$ Besov spaces in the frequency region $|\xi_3| \lesssim |\xih|$, while they still used full $L^2$ framework in the region $|\xi_3| \gtrsim |\xih|$ due to the circumstance that we need the $L^2$ energy argument in order to control the $x_3$-derivative loss of $u_3 \partial_{x_3} u_{\rm h}$ in the high frequency part of vertical direction.
Using the relation $\partial_{x_3}u_3 = -\divh u_{\rm h}$, we see that \eqref{eq:ANS} could be regarded as a linear equation for $u_3$. 
Zhang \cites{Z09,Z09-E} focused on this and showed the global well-posedness without smallness on $u_{0,3}$ in the sense that 
\begin{align}\label{small:1}
    \n{u_{0,\rm h}}_{\dot{\mathcal{B}}^{0,\frac{1}{2}}}
    \exp\lp{C\sp{\n{u_{0,3}}_{\dot{\mathcal{B}}^{0,\frac{1}{2}}}+1}^8}
    \ll 1.
\end{align}
By Paicu and Zhang \cite{PZ11}, this condition was improved as
\begin{align}\label{small:2}
    \n{u_{0,\rm h}}_{\dot{\mathcal{B}}^{0,\frac{1}{2}}}
    \exp\lp{C\n{u_{0,3}}_{\dot{\mathcal{B}}^{0,\frac{1}{2}}}^4}
    \ll 1
\end{align}
and they also considered the well-posedness in the framework of \cite{CZ07} under the condition $\n{u_{0,\rm h}}_{\dot{\mathcal{B}}_4^{-\frac{1}{2},\frac{1}{2}}}\exp\lp{C\n{u_{0,3}}_{\dot{\mathcal{B}}_4^{-\frac{1}{2},\frac{1}{2}}}^4}\ll 1$.
The first author \cite{F-pre} recently gave a further improvement of the condition \eqref{small:2} as
\begin{align}\label{small:3}
    \n{u_{0,\rm h}}_{\dot{\mathcal{B}}^{0,\frac{1}{2}}}
    \exp\lp{C\n{u_{0,3}}_{\dot{\mathcal{B}}^{0,\frac{1}{2}}}^2}
    \ll 1.
\end{align}
For the large time behavior of the global solutions, Ji et al. \cite{JWY21} proved that global small classical solutions to \eqref{eq:ANS} behave like the $2$D heat kernel in the $L^2(\R^3)$ framework. This result was improved by Xu and Zhang \cite{XZ22}, where they revealed the enhanced dissipation mechanism, that is, the horizontal components behave like the $2$D heat kernel while the vertical component behaves like the $3$D heat kernel. We refer to \cite{F23} for the extension to $L^p(\R^3)$ with $1\leq p \leq \infty$ and the asymptotic expansion of classical solutions.

On the analyticity of solutions, there are many literature on the isotropic Navier--Stokes equations:
\begin{align}
    \begin{cases}
        \partial_t u - \Delta u + (u \cdot \nabla) u + \nabla P = 0, \qquad & t>0,x \in \R^3, \\
        \div u = 0, & t \geq 0, x \in \R^3, \\
        u(0,x)=u_0(x), & x \in \R^3.
    \end{cases}
\end{align}
Foias and Temam \cite{FT-89} considered the $3$D periodic domain case and showed that the local solution for $H^1$-initial data is spatially analytic in the sense that $e^{tA^{1/2}}A^{1/2}u(t) \in L^2(\mathbb{T}^3)$ for all existence time $t>0$, where $A=-\mathbb{P}\Delta$ is the Stokes operator.
Later, this result was improved by \cites{HS-11,OT-00} and they proved the small global Fujita--Kato solution is analytic in $x$; for small data $u_0 \in \dH^{\frac{1}{2}}(\mathbb{R}^3)$, the global solution $u$ satisfies $e^{\sqrt{t}|\nabla|}u(t) \in C([0,\infty);\dH^{\frac{1}{2}}(\mathbb{R}^3)) \cap L^2(0,\infty;\dH^{\frac{3}{2}}(\mathbb{R}^3))$.
Bae et al. \cite{BBT-12} improved the functional framework to the scaling critical Besov spaces $\dB_{p,q}^{\frac{3}{p}-1}(\mathbb{R}^3)$ ($1 \leq p < \infty$, $1 \leq q \leq \infty$) and Nakasato \cite{N-pre} considered the end-point case $p=\infty$ by using the Fourier--Besov spaces. 
The analyticity of solutions to the anisotropic Navier--Stokes equations \eqref{eq:ANS} is more subtle due to the lack of vertical dissipation. 
Very recently, Liu and Zhang \cite{LZ24} used a new anisotropic Besov space $\fB^{s,\sigma}(\mathbb{R}^3)$ defined via the norm
\begin{align}
    \n{f}_{\fB^{s,\sigma}}
    :=
    \sum_{k,j \in \mathbb{Z}}
    2^{sk}2^{\sigma j}
    \n{\Deltahh_k\Deltav_j f}_{L^2(\mathbb{R}^3)},
\end{align}
and showed that the global solution is analytic in $x \in \mathbb{R}^3$ provided that the initial data is analytic in $x_3$ and small in $\fB^{0,\frac{1}{2}}(\mathbb{R}^3) \cap \fB^{-\alpha,\frac{1}{2}-\beta}(\mathbb{R}^3)$ with $\alpha,\beta \in (0,1)$ satisfying $\alpha+\beta > 1$.
More precisely, they showed that if the initial data satisfies
\begin{align}\label{anal-small-1}
    \n{e^{\rho_0|\partial_{x_3}|}u_0}_{\fB^{0,\frac{1}{2}}}
    +
    \n{e^{\rho_0|\partial_{x_3}|}u_0}_{\fB^{0,\frac{1}{2}}}^{1-\frac{1}{\alpha+\beta}}
    \n{e^{\rho_0|\partial_{x_3}|}u_0}_{\fB^{-\alpha,\frac{1}{2}-\beta}}^{\frac{1}{\alpha+\beta}}
    \ll 1
\end{align}
with some $\rho_0>0$,
then \eqref{eq:ANS} possesses a unique global solution $u$ satisfying\footnote{The tilde spaces are the Chemin--Lerner type spaces associated with the anisotropic Besov spaces. }
\begin{align}
    \n{e^{r\sqrt{t}|\nablah|+\frac{\rho_0}{2}|\partial_{x_3}|}u}_{\widetilde{L^{\infty}}(0,\infty;\fB^{0,\frac{1}{2}}) \cap \widetilde{L^2}(0,\infty;\fB^{1,\frac{1}{2}})}
    \leq
    C
    \n{e^{\rho_0|\partial_{x_3}|}u_0}_{\fB^{0,\frac{1}{2}}},
\end{align}
where $0<r<2\sqrt{1-\alpha}$.
Here, we remark that in contrast to the isotropic case, the assumption \eqref{anal-small-1} of the $x_3$-analyticity for the initial data was needed in their result due to the lack of dissipation in $x_3$.
Finally, we notice that Li et al. \cite{Li-Xu-Zha-pre} reported a paper on arXiv after the preparation of this paper. In there, it was shown that a unique small global solution to \eqref{eq:ANS} on a strip domain constructed and they investigated the Gevrey analyticity of the solution.

The aim of this paper is to improve the result of \cite{LZ24} in the following sense:
\begin{itemize}
    \item 
    We prove the analyticity not only for the space variables $x \in \mathbb{R}^3$ but also for the time variable $t>0$.
    \item 
    We relax the assumptions on the initial data in the sense that we remove {\it the analyticity in $x_3$} and {\it the smallness condition} from $u_{0,3}$ and assume them only for $u_{0,\rm h}$. 
    \item 
    We consider the above aims in the scaling critical $L^p(\mathbb{R}^3)$ framework by using a new anisotropic Besov space $\fB_p^{s,\sigma}(\mathbb{R}^3)$ defined\footnote{See Sec \ref{sec:LP} for the precise definition.} 
    via the norm
    \begin{align}
        \n{f}_{\fB^{s,\sigma}_p}
        :={}&
        \sum_{k,j \in \mathbb{Z}}
        2^{s k}
        2^{\sigma j}
        \n{\Deltahh_k \Deltav_j f}_{L^p(\R^3)}.
    \end{align}
\end{itemize}

\noindent 
{\bf The main result.}
Now, our main result reads as follows.
\begin{thm}\label{thm:anal}
Let $p$ and $\theta$ satisfy
\begin{align}\label{p-theta}
    1\leq p <3,\quad \max\Mp{0,2\sp{1-\f{2}{p}}} <\theta<\min\Mp{1,\f{2}{p}}.
\end{align} 
Then, there exist positive constants $\eta=\eta(p,\theta)$, $C=C(p,\theta)$, and $r=r(p,\theta)$ such that if
the initial datum $u_0=(u_{0,{\rm h}},u_{0,3})$, with $\div u_0=0$ and
\begin{align}
        &
        u_{0,{\rm h}}\in D^{\rm h}_{p,\theta}:= \fB_p^{\frac{2}{p}-1,\frac{1}{p}}(\mathbb{R}^3) 
        \cap 
        \fB_p^{\frac{2}{p}-1+\theta,\frac{1}{p}-\theta}(\mathbb{R}^3) 
        \cap 
        \fB_p^{\frac{2}{p}-2+\theta,\frac{1}{p}-\theta}(\mathbb{R}^3),\\ &
        u_{0,3} \in D^{\rm v}_{p,\theta}:=\fB_p^{\frac{2}{p}-1,\frac{1}{p}}(\mathbb{R}^3)  
        \cap 
        \fB_p^{\frac{2}{p}-2+\theta,\frac{1}{p}-\theta}(\mathbb{R}^3),
\end{align}
satisfies that only 
$u_{0,{\rm h}}$ is small and real analytic in $x_3 \in \mathbb{R}$ in the sense of
\begin{align}\label{anal-small-2}
    \sum_{\beta_3 \in \mathbb{N}\cup \{0\}}
    \frac{\rho_0^{\beta_3}}{\beta_3!}
    \n{\partial_{x_3}^{\beta_3}u_{0,{\rm h}}}_{D^{\rm h}_{p,\theta}}
    \leq
    \eta
    \exp
    \lp{-C\n{u_{0,3}}_{D^{\rm v}_{p,\theta}}}
\end{align}
with some positive constant $\rho_0$, \eqref{eq:ANS} admits a time-space analytic global solution $u=(\uh,u_{3})$ satisfying 
\begin{align}
    &   
    \sum_{
    \substack{\alpha \in \mathbb{N}\cup \{ 0 \} \\ \beta \in (\mathbb{N}\cup \{ 0 \})^3}
    }
    \frac{
    r^{\alpha+|\betah|}(\rho_0/2)^{\beta_3}  
    }
    {
    \alpha!\beta!
    }
    \n{
    t^{\alpha+\frac{|\beta_{\rm h}|}{2}}
    \partial_{t}^{\alpha}\partial_{x}^{\beta}
    u_{\rm h} 
    }_{S^{\rm h}_{p,\theta}} 
    \leq C\eta,  \\
    & 
    \sum_{
    \substack{\alpha \in \mathbb{N}\cup \{ 0 \} \\ \betah \in (\mathbb{N}\cup \{ 0 \})^2}
    }
    \frac{r^{\alpha+|\betah|}}{\alpha!\betah!}
    \n{
    t^{\alpha+\frac{|\beta_{\rm h}|}{2}}\partial_{t}^{\alpha}\partial_{\xh}^{\betah}
    u_3
    }_{S^{\rm v}_{p,\theta}}  
    \leq
    C
    \n{u_{0,3}}_{D^{\rm v}_{p,\theta}}+C\eta^2,
\end{align}
where the spaces $S^{\rm h}_{p,\theta}$ and $S^{\rm v}_{p,\theta}$ are defined by
\begin{align}
        S^{\rm h}_{p,\theta}
        :={}&
        \Mp{
        u_{\rm h} \in L^{\infty}\sp{0,\infty;D^{\rm h}_{p,\theta}}
        \ ;\ 
        \nablah^2 u_{\rm h} \in L^1\sp{0,\infty;D^{\rm h}_{p,\theta}}
        },
        \\
        S^{\rm v}_{p,\theta}
        :={}&
        \Mp{
        u_3 \in L^{\infty}\sp{0,\infty;D^{\rm v}_{p,\theta}}
        \ ;\ 
        \nablah^2 u_3 \in L^1\sp{0,\infty;D^{\rm v}_{p,\theta}}
        }.
\end{align}
Moreover, this analytic solution is unique in the case of $1 \leq p \leq 2$. 
\end{thm}
\begin{rem}
    Let us give some comments on Theorem \ref{thm:anal}.
    \begin{enumerate}
        \item 
        We first mention on our functional framework.
        The function spaces $\fB_p^{\frac{2}{p}-1,\frac{1}{p}}(\mathbb{R}^3) 
        $ and $
        \fB_p^{\frac{2}{p}-1+\theta,\frac{1}{p}-\theta}(\mathbb{R}^3) 
        $
        are the scaling critical, whereas the space
        $
        \fB_p^{\frac{2}{p}-2+\theta,\frac{1}{p}-\theta}(\mathbb{R}^3)$
        is a supercritical auxiliary space.
        For this auxiliary space,
        putting $\alpha=1-\theta$ and $\beta=\theta$, we see that our setting is the limiting case $\alpha + \beta = 1$ in the initial condition \eqref{anal-small-1} proposed by \cite{LZ24}, which implies we relax the additional low frequency regularity condition on the initial data.
        \item 
        Let us compare our framework with the well-posedness on the $L^p$ framework by \cites{CZ07,ZF08}.
        As stated above, their results need $L^2(\mathbb{R}_{x_3})$ in the vertical direction and also assumed $L^2(\mathbb{R}^2_{\xh})$ when the horizontal frequencies are higher than the vertical frequency. These $L^2$ restrictions were required for controlling the derivative loss of $\partial_{x_3}$ in the nonlinear term.
        Meanwhile, our result is in the framework of anisotropic Besov spaces based on $L^p(\mathbb{R}^3)$ framework. We emphasize that our paper is the first paper to achieve $L^p$ for all variables in the mathematical history of the anisotropic Navier--Stokes equations.
        The reason why we are able to treat in the full $L^p(\mathbb{R}^3)$ framework is that we do not need to handle the derivative loss appearing in $u_3 \partial_{x_3} u_{\rm h}$ since we consider the analytic framework and differentiate the solution infinitely many times.
        \item 
        In comparison with the initial condition of \cite{LZ24}, that is \eqref{anal-small-1}, our result does not assumed the analyticity of $u_{0,3}$ for $x_3$-direction. Moreover, our condition \eqref{anal-small-2} does not require the smallness of $u_{0,3}$. This condition follows the genealogy of \cites{Z09,Z09-E,PZ11,F-pre} and provides one of the best form in the sense that the power $k$ of the norm for $u_{0,3}$ is $1$ whereas \cites{Z09,Z09-E}, \cite{PZ11}, and \cite{F-pre} treated $k=8$, $k=4$, and $k=2$, respectively; see \eqref{small:1}, \eqref{small:2}, and \eqref{small:3}.
        \item 
        We may obtain the analyticity of $u_3$ in $x_3$-direction immediately from Theorem \ref{thm:anal}. 
        Indeed, using $\partial_{x_3}u_3 = - \divh u_{\rm h}$, we have 
        \begin{align}
            &\sum_{
            \substack{\alpha \in \mathbb{N}\cup \{ 0 \} \\ \beta \in (\mathbb{N}\cup \{ 0 \})^3}
            }
            \frac{
            r^{\alpha+|\betah|}(\rho_0/2)^{\beta_3} 
            }
            {
            \alpha!\beta!
            }
            \n{t^{\alpha+\frac{|\beta_{\rm h}|}
            {2}}
            \partial_{t}^{\alpha}\partial_x^{\beta}
            u_3}_{\widetilde{S}^{\rm h}_{p,\theta}}  
            \\
            &\quad =
            \sum_{
            \substack{\alpha \in \mathbb{N}\cup \{ 0 \} \\ \beta \in (\mathbb{N}\cup \{ 0 \})^3}
            }
            \frac{
            r^{\alpha+|\betah|}(\rho_0/2)^{\beta_3} 
            }
            {
            \alpha!\beta!
            }
            \n{
            t^{\alpha+\frac{|\beta_{\rm h}|}{2}}
          \partial_{t}^{\alpha}     
            \partial_x^{\beta}
            \frac{\nablah}{|\nablah|} \cdot u_{\rm h}
            }_{S^{\rm h}_{p,\theta}}    \\
            &\quad 
            \leq
            C
            \sum_{
            \substack{\alpha \in \mathbb{N}\cup \{ 0 \} \\ \beta \in (\mathbb{N}\cup \{ 0 \})^3}
            }
            \frac{
            r^{\alpha+|\betah|}(\rho_0/2)^{\beta_3}  
            }
            {
            \alpha!\beta!
            }
            \n{
            t^{\alpha+\frac{|\beta_{\rm h}|}{2}}
            \partial_{t}^{\alpha}
            \partial_{x}^{\beta} u_{\rm h}
            }_{S^{\rm h}_{p,\theta}}   \leq C \eta,
        \end{align}
        where we have defined the semi-norm $\n{z}_{\widetilde{S}_{p,\theta}^{\rm h}}:=\n{|\nablah|^{-1}\partial_{x_3}z}_{S_{p,\theta}^{\rm h}}$.
    \item
    As is well-known for the research of analytic solutions to the Euler equations, the $x_3$-hyperbolic nature of the nonlinear term $u_3 \partial_{x_3}u_{\rm h}$ affects the vertical radius of the convergence and it shrinks for $t>0$. 
    Indeed, our radius of convergence for vertical direction is given by $\rho_0/2$. 
    The behavior of the loosing radius of the vertical analyticity for $t>0$ may be described more precisely; from our method to obtain the global analytic a priori estimates, we see that the vertical radius of the convergence is given by $\rho_0e^{-\lambda f(t)}$, where $\lambda>0$ is an absolute constant and   
    \begin{align}\label{eq:f}
    \begin{split}
    f(t) 
    ={}&
    \sum_{
    \substack{\alpha \in \mathbb{N}\cup \{ 0 \} \\ \betah \in (\mathbb{N}\cup \{ 0 \})^2}
    }
    \frac{r^{\alpha+|\betah|}}{\alpha!\betah!}
    \int_0^t \n{t^{\alpha+\frac{|\betah|}{2}}\partial_{t}^{\alpha}\partial_{\xh}^{\betah}u_3(\tau)}_{\fB_p^{\frac{2}{p},\frac{1}{p}}} d\tau 
    \\
    &
    +
    \sum_{
    \substack{\alpha \in \mathbb{N}\cup \{ 0 \} \\ \beta \in (\mathbb{N}\cup \{ 0 \})^3}
    }
    \frac{r^{\alpha+|\betah|}\rho_0^{\beta_3}}{\alpha!\beta!}
    \int_0^t
    \n{t^{\alpha+\frac{|\betah|}{2}}\partial_{t}^{\alpha}\partial_{x}^{\beta}u_{\rm h}(\tau)}_{\fB_p^{\frac{2}{p}+1,\frac{1}{p}}}
    e^{-   \lambda  \beta_3 f(\tau)} 
    d\tau
    \end{split}
    \end{align}
    with $f(0)=0$, if \eqref{eq:f} admits a solution.
    \end{enumerate}
\end{rem}

\noindent
{\bf Structure of this paper and the strategy of the proof of main result.}
This paper is arranged as follows. In Section \ref{sec:LP}, we recall some basic facts from anisotropic Littlewood--Paley theory. Section \ref{sec:glo-est} is devoted to the global a priori estimates, which is the key of our analysis. Finally, we present the proof of our main result in Section \ref{sec:pf}.

Next, let us briefly describe the strategy of the proof of Theorem \ref{thm:anal}, especially focusing on the difficulty appearing in the steps.
In the previous study, 
Liu--Zhang \cite{LZ24} 
introduce the analytic operator $e^{\sqrt{t}|\nablah|+\frac{\rho_0}{2}|\partial_{x_3}|}$ and show the analyticity of the small global solutions to \eqref{eq:ANS} as mentioned above.
In contrast to this, we directly calculate the higher derivative of the solutions, since we need more detailed analysis of the difference between the anisotropic properties of $u_{\rm h}$ and $u_3$.
Therefore, our central problem is how to converge the power series of the solution while facing the problem of derivative loss in the $x_3$-direction. 
To proceed, we encounter three main difficulties, which we explain step by step.
For the sake of simple presentation, 
we consider the case of $\alpha=0$ and $\betah=(\beta_1,\beta_2)=(0,0)$ and also set $\rho_0=1$.
Differentiating the equation of $\uh$ by $x_3$ for $\beta_3$-times, we see that
\begin{align}\label{x_3-eq}
  \p_t \frac{\p_{x_3}^{\beta_3} \uh}{\beta_3!} 
  &
  - \Deltah \frac{\p_{x_3}^{\beta_3} \uh}{\beta_3!}
  + \frac{1}{\beta_3!}
  \p_{x_3} ^{\beta_3} \sp{  \sp{\uh\cdot \nablah} \uh }
  + 
  \frac{1}{\beta_3!}
  \p_{x_3} ^{\beta_3} \sp{  u_3\partial_{x_3} \uh }
  +
  \nablah \frac{\p_{x_3} ^{\beta_3} p}{\beta_3!}
  =0. 
\end{align}
From the maximal regularity of the horizontal heat equations, we have 
\begin{align}
    &
    \frac{1}{\beta_3!}
    \n{\p_{x_3} ^{\beta_3} \uh}_{{L^{\infty}_T}(\fB_p^{\frac{2}{p}-1,\frac{1}{p}}) \cap L^1_T(\fB_p^{\frac{2}{p}+1,\frac{1}{p}})}
    \\
    &\quad 
    \leq
    \frac{C}{\beta_3!}
    \n{\partial_{x_3}^{\beta_3}u_0}_{\fB_p^{\frac{2}{p}-1,\frac{1}{p}}}
    +
    \frac{C}{\beta_3!}
    \n{\partial_{x_3}^{\beta_3}\sp{  \sp{\uh\cdot \nablah} \uh }}
    _{ L^1_{T} ( \fB_p^{\frac{2}{p}-1,\frac{1}{p}} )
    }\label{intro:max-2}
    \\
    &\qquad
    +
    \frac{C}{\beta_3!}
    \n{\partial_{x_3}^{\beta_3}\sp{  u_3\partial_{x_3} \uh }}_{L^1_T(\fB_p^{\frac{2}{p}-1,\frac{1}{p}})}
    +
    \frac{C}{\beta_3!}
    \n{\nablah \partial_{x_3}^{\beta_3} p}_{L^1_T(\fB_p^{\frac{2}{p}-1,\frac{1}{p}})}
\end{align}
Then, the second nonlinear term is written as 
\begin{align}\label{dif:1}
    \begin{aligned}
    \frac{1}{\beta_3!}
    \partial_{x_3}^{\beta_3}
    (u_3 \partial_{x_3}^{\beta_3}\uh)
    ={}&
    \beta_3
    u_3 \frac{\partial_{x_3}^{\beta_3+1}u_{\rm h}}{(\beta_3+1)!}\\
    &
    -
    \sum_{\beta_3=\beta_3'+\beta_3'', \beta_3' \geq 1}
    \frac{\beta_3''+1}{\beta_3'}
    \frac{\divh \partial_{x_3}^{\beta_3'-1}u_{\rm h}}{(\beta_3'-1)!}
    \frac{\partial_{x_3}^{\beta_3''+1}u_{\rm h}}{(\beta_3''+1)!}.
    \end{aligned}
\end{align}
All coefficients on the right-hand side are of order $O(\beta_3)$.
To control this $\beta_3$, we consider following ordinary differential equation\footnote{This is a simpler version that only considers the vertical analyticity. Considering the analyticity in $t$ and $x$, we use more general and complicated one; see \eqref{ODE}.}:
\begin{align}\label{eq:simple-f}
    f'(t)= \n{u_3(t)}_{\fB_p^{\frac{2}{p},\frac{1}{p}}} + \sum_{\beta_3 \in \mathbb{N} \cup \{ 0 \}} \frac{1}{\beta_3!} \n{\partial_{x_3}^{\beta_3}u_{\rm h}(t)}_{\fB_p^{\frac{2}{p}+1,\frac{1}{p}}}e^{-\lambda\beta_3f(t)}, \quad t>0
\end{align}
with the initial condition $f(0)=0$.
We may overcome the trouble by multiplying the equation by $e^{-\lambda \beta_3 f(t)}$.
Then, we modify the maximal regularity \eqref{intro:max-1} by the estimate for $\partial_{x_3}^{\beta_3}\uh(t)e^{-\lambda \beta_3f(t)}/\beta_3!$ as 
\begin{align}
    &
    \frac{1}{\beta_3!}
    \n{\p_{x_3} ^{\beta_3} \uh e^{-\lambda\beta_3f(t)}}_{{L^{\infty}_T}(\fB_p^{\frac{2}{p}-1,\frac{1}{p}}) \cap L^1_T(\fB_p^{\frac{2}{p}+1,\frac{1}{p}})}\\
    &\qquad
    +
    \frac{1}{\beta_3!}
    \int_0^T
    \n{
    \partial_{x_3}^{\beta_3}u_{\rm h}(t)
    e^{-\lambda\beta_3f(t)}
    }_{\fB_p^{\frac{2}{p}-1,\frac{1}{p}}}
    \lambda \beta_3 f'(t)
    dt\\
    &\quad 
    \leq
    \frac{C}{\beta_3!}
    \n{\partial_{x_3}^{\beta_3}u_0}_{
    \fB_p^{\frac{2}{p}-1,\frac{1}{p}}
    }
    +
    \frac{C}{\beta_3!}
    \n{\partial_{x_3}^{\beta_3}\sp{  \sp{\uh\cdot \nablah} \uh }e^{-\lambda\beta_3f(t)}}_{ 
    L^1_T( \fB_p^{\frac{2}{p}-1,\frac{1}{p}} )
    }
    \label{intro:max-1}
    \\
    &\qquad
    +
    \frac{C}{\beta_3!}
    \n{\partial_{x_3}^{\beta_3}\sp{  u_3\partial_{x_3} \uh }e^{-\lambda\beta_3f(t)}}_{L^1_T(\fB_p^{\frac{2}{p}-1,\frac{1}{p}})}\\
    &\qquad
    +
    \frac{C}{\beta_3!}
    \n{\nablah \partial_{x_3}^{\beta_3} pe^{-\lambda\beta_3f(t)}}_{L^1_T(\fB_p^{\frac{2}{p}-1,\frac{1}{p}})}.
\end{align}
Then, for the problematic nonlinear estimate, using product estimate 
\begin{align}
    &
    \n{u_3\partial_{x_3}^{\beta_3+1}u_{\rm h}}_{\fB_p^{\frac{2}{p}-1,\frac{1}{p}}}
    \leq
    C
    \n{u_3}_{\fB_p^{\frac{2}{p},\frac{1}{p}}}
    \n{\partial_{x_3}^{\beta_3+1}u_{\rm h}}_{\fB_p^{\frac{2}{p}-1,\frac{1}{p}}},\\
    &
    \n{\divh \partial_{x_3}^{\beta_3'-1}u_{\rm h}
    \partial_{x_3}^{\beta_3''+1}u_{\rm h}}_{\fB_p^{\frac{2}{p}-1,\frac{1}{p}}}
    \leq
    C
    \n{\partial_{x_3}^{\beta_3'-1}u_{\rm h}}_{\fB_p^{\frac{2}{p}+1,\frac{1}{p}}}
    \n{\partial_{x_3}^{\beta_3''+1}u_{\rm h}}_{\fB_p^{\frac{2}{p}-1,\frac{1}{p}}},
\end{align}
we see that 
\begin{align}
    &\sum_{\beta_3 \in \mathbb{N} \cup \{ 0 \}}
    \frac{1}{\beta_3!}
    \n{\partial_{x_3}^{\beta_3}
    (u_3 \partial_{x_3}^{\beta_3}\uh)e^{-\lambda \beta_3f(t)}}_{L^1(0,T;\fB_p^{\frac{2}{p}-1})}
    \\
    &
    \quad
    \leq
    \frac{C}{\lambda}
    \sum_{\beta_3 \in \mathbb{N} \cup \{ 0 \}}
    \frac{1}{\beta_3!}
    \int_0^T
    \n{\partial_{x_3}^{\beta_3}u_{\rm h}(t)}_{\fB_p^{\frac{2}{p}-1,\frac{1}{p}}}
    e^{-\lambda\beta_3f(t)}
    \lambda \beta_3 f'(t)
    dt. 
\end{align}
The right-hand side above may be absorbed in the $\beta_3$-summation of \eqref{intro:max-1}, by choosing $\lambda$ sufficiently large.
We remark that the first term of the right hand side of \eqref{eq:simple-f} is $-1$th super-critical derivative exponent for horizontal direction.
This term is necessary for the estimate of the first term of the right hand side of \eqref{dif:1} as this nonlinear term does not have the horizontal derivative.\footnote{The other nonlinear terms possesses the first-order horizontal derivatives.}
When controlling this term in $L^1(0,T)$, we meet the super-critical regularity $\fB_p^{2/p-2+\theta,1/p-\theta}(\mathbb{R}^3)$.

The next difficulty appears in estimating the pressure term. 
We remark that the previous study \cite{LZ24} did not need to treat it by using the elimination via the divergence free condition whereas we have to estimate the pressure precisely since our result requires separate calculations for the horizontal and vertical directions of the velocity field.
We may write the horizontal pressure as 
\begin{align}
    \nablah P
    &=
    \sum_{\ell,m=1}^2
    \nablah \partial_{x_{\ell}}\partial_{x_m}(-\Delta)^{-1}
    (u_{\ell}u_m)
    +
    2
    \nablah \partial_{x_3}(-\Delta)^{-1}
    (u_{\rm h}\cdot \nablah u_3)\\
    &
    =:
    \nablah P_1 + \nablah P_2
\end{align}
and $P_2$ is the problematic one.
Indeed, considering the case of $(\alpha,\betah)=(0,0)$ and $\rho_0=1$, we have 
\begin{align}
    \frac{1}{\beta_3!}
    \nablah \partial_{x_3}^{\beta_3}P_2
    ={}&
    \frac{2}{\beta_3!} 
    \nablah\partial_{x_3}(-\Delta)^{-1} 
    \partial_{x_3}^{\beta_3 } u_{\rm h} \cdot  \nablah u_3
    \\
    &
    -
    \sum_{\beta_3'+\beta_3''=\beta_3, \beta_3'' \geq 1  } 
    \frac{2}{\beta_3'!\beta_3''!}
    \nablah \divh (-\Delta)^{-1}
    \partial_{x_3}
    \sp{\partial_{x_3}^{\beta_3'}u_{\rm h}
    \divh \partial_{x_3}^{\beta_3''-1}u_{\rm h}}
    \\
    &
    +
    \sum_{\beta_3'+\beta_3''=\beta_3, \beta_3'' \geq 1  } 
    \frac{2}{\beta_3'!\beta_3''!}
    \nablah\partial_{x_3}(-\Delta)^{-1}
    \sp{\divh \partial_{x_3}^{\beta_3'} u_{\rm h} 
    \divh \partial_{x_3}^{\beta_3''-1}u_{\rm h}}.
\end{align}
Here the last line above is the most harmful. 
To estimate this problematic one by the norm of the critical space $\fB_p^{\frac{2}{p}-1,\frac{1}{p}}(\mathbb{R}^3)$, it is natural to meet 
\begin{align}
    \n{\divh \partial_{x_3}^{\beta_3'} u_{\rm h} 
    \divh \partial_{x_3}^{\beta_3''-1}u_{\rm h}}_{\fB_p^{\frac{2}{p}-1,\frac{1}{p}}}.
\end{align}
However, it seems impossible to bound this by suitable norms since the order of the horizontal derivative in it is in excess of one.
We overcome this by adjusting the order of derivatives using the structure of the Riesz transform $\nablah\partial_{x_3}(-\Delta)^{-1}$. We notice that this step requires us the auxiliary regularities $\fB_p^{2/p-1+\theta,1/p-\theta}(\mathbb{R}^3) \cap \fB_p^{2/p-2+\theta,1/p-\theta}(\mathbb{R}^3)$.

The final obstacle comes from the polynomial weight of $t$ for the case of $(\alpha,\betah) \neq (0,0)$. 
In this paragraph, we only consider the case of $\beta_3=0$ for the sake of simplicity.
A direct calculation gives
\begin{align}
    &
    \p_t \frac{r^{\alpha+|\betah|} t^{\alpha+\f{|\betah|}{2}} }{\alpha!\betah!}
    \partial_{t,\xh}^{\alpha,\betah}
    u
    - 
    \Deltah  \frac{r^{\alpha+|\betah|} t^{\alpha+\f{|\betah|}{2}} }{\alpha!\betah!}
    \partial_{t,\xh}^{\alpha,\betah}
    u
    +
    \frac{r^{\alpha+|\betah|} t^{\alpha+\f{|\betah|}{2}} }{\alpha!\betah!}
    \partial_{t,\xh}^{\alpha,\betah}
    [(u\cdot \nabla)u]
    \\
    &\quad
    +
    \frac{r^{\alpha+|\betah|} t^{\alpha+\f{|\betah|}{2}} }{\alpha!\betah!}
    \partial_{t,\xh}^{\alpha,\betah}
    \nabla P
    =
    \sp{ \alpha+\f{|\betah|}{2} }
    \frac{r^{\alpha+|\betah|} t^{\alpha+\f{|\betah|}{2}-1} }{\alpha!\betah!}
    \partial_{t,\xh}^{\alpha,\betah}
    u. \label{diff-wei-1}
\end{align}
We focus on the new term on the right-hand side of \eqref{diff-wei-1}.
This new term can be controlled by following the method provided by Iwabuchi \cite{Iwa-20} for the analyticity of the critical Burgers equation.
However, it seems difficult to treat the case of $\alpha=0$ and $|\betah|=1$ since the time weight $t^{\alpha+|\betah|/2-1}=t^{-1/2}$ in the right handside of \eqref{diff-wei-1} has a singularity at $t=0$.
To circumvent this, we do not use \eqref{diff-wei-1} but consider the Duhamel form 
\begin{align}
    rt^{\frac{1}{2}}
    \partial_{\xh}^{\betah}
    u(t)
    =
    rt^{\frac{1}{2}}
    \partial_{\xh}^{\betah}
    e^{t \Deltah}
    u_0
    -
    rt^{\frac{1}{2}}
    \partial_{\xh}^{\betah}
    \int_0^t 
    e^{(t-\tau)\Deltah}
    \lp{(u\cdot \nabla)u+\nabla P}(\tau)
    d\tau
\end{align}
and apply the standard argument for the decay estimates, that is, we decompose the time integral of the Duhamel term as $\int_0^t = \int_0^{t/2} +\int_{t/2}^t$ and estimate each integral by appropriate calculations.

Now, we overcome three difficulties as mentioned above and obtain the global analytic a priori estimates of the solutions to \eqref{eq:ANS}.
To complete the proof, we make use of the a priori estimate to provide the uniform boundedness of the frequency cut-off approximation system for \eqref{eq:ANS} and finally obtain the global analytic solutions by the compactness argument.

\section{Anisotropic Littlewood--Paley theory}\label{sec:LP}
In this section, we establish the anisotropic Littlewood--Paley theory. To this end, we first define the Littlewood--Paley decomposition. 
Let $\phi \in C_c^{\infty}([0,\infty))$ satisfy $0 \leq \phi(r) \leq1$ on $[0,\infty)$, $\supp \phi \subset \lp{2^{-1},2}$, and  
\begin{align}\label{phi}
    \sum_{j \in \mathbb{Z}} \phi(2^{-j}r) = 1
\end{align}
for every fixed $r>0$.
Let $\phih_k(\xih):=\phi(2^{-k}|\xih|)$ and $\phiv_j(\xi_3):=\phi(2^{-j}|\xi_3|)$ for $\xih \in \mathbb{R}^2$, $\xi_3 \in \mathbb{R}$, and $k,j \in \mathbb{Z}$.
We define the anisotropic homogeneous Littlewood--Paley frequency localization operators as 
\begin{align}
    \Deltahh_k f := \mathscr{F}^{-1}_{\mathbb{R}^2}\lp{\phih_k(\xih)\mathscr{F}_{\mathbb{R}^2}[f](\xih)},\quad
    \Deltav_j f := \mathscr{F}^{-1}_{\mathbb{R}}\lp{\phiv_j(\xi_3)\mathscr{F}_{\mathbb{R}}[f](\xi_3)}
\end{align}
for all $k,j \in \mathbb{Z}$.
We also define the anisotropic inhomogeneous Littlewood--Paley decomposition as 
\begin{align}
    &\iDh_k f
    :=
    \begin{cases}
        \Deltahh_k f, & (k \geq 1), 
        \\
        \mathscr{F}^{-1}_{\mathbb{R}^2}\lp{\widetilde{\phi}(|\xih|)\mathscr{F}_{\mathbb{R}^2}[f](\xih)}, & (k=0),
    \end{cases}
    \\
    &\iDv_j f
    :=
    \begin{cases}
        \Deltav_j f, & (j \geq 1), 
        \\
        \mathscr{F}^{-1}_{\mathbb{R}}\lp{\widetilde{\phi}(|\xi_3|)\mathscr{F}_{\mathbb{R}}[f](\xih)}, & (j=0)
    \end{cases}
\end{align}
for all $k,j \in \mathbb{N}\cup \{0\}$, where we have set 
\begin{align}
    \widetilde{\phi}(r):= 1 - \sum_{j \geq 1}\phi(2^{-j}r).
\end{align}

Now we introduce the definition of anisotropic homogeneous and inhomogeneous Besov spaces. Observe that the anisotropic homogeneous Besov spaces was also defined in Chemin et al. \cite{CPZ14}.  
\begin{df}\label{def:Besov}
    Let $s,\sigma \in \mathbb{R}$ and $1\leq p \leq \infty$.
    \begin{enumerate}
    \item 
    The anisotropic homogeneous Besov space $\fB^{s,\sigma}_p(\mathbb{R}^3)$ as  
    \begin{align}
        \fB_p^{s,\sigma}(\mathbb{R}^3)
        :={}&
        \Mp{
        f \in \mathscr{S}'(\mathbb{R}^3)/\mathscr{P}(\mathbb{R}^3)\ ;\ \n{f}_{\fB_p^{s,\sigma}}<\infty},
        \\
        \n{f}_{\fB^{s,\sigma}_p}
        :={}&
        \sum_{k,j \in \mathbb{Z}}
        2^{s k}
        2^{\sigma j}
        \n{\Deltahh_k \Deltav_j f}_{L^p(\R^3)}.
    \end{align}
    Here, $\mathscr{P}(\mathbb{R}^3)$ is the set of all polynomials on $\mathbb{R}^3$ that is given by the product of polynomials on $\mathbb{R}^2_{\xh}$ and polynomials on $\mathbb{R}_{x_3}$.
    \item 
    The anisotropic inhomogeneous Besov space $\ifB^{s,\sigma}_p(\mathbb{R}^3)$ as  
    \begin{align}
        \ifB_p^{s,\sigma}(\mathbb{R}^3)
        :={}&
        \Mp{
        f \in \mathscr{S}'(\mathbb{R}^3)\ ;\ \n{f}_{\ifB_p^{s,\sigma}}<\infty},
        \\
        \n{f}_{\ifB_p^{s,\sigma}}
        :={}&
        \sum_{k,j \in \mathbb{N}\cup \{ 0\}}
        2^{s k}
        2^{\sigma j}
        \n{\iDh_k \iDv_j f}_{L^p(\R^3)}.
    \end{align}
    \end{enumerate}
\end{df}


For any $s,\sigma \in \R,1\leq p,r \leq \infty$ and $f \in L^1(I;[0,\infty])$, 
we define the weighted norm as 
\begin{align}
    \| F\|_{L^1_f(I;\fB^{s,\sigma}_{p})} 
    :={}& 
    \n{\n{F(t)}_{\fB_p^{s,\sigma}}}_{L^1(I;f(t)dt)}
    \\
    ={}&
    \sum_{k,j \in \mathbb{Z}}
    2^{s j}
    2^{\sigma k} 
    \int_{I} 
    \n{\Deltahh_j \Deltav_k F(t)}_{L^p(\R^3)} f(t)dt
\end{align}
for all $F \in {L^{\infty}}(I;\fB_p^{s,\sigma}(\mathbb{R}^3))$.

Now we prepare some useful lemmas. 
We begin with the product estimates in the homogeneous anisotropic Besov spaces.
\begin{lemm}\label{lemm:prod}
    Let $1 \leq p < \infty$. 
    Let $s_0,s_1,s_2,s_3,s_4,\sigma,\sigma_1,\sigma_2 \in \mathbb{R}$ satisfy
    \begin{gather}
        s_0=s_1+s_2=s_3+s_4 > \max \Mp{0, 2\sp{\frac{2}{p}-1}},
        \quad
        s_1,s_2,s_3,s_4 \leq \frac{2}{p}, \\
        \sigma > - \min \Mp{\frac{1}{p},1-\frac{1}{p}},
        \quad
        \sigma_1,\sigma_2 \geq 0.
    \end{gather}
    Then, there exists a positive constant $C=C(s_0,s_1,s_2,s_3,s_4,\sigma,\sigma_1,\sigma_2,p)$ such that 
    \begin{align}
        \n{fg}_{\fB_p^{s_0-\frac{2}{p},\sigma}}
        \leq
        C
        \n{f}_{\fB_p^{s_1,\frac{1}{p}-\sigma_1}}
        \n{g}_{\fB_p^{s_2,\sigma+\sigma_1}}
        +
        C
        \n{f}_{\fB_p^{s_3,\sigma+\sigma_2}}
        \n{g}_{\fB_p^{s_4,\frac{1}{p}-\sigma_2}}
    \end{align}
    for all $f \in \fB_p^{s_1,\frac{1}{p}-\sigma_1}(\mathbb{R}^3) \cap \fB_p^{s_3,\sigma+\sigma_2}(\mathbb{R}^3)$ and $g \in \fB_p^{s_2,\sigma+\sigma_1}(\mathbb{R}^3) \cap \fB_p^{s_4,\frac{1}{p}-\sigma_2}(\mathbb{R}^3)$.
\end{lemm}
\begin{rem}
    Lemma \ref{lemm:prod} implies\footnote{To see this, we choose $s=s_0-2/p$, $s_1=s_3=s$, $s_2=s_4=2/p$, $\sigma_1=1/p-\sigma$, and $\sigma_2=0$.}
    that for $1 \leq p < \infty$, $s$ and $\sigma$ satisfying
    \begin{align}
        -\min \Mp{\frac{2}{p},2\sp{1-\frac{1}{p}}} < s \leq \frac{2}{p}, 
        \quad 
        -\min \Mp{\frac{1}{p},1-\frac{1}{p}} < \sigma \leq \frac{1}{p}, 
    \end{align}
    there exists a positive constant $C=C(s,\sigma,p)$ such that 
    \begin{align}
        \n{fg}_{\fB_p^{s,\sigma}}
        \leq
        C
        \n{f}_{\fB_p^{s,\sigma}}
        \n{g}_{\fB_p^{\frac{2}{p},\frac{1}{p}}}
    \end{align}
    for all $f \in \fB_p^{s,\sigma}(\mathbb{R}^3)$ and $g \in \fB_p^{\frac{2}{p},\frac{1}{p}}(\mathbb{R}^3)$.
\end{rem}
\begin{proof}[Proof of Lemma \ref{lemm:prod}]
    It is well-known, see for instance \cite{Che-Mia-Zha-10}*{Lemma 2.8}, that for any $1 \leq p < \infty$, $1 \leq r,r_1,r_2 \leq \infty$, and $s_0,s_1,s_2 \in \mathbb{R}$ with
    \begin{align}
        \frac{1}{r} = \frac{1}{r_1} + \frac{1}{r_2}, \quad
        s_0=s_1+s_2 > \max \Mp{0,2\sp{\frac{2}{p}-1}}, \quad s_1,s_2\leq \frac{2}{p},
    \end{align}
    there holds
    \begin{align}\label{2D-prod}
        \n{\phi\psi}_{\widetilde{L^r}(\mathbb{R};\dB_{p,1}^{s_0-\frac{2}{p}}(\mathbb{R}^2))} 
        \leq 
        C 
        \n{\phi}_{\widetilde{L^{r_1}}(\mathbb{R};\dB_{p,1}^{s_1}(\mathbb{R}^2))} 
        \n{\psi}_{\widetilde{L^{r_2}}(\mathbb{R};\dB_{p,1}^{s_2}(\mathbb{R}^2))}
    \end{align}
    for all $\phi \in \widetilde{L^{r_1}}(\mathbb{R};\dB_{p,1}^{s_1}(\mathbb{R}^2))$ and $\psi \in \widetilde{L^{r_2}}(\mathbb{R};\dB_{p,1}^{s_2}(\mathbb{R}^2))$.
    Here, $\widetilde{L^r}(\mathbb{R};\dB_{p,1}^s(\mathbb{R}^2))$ ($1\leq p ,r \leq \infty$ and $s \in \mathbb{R}$) is the Chemin--Lerner space defined as the set of all functions $f \in L^r(\mathbb{R};\dB_{p,1}^s(\mathbb{R}^2))$ satisfying
    \begin{align}
        \n{f}_{\widetilde{L^r}(\mathbb{R};\dB_{p,1}^s(\mathbb{R}^2))}
        :=
        \sum_{k \in \mathbb{Z}}
        2^{sk}
        \n{\Deltahh_k f}_{L^r(\mathbb{R};L^p(\mathbb{R}^2))} < \infty.
    \end{align}
    By the Bony decomposition, we split $fg$ as 
    \begin{align}
        fg = T^{\rm v}_fg + R^{\rm v}(f,g) + T^{\rm v}_gf,
    \end{align}
    where 
    \begin{align}
        T^{\rm v}_fg
        :=
        \sum_{j \in\mathbb{Z}} \sum_{\ell \leq j-3} \Deltav_\ell f \Deltav_jg, 
        \quad
        R^{\rm v}(f,g)
        :=
        \sum_{j \in\mathbb{Z}} \sum_{|\ell -j|\leq 2} \Deltav_\ell f \Deltav_jg.
    \end{align}
    We then see that 
    \begin{align}
        \Deltav_mT^{\rm v}_fg
        =
        \Deltav_m
        \sum_{|j-m|\leq2} \sum_{\ell \leq j-3} \Deltav_\ell f \Deltav_jg, 
        \quad
        \Deltav_mR^{\rm v}(f,g)
        =
        \Deltav_m
        \sum_{j \geq m-2} \sum_{|\ell -j|\leq 2} \Deltav_\ell f \Deltav_jg.
    \end{align}
    It follows from \eqref{2D-prod} that 
    \begin{align}
        &\n{T^{\rm v}_fg}_{\fB_p^{s_0-\frac{2}{p},\sigma}}
        ={}
        \sum_{k,m \in\mathbb{Z}}
        2^{ (s_0-\f{2}{p} ) k }
        2^{\sigma m}
        \n{ \Deltahh_k  \Deltav_m T^{\rm v}_fg}_{L^p(\mathbb{R}^3)}\\
        &\quad={} 
        \sum_{m \in\mathbb{Z}}
        2^{\sigma m}
        \n{\Deltav_m T^{\rm v}_fg}_{\widetilde{L^p}(\mathbb{R};\dB_{p,1}^{s_0-\frac{2}{p}}(\mathbb{R}^2))}\\
        &\quad\leq{}
        C
        \sum_{m \in\mathbb{Z}}
        2^{\sigma m}
        \sum_{|j-m|\leq2} \sum_{\ell \leq j-3} \n{\Deltav_\ell f \Deltav_jg}_{\widetilde{L^p}(\mathbb{R};\dB_{p,1}^{s_0-\frac{2}{p}}(\mathbb{R}^2))}\\
        &\quad\leq{}
        C
        \sum_{m \in\mathbb{Z}}
        2^{\sigma m}
        \sum_{|j-m|\leq2} \sum_{\ell \leq j-3} \n{\Deltav_\ell f}_{\widetilde{L^{\infty}}(\mathbb{R};\dB_{p,1}^{s_1}(\mathbb{R}^2))} \n{\Deltav_jg}_{\widetilde{L^p}(\mathbb{R};\dB_{p,1}^{s_2}(\mathbb{R}^2))}\\
        &\quad\leq{}
        C
        \sum_{m \in\mathbb{Z}}
        2^{\sigma m}
        \sum_{|j-m|\leq2} \sum_{\ell \leq j-3} 2^{\frac{1}{p}\ell}\n{\Deltav_\ell f}_{\widetilde{L^p}(\mathbb{R};\dB_{p,1}^{s_1}(\mathbb{R}^2))} \n{\Deltav_jg}_{\widetilde{L^p}(\mathbb{R};\dB_{p,1}^{s_2}(\mathbb{R}^2))}\\
        &\quad\leq{}
        C
        \sum_{m \in\mathbb{Z}}
        2^{(\sigma+\sigma_1) m}
        \sum_{|j-m|\leq2} \sum_{\ell \leq j-3} 2^{(\frac{1}{p}-\sigma_1)\ell}\n{\Deltav_\ell f}_{\widetilde{L^p}(\mathbb{R};\dB_{p,1}^{s_1}(\mathbb{R}^2))} \n{\Deltav_jg}_{\widetilde{L^p}(\mathbb{R};\dB_{p,1}^{s_2}(\mathbb{R}^2))}\\
        &\quad\leq{}
        C
        \sum_{\ell \in \mathbb{Z}} 
        2^{(\frac{1}{p}-\sigma_1)\ell}\n{\Deltav_\ell f}_{\widetilde{L^p}(\mathbb{R};\dB_{p,1}^{s_1}(\mathbb{R}^2))} 
        \sum_{m \in\mathbb{Z}}
        2^{(\sigma+\sigma_1) m}
        \n{\Deltav_m g}_{\widetilde{L^p}(\mathbb{R};\dB_{p,1}^{s_2}(\mathbb{R}^2))}\\
        &\quad={}
        C
        \n{f}_{\fB_p^{s_1,\frac{1}{p}-\sigma_1}}
        \n{g}_{\fB_p^{s_2,\sigma+\sigma_1}}.
    \end{align}
    Similarly, we have 
    \begin{align}
        \n{T^{\rm v}_gf}_{\fB_p^{s_0-\frac{2}{p},\sigma}}
        \leq
        C
        \n{f}_{\fB_p^{s_3,\sigma+\sigma_2}}
        \n{g}_{\fB_p^{s_4,\frac{1}{p}-\sigma_2}}.
    \end{align}
    Using \eqref{2D-prod}, we have for $2 \leq p < \infty $ that
    \begin{align}
        &
        \n{R^{\rm v}(f,g)}_{\fB_p^{s_0-\frac{2}{p},\sigma}}
        =
        \sum_{m \in \mathbb{Z}}
        2^{\sigma m }
        \n{\Deltav_mR^{\rm v}(f,g)}_{\widetilde{L^p}(\mathbb{R};\dB_{p,1}^{s_0-\frac{2}{p}}(\mathbb{R}^2))}\\
        &\quad
        \leq
        C
        \sum_{m \in \mathbb{Z}}
        2^{(\sigma + \frac{1}{p})m}
        \n{\Deltav_mR^{\rm v}(f,g)}_{\widetilde{L^{\frac{p}{2}}}(\mathbb{R};\dB_{p,1}^{s_0-\frac{2}{p}}(\mathbb{R}^2))}
        \\
        &\quad
        \leq
        C
        \sum_{m \in \mathbb{Z}}
        \sum_{j \geq m-2} \sum_{|\ell -j|\leq 2}
        2^{(\sigma+\frac{1}{p})(m-j)}
        2^{(\sigma+\frac{1}{p})j}
        \n{ \Deltav_\ell f}_{\widetilde{L^p}(\mathbb{R};\dB_{p,1}^{s_1}(\mathbb{R}^2))} 
        \n{\Deltav_jg}_{\widetilde{L^p}(\mathbb{R};\dB_{p,1}^{s_2}(\mathbb{R}^2))}\\
        &\quad
        \leq
        C
        \sum_{j \in \mathbb{Z}}
        \sum_{|\ell -j|\leq 2}
        2^{(\sigma+\frac{1}{p})j}
        \n{ \Deltav_\ell f}_{\widetilde{L^p}(\mathbb{R};\dB_{p,1}^{s_1}(\mathbb{R}^2))} 
        \n{\Deltav_jg}_{\widetilde{L^p}(\mathbb{R};\dB_{p,1}^{s_2}(\mathbb{R}^2))}\\
        &\quad\leq{}
        C
        \sum_{\ell \in \mathbb{Z}} 
        2^{(\frac{1}{p}-\sigma_1)\ell}\n{\Deltav_\ell f}_{\widetilde{L^p}(\mathbb{R};\dB_{p,1}^{s_1}(\mathbb{R}^2))} 
        \sum_{j \in\mathbb{Z}}
        2^{(\sigma+\sigma_1) j}
        \n{\Deltav_jg}_{\widetilde{L^p}(\mathbb{R};\dB_{p,1}^{s_2}(\mathbb{R}^2))}\\
        &\quad={}
        C
        \n{f}_{\fB_p^{s_1,\frac{1}{p}-\sigma_1}}
        \n{g}_{\fB_p^{s_2,\sigma+\sigma_1}}.
    \end{align}
    For the case of $1 \leq p <2$, we have 
    \begin{align}
        &
        \n{R^{\rm v}(f,g)}_{\fB_p^{s_0-\frac{2}{p},\sigma}}
        =
        \sum_{m \in \mathbb{Z}}
        2^{\sigma m }
        \n{\Deltav_mR^{\rm v}(f,g)}_{\widetilde{L^p}(\mathbb{R};\dB_{p,1}^{s_0-\frac{2}{p}}(\mathbb{R}^2))}\\
        &\quad
        \leq
        C
        \sum_{m \in \mathbb{Z}}
        2^{(\sigma + (1-\frac{1}{p}))m}
        \n{\Deltav_mR^{\rm v}(f,g)}_{\widetilde{L^1}(\mathbb{R};\dB_{p,1}^{s_0-\frac{2}{p}}(\mathbb{R}^2))}
        \\&\quad
        \leq
        C
        \sum_{j \in \mathbb{Z}}
        \sum_{|\ell -j|\leq 2}
        2^{(\sigma+(1-\frac{1}{p}))j}
        \n{ \Deltav_\ell f}_{\widetilde{L^{\frac{p}{p-1}}}(\mathbb{R};\dB_{p,1}^{s_1}(\mathbb{R}^2))} 
        \n{\Deltav_jg}_{\widetilde{L^p}(\mathbb{R};\dB_{p,1}^{s_2}(\mathbb{R}^2))}\\
        &\quad\leq{}
        C
        \sum_{j \in \mathbb{Z}}
        \sum_{|\ell -j|\leq 2}
        2^{(\sigma+\frac{1}{p})j}
        \n{ \Deltav_\ell f}_{\widetilde{L^{p}}(\mathbb{R};\dB_{p,1}^{s_1}(\mathbb{R}^2))} 
        \n{\Deltav_jg}_{\widetilde{L^p}(\mathbb{R};\dB_{p,1}^{s_2}(\mathbb{R}^2))}\\
        &\quad\leq{}
        C
        \sum_{\ell \in \mathbb{Z}} 
        2^{(\frac{1}{p}-\sigma_1)\ell}\n{\Deltav_\ell f}_{\widetilde{L^p}(\mathbb{R};\dB_{p,1}^{s_1}(\mathbb{R}^2))} 
        \sum_{j \in\mathbb{Z}}
        2^{(\sigma+\sigma_1) j}
        \n{\Deltav_jg}_{\widetilde{L^p}(\mathbb{R};\dB_{p,1}^{s_2}(\mathbb{R}^2))}\\
        &\quad={}
        C
        \n{f}_{\fB_p^{s_1,\frac{1}{p}-\sigma_1}}
        \n{g}_{\fB_p^{s_2,\sigma+\sigma_1}}.
    \end{align}
    Thus, we complete the proof. 
    Here, we remark that it is the estimate of $R^{\rm v}(f,g)$ that the assumption of $\sigma$ comes into play. 
\end{proof}
By the similar argument as above and the fact that
$B_{p,1}^\tau(\mathbb{R}^d)$ is a Banach algebra for $\tau\geq d/p$, $d \in \mathbb{N}$, we immediately obtain the following lemma.
\begin{lemm}\label{lemm:prod-inhomo}
    For $1 \leq p \leq \infty$, $s\geq 2/p$, and $\sigma \geq 1/p$, there exists a positive constant $C=C(p,s,\sigma)$ such that 
    \begin{align}
        \n{fg}_{\ifB_p^{s,\sigma}}
        \leq
        C
        \n{f}_{\ifB_p^{s,\sigma}}
        \n{g}_{\ifB_p^{s,\sigma}}
    \end{align}
    for all $f,g \in \ifB_p^{s,\sigma}(\mathbb{R}^3)$.
\end{lemm}

For the convenience of the reader, we recall the following two lemmas that will be tacitly used in the sequel without pointing out each time.
First, we recall the boundedness of zeroth Fourier multipliers in anisotropic Besov spaces from \cite{FL-24}*{Lemma 3.6}. 
For a function $M=M(\xi)$ and $M'=M'(\xih)$ on the Fourier space, we define their Fourier multipliers $\mathscr{M}_M$ and $\mathscr{M}_{M'}$ by 
\begin{align}
    \mathscr{M}_Mf
    :=
    \mathscr{F}_{\R^3}^{-1}
    \lp{
    M(\xi)
    \mathscr{F}_{\R^3}[f](\xi)
    }, 
    \quad
    \mathscr{M}_{M'}f
    :=
    \mathscr{F}_{\R^3}^{-1}
    \lp{
    M'(\xih)
    \mathscr{F}_{\R^3}[f](\xi)
    }.
\end{align}
\begin{lemm}\label{lemm:M-bdd}
        Let $M=M(\xi)\in C^{\infty}(\R^3_{\xi}\setminus \{0\})$ and $M'=M'(\xih) \in C^{\infty}(\R^2_{\xih}\setminus \{0\})$ be homogeneous functions of $0$-th order.
        Then, there exists a positive constant $C=C(M,M')$ such that 
        \begin{align}
            \n{\mathscr{M}_{M}f}_{ \fB^{s,\sigma}_p}
            \leq{}
            C\n{f}_{ \fB^{s,\sigma}_p}, \qquad
            \n{\mathscr{M}_{M'}f}_{ \fB^{s,\sigma}_p}
            \leq{}
            C\n{f}_{ \fB^{s,\sigma}_p}
        \end{align}
        for all $  s,\sigma \in \R,1\leq p \leq \infty$, and $f \in \fB^{s,\sigma}_p(\R^3)$.
\end{lemm}
To proceed, we recall the interpolation inequality in anisotropic Besov spaces from \cite{FL-24}*{Lemma 3.4}.
\begin{lemm}\label{lemm:inter}
    Let $1\leq p \leq \infty$.
    Let $s,s_1,s_2,\sigma,\sigma_1,\sigma_2 \in \mathbb{R}$ and $0 \leq \vartheta \leq 1$ 
    satisfy 
    $s=\vartheta s_1+(1-\vartheta)s_2$, 
    and 
    $\sigma=\vartheta \sigma_1+(1-\vartheta)\sigma_2$.
    Then, it holds
    \begin{align}
        \n{f}_{\fB_{p}^{s,\sigma}}
        \leq 
        \n{f}_{\fB_{p}^{s_1,\sigma_1}}^{\vartheta}
        \n{f}_{\fB_{p}^{s_2,\sigma_2}}^{1-\vartheta}
    \end{align}
    for all $f \in \fB_{p}^{s_1,\sigma_1}(\R^3) \cap \fB_{p}^{s_2,\sigma_2}(\R^3)$.
\end{lemm}

Finally, we consider the properties of the compactness properties, which are to be used for the convergence of the approximate solutions for \eqref{eq:ANS} in the proof of Theorem \ref{thm:anal}. 
We begin with the Fatou property in the anisotropic Besov spaces.
\begin{lemm}\label{lemm:Fatou}
    Let $1 \leq p \leq \infty$ and $s,\sigma \in \mathbb{R}$.
    Then, there exists a positive constant $C$ such that 
    for any bounded sequence $\{f_n \}_{n=1}^{\infty}$ in $\ifB_p^{s,\sigma}(\mathbb{R}^3)$,
    there exists a subsequence $\{f_{n_m} \}_{m=1}^{\infty}$ and a $f \in \ifB_p^{s,\sigma}(\mathbb{R}^3)$ such that $f_{n_m} \to f$ in $\mathscr{S}'(\mathbb{R}^3)$ as $m \to \infty$ and  
    \begin{align}
        \n{f}_{\ifB_p^{s,\sigma}}
        \leq 
        C
        \liminf_{m \to \infty}
        \n{f_{n_m}}_{\ifB_p^{s,\sigma}}.
    \end{align}
\end{lemm}
\begin{proof}
The proof follows from \cite{Bah-Che-Dan-11}*{Theorem 2.25}, so we only provide the main steps. 
By the Bernstein inequality, we see for any $k,j\in \mathbb{N}\cup \{0\}$ the sequence $\{ \iDh_k \iDv_j f_n \}_{n=1}^{\infty}$ is bounded in $L^p(\R^3)\cap L^{\infty}(\R^3)$. By the diagonal argument, there exists a subsequence $\{ f_{n_{m}} \}_{m=1}^{\infty}$ and a sequence $\{\widetilde{f}_{k,j} \}_{k,j \in \mathbb{N}\cup \{0\}}$ of smooth functions with Fourier transform supported in 
\begin{align}
 &  \{ \xi=(\xih,\xi_3)\ ;\ 2^{k-1}\leq |\xih|\leq 2^{k+1},\ 2^{j-1}\leq |\xi_3|\leq 2^{j+1} \} 
\quad \text{if}\quad k,j \geq 1, \\
&
\{ \xi=(\xih,\xi_3)\ ;\   |\xih|\leq 2 ,\ 2^{j-1}\leq |\xi_3|\leq 2^{j+1} \}
     \quad \text{if}\quad k=0,j \geq 1, \\
 &  \{ \xi=(\xih,\xi_3)\ ;\  2^{k-1}\leq |\xih|\leq 2^{k+1},\ |\xi_3|\leq 2 \}
     \quad \text{if}\quad k\geq 1, j=0, \\
&
     \{ \xi=(\xih,\xi_3)\ ;\   |\xih|\leq 2,\ |\xi_3|\leq 2 \}
     \quad \text{if}\quad k=j=0, 
\end{align}
and such that for any $k,j\in \mathbb{N}\cup \{0\}$ and $\psi \in \mathscr{S}(\mathbb{R}^3)$ 
\begin{align}
    \lim_{m \rightarrow \infty}
    \lr{
    \iDh_k \iDv_j f_{n_{m}}, 
    \psi
    }
    = \lr{
    \widetilde{f}_{k,j},
    \psi
    }, \qquad 
    \n{
    \widetilde{f}_{k,j}
    }_{L^p} 
    \leq  \liminf_{m \rightarrow \infty}
    \n{
    \iDh_k \iDv_j   f_{n_{m}}
    }_{L^p} . 
\end{align}
Then, we see that the sequence $\Mp{\{ 2^{s k} 2^{\sigma j} 
\| \iDh_k \iDv_j f_{n_{m}} \|_{L^p} \}_{k,j\in \mathbb{N}\cup \{0\}}}_{m\in \mathbb{N}}$ is bounded in $\ell^{1}((\mathbb{N}\cup\{0\})^2)$. Thus, there exists a sequence $\{ c_{k,j}\}_{k,j\in \mathbb{N}\cup \{0\}}$ of $\ell^{1}((\mathbb{N}\cup\{0\})^2)$ such that 
for any sequence $\{ d_{k,j}\}_{k,j\in \mathbb{N}\cup \{0\}}$ satisfying $d_{k,j} \rightarrow 0$ as $|(k,j)|\rightarrow \infty$, it holds 
\begin{align}
 &  \lim_{m \rightarrow \infty}
   \sum_{k,j\in \mathbb{N}\cup \{0\}}
   2^{s k} 2^{\sigma j} \n{  \iDh_k \iDv_j f_{n_{m}} }_{L^p}
   d_{k,j}
   =
   \sum_{k,j\in \mathbb{N}\cup \{0\}}
   c_{k,j} d_{k,j}, \\
& \n{ \{  c_{k,j} \}_{k,j\in \mathbb{N}\cup \{0\}} }_{\ell^{1}((\mathbb{N}\cup\{0\})^2)} 
\leq 
 \liminf_{m \rightarrow \infty}
 \n{   f_{n_{m}}  }_{\ifB_p^{s,\sigma}}. 
\end{align}
From the above relation, we see that the summation  
\begin{align}
    f:=\sum_{k,j\in \mathbb{N}\cup \{0\}} \widetilde{f}_{k,j} 
\end{align}
converges in $\mathscr{S}'(\mathbb{R}^3)$. Moreover, from the definition of $\widetilde{f}_{k,j}$ we deduce that for all $0\leq M_1 <N_1,0\leq M_2<N_2$, 
\begin{align}
    \sum_{k=M_1}^{N_1}  \sum_{j=M_2}^{N_2}
     \iDh_k \iDv_j  f 
     = \lim_{m \rightarrow \infty}
     \sum_{k=M_1}^{N_1}  \sum_{j=M_2}^{N_2} 
     \iDh_k \iDv_j f_{n_{m}}, 
\end{align}
which shows that $ f_{n_{m}}$ converges to $f$ in $\mathscr{S}'(\mathbb{R}^3)$.    
\end{proof}

Similarly to Lemma \ref{lemm:Fatou} and \cite{Bah-Che-Dan-11}*{Theorem 2.25}, we also obtain the Fatou property for homogeneous anisotropic Besov spaces. We remark that the constraint $s_1 \leq 2/p$ below is used to ensure the convergence of series of the homogeneous Littlewood--Paley decomposition. 
\begin{lemm}\label{lemm:Fatou-time}
    Let $1 \leq p \leq \infty$ and $s_1,s_2,\sigma \in \mathbb{R}$ with $s_1 \leq 2/p$ and $\sigma \leq 1/p$.
    Then, there exists a positive constant $C$ such that for any $\{ f_n \}_{n=1}^{\infty} \subset \fB_p^{s_1,\sigma}(\mathbb{R}^3) \cap \fB_p^{s_2,\sigma}(\mathbb{R}^3)$ be a bounded sequence, there exists a subsequence $\{f_{n_m}\}_{m=1}^{\infty}$ and $f \in \fB_p^{s_1,\sigma}(\mathbb{R}^3) \cap \fB_p^{s_2,\sigma}(\mathbb{R}^3)$ such that 
    $f_{n_m} \to f$ in $\mathscr{S}'(\mathbb{R}^3)$ and 
    \begin{align}
        \n{f}_{\fB_p^{s_k,\sigma} }
        \leq
        C
        \liminf_{m \to \infty}
        \n{f_{n_m}}_{\fB_p^{s_k,\sigma}}, 
        \quad k=1,2.
    \end{align} 
\end{lemm}
Next, we investigate the relation between the anisotropic homogeneous and inhomogeneous Besov spaces.
\begin{lemm}\label{lemm:embedding}
    Let $1 \leq p \leq \infty$ and $s,\sigma>0$.
    \begin{enumerate}
        \item 
        There exists a positive constant $C=C(p,s,\sigma)$ such that
        \begin{align}
            C^{-1}\n{f}_{\ifB_p^{s,\sigma}} 
            \leq{}& 
            \n{f}_{L^p} + \n{f}_{\fB_p^{s,\sigma}}
            \\
            &+
            \sum_{j \geq 1}
            2^{\sigma j}
            \n{\Deltav_j f}_{L^p}
            +
            \sum_{k \geq 1} 
            2^{sk}
            \n{\Deltahh_k f}_{L^p}
            \leq C \n{f}_{\ifB_p^{s,\sigma}}
        \end{align}
        for all $f \in \ifB_p^{s,\sigma}(\mathbb{R}^3)$. 
        \item 
        For any $R>0$, there exists a positive constant $C_R=C_R(p,s,\sigma)$ such that 
        \begin{align}
            \n{\psi f}_{\ifB_p^{s,\sigma}}
            \leq
            C_R
            \n{\psi f}_{\fB_p^{s,\sigma}}
        \end{align}
        for all $\psi \in C_c^{\infty}(\mathbb{R}^3)$ with $\supp \psi \subset B_R:=\{x \in\mathbb{R}^3\ ;\ |x| \leq R\}$ and $f \in \fB_p^{s,\sigma}(\mathbb{R}^3)$.
    \end{enumerate}
\end{lemm}
\begin{proof}
Let us decompose the homogeneous norm of $f$ as 
    \begin{align}\label{dec}
        \n{f}_{\fB_p^{s,\sigma}}
        ={}&
        \sum_{k \leq 0} 
        \sum_{j \leq 0}
        2^{sk}2^{\sigma j}
        \n{\Deltahh_k\Deltav_j f}_{L^p}
        +
        \sum_{k \leq 0} 
        \sum_{j \geq 1}
        2^{sk}2^{\sigma j}
        \n{\Deltahh_k\Deltav_j f}_{L^p}\\
        &
        +
        \sum_{k \geq 1} 
        \sum_{j \leq 0}
        2^{sk}2^{\sigma j}
        \n{\Deltahh_k\Deltav_j f}_{L^p}
        +
        \sum_{k \geq 1}
        \sum_{j \geq 1} 
        2^{sk}2^{\sigma j}
        \n{\Deltahh_k\Deltav_j f}_{L^p}.
    \end{align}
    We then, see that 
    \begin{align}
        &\sum_{k \leq 0} 
        \sum_{j \leq 0}
        2^{sk}2^{\sigma j}
        \n{\Deltahh_k\Deltav_j f}_{L^p}
        \leq
        C
        \sum_{k \leq 0} 
        \sum_{j \leq 0}
        2^{sk}2^{\sigma j}
        \n{\iDh_0\iDv_0f}_{L^p}
        \leq 
        C
        \n{\iDh_0\iDv_0f}_{L^p},
        \\
        &
        \begin{aligned}
        \sum_{k \leq 0} 
        \sum_{j \geq 1}
        2^{sk}2^{\sigma j}
        \n{\Deltahh_k\Deltav_j f}_{L^p}
        \leq 
        C
        \sum_{k \leq 0} 
        2^{sk}
        \sum_{j \geq 1}
        2^{\sigma j}
        \n{\iDh_0\iDv_j f}_{L^p}
        \leq 
        C
        \sum_{j \geq 1}
        2^{\sigma j}
        \n{\iDh_0\iDv_j f}_{L^p},
        \end{aligned}
        \\
        &
        \sum_{k \geq 1} 
        \sum_{j \leq 0}
        2^{sk}2^{\sigma j}
        \n{\Deltahh_k\Deltav_j f}_{L^p}
        \leq
        C
        \sum_{j \leq 0}
        2^{\sigma j}
        \sum_{k \geq 1} 
        2^{sk}
        \n{\iDh_k\iDv_0 f}_{L^p}
        \leq
        C
        \sum_{k \geq 1} 
        2^{sk}
        \n{\iDh_k\iDv_0 f}_{L^p},
    \end{align}
    For the estimate of $\n{f}_{L^p}$, it holds
    \begin{align}
        \n{f}_{L^p}
        \leq
        \sum_{j,k \in \mathbb{N}\cup\{0\}}
        \n{\iDh_k\iDv_jf}_{L^p}
        \leq
        C
        \sum_{j,k \in \mathbb{N}\cup\{0\}}
        2^{sk}2^{\sigma j}
        \n{\iDh_k\iDv_jf}_{L^p}
        \leq
        C
        \n{f}_{\ifB_p^{s,\sigma}}.
    \end{align}
    Moreover, we see that 
    \begin{align}
        &
        \sum_{j \geq 1}
        2^{\sigma j}
        \n{\Deltav_j f}_{L^p}
        \leq
        \sum_{j \geq 1}
        2^{\sigma j}
        \n{\iDh_0\Deltav_j f}_{L^p}
        +
        \sum_{j,k \geq 1}
        2^{sk}2^{\sigma j}
        \n{ \Deltahh_k  \Deltav_j f}_{L^p}
        \leq
        C
        \n{f}_{\ifB_p^{s,\sigma}},
        \\
        &
        \sum_{k \geq 1} 
        2^{sk}
        \n{\Deltahh_k f}_{L^p}
        \leq
        \sum_{k \geq 1} 
        2^{sk}
        \n{\Deltahh_k \iDv_0 f}_{L^p}
        +
        C
        \sum_{k,j \geq 1} 
        2^{sk}2^{\sigma j}
        \n{\Deltahh_k \Deltav_j f}_{L^p}
        \leq
        C
        \n{f}_{\ifB_p^{s,\sigma}}.
    \end{align}
    Thus, we have 
    \begin{align}
        \n{f}_{L^p} 
        + 
        \n{f}_{\fB_p^{s,\sigma}}
        +
        \sum_{j \geq 1}
        2^{\sigma j}
        \n{\Deltav_j f}_{L^p}
        +
        \sum_{k \geq 1} 
        2^{sk}
        \n{\Deltahh_k f}_{L^p}
        \leq C \n{f}_{\ifB_p^{s,\sigma}}.
    \end{align}
    For the reverse inequality, we see that 
    \begin{align}
        \n{f}_{\ifB_p^{s,\sigma}}
        ={}& 
        \n{\iDh_0\iDv_0f}_{L^p}
        +\sum_{j \geq 1}
        2^{\sigma j}
        \n{\iDh_0\Deltav_j f}_{L^p}
        \\
        &+
        \sum_{k \geq 1} 
        2^{sk}
        \n{\Deltahh_k\iDv_0 f}_{L^p}
        +
        \sum_{k \geq 1}
        \sum_{j \geq 1} 
        2^{sk}2^{\sigma j}
        \n{\Deltahh_k\Deltav_j f}_{L^p}
        \\ 
        \leq{}& 
        C
        \n{f}_{L^p}
        +
        C
        \sum_{j \geq 1}
        2^{\sigma j}
        \n{\Deltav_j f}_{L^p}
        +
        C
        \sum_{k \geq 1} 
        2^{sk}
        \n{\Deltahh_kf}_{L^p}
        +
        \n{f}_{\fB_p^{s,\sigma}}.
\end{align}  
This completes the proof of the first assertion.

The proof of the second assertion is inspired by \cite{Bah-Che-Dan-11}*{Proposition 2.93}. By the first assertion, we see that 
\begin{align}
    \n{\psi f}_{\ifB_p^{s,\sigma}}
    \leq{}&
    C 
    \n{\psi f}_{L^p(B_R)} 
    + 
    C
    \n{\psi f}_{\fB_p^{s,\sigma}}\\
    &
    + 
    C
    \sum_{j \geq 1}
    2^{\sigma j}
    \n{\Deltav_j (\psi f)}_{L^p((-R,R)^2 \times \mathbb{R})}
    +
    C
    \sum_{k \geq 1} 
    2^{sk}
    \n{\Deltahh_k (\psi f) }_{L^p(\mathbb{R}^2 \times (-R,R))}  
    . 
\end{align}
We begin with the estimate of $\n{\psi f}_{L^p(B_R)}$. Let $k_R,j_R\in \mathbb{Z}$ be negative integers to be determined later. 
Then, it holds that 
\begin{align}
\psi f
& =
\sum_{k\leq k_R,j\leq j_R} \Deltahh_{k} \Deltav_{j}  (\psi f)
+
\sum_{k> k_R,j > j_R} \Deltahh_{k} \Deltav_{j} (\psi f) \\
& \quad 
+ 
\sum_{k\leq k_R,j> j_R} \Deltahh_{k} \Deltav_{j} (\psi f)
+
\sum_{k > k_R,j\leq j_R} \Deltahh_{k} \Deltav_{j} (\psi f). 
\end{align}
Due to $\supp (\psi f) \subset B(0,R)$, we have 
\begin{align}
    \n{\psi f}_{L^p( B_R  ) } 
    &  \leq 
    \n{
    \sum_{k\leq k_R,j\leq j_R} \Deltahh_{k} \Deltav_{j}  (\psi f)
    }_{L^p( B_R  ) } 
    + 
    \n{
    \sum_{k> k_R,j > j_R} \Deltahh_{k} \Deltav_{j} (\psi f) 
    }_{L^p( B_R  ) }  
    \\
    &
    \quad + 
    \n{
    \sum_{k \leq k_R,j> j_R} \Deltahh_{k} \Deltav_{j} (\psi f)
    }_{L^p( B_R  ) } 
    +
    \n{
    \sum_{k > k_R, j\leq j_R} \Deltahh_{k} \Deltav_{j} (\psi f)
    }_{L^p( B_R  ) } 
    . 
\end{align}
By Bernstein's inequality and the fact that $\supp (\psi f) \subset B(0,R)$, we have 
\begin{align}
    \n{
    \sum_{k\leq k_R,j\leq j_R} \Deltahh_{k} \Deltav_{j}  (\psi f)
    }_{L^p( B_R  ) } 
    &  
    \leq 
    CR^{\f{3}{p}} 
    \n{
    \sum_{k\leq k_R,j\leq j_R} \Deltahh_{k} \Deltav_{j}  (\psi f)
    }_{L^{\infty}( \R^3  ) }  \\
    & 
    \leq 
    C  R^{\f{3}{p}}  2^{2 k_R} 2^{j_R} \n{ \psi f}_{L^1(\R^3)}
    \\
    & \leq 
    C R^3  2^{2 k_R} 2^{j_R} \n{ \psi f}_{L^p(\R^3)} . 
\end{align}
In view of $s,\sigma>0$, we see
\begin{align}
  \n{
  \sum_{k > k_R,j > j_R} \Deltahh_{k} \Deltav_{j} (\psi f) 
  }_{L^p( B_R  ) }
  \leq {}&
   \n{
  \sum_{k > k_R,j > j_R} \Deltahh_{k} \Deltav_{j} (\psi f) 
  }_{L^p( \R^3  ) } \\
  \leq {}&
  C 2^{-s k_R} 2^{- \sigma j_R} \n{\psi f}_{\fB_p^{s,\sigma}} . 
\end{align}
For the other terms, we have 
\begin{align}
    &
    \n{
    \sum_{k \leq k_R,j> j_R} \Deltahh_{k} \Deltav_{j} (\psi f)
    }_{L^p( B_R )}
    \\
    &\quad
    \leq{}
    C
  R^{\frac{2}{p}}  
    \sum_{j>j_R}
    \n{
    \sum_{k \leq k_R}
    \Deltahh_{k} \Deltav_{j} (\psi f)
    }_{L^p( \mathbb{R}_{x_3};L^{\infty}(\mathbb{R}^2_{\xh}) )}
    \\
    &\quad
    \leq{}
    C
   R^{\frac{2}{p}} 
    2^{k_R}
    \sp{
    \n{
    \psi f
    }_{L^p( \mathbb{R}_{x_3};L^1((-R,R)^2) )}
    +
    \sum_{j \geq 1}
    2^{\sigma j}
    \n{
    \Deltav_{j} (\psi f)
    }_{L^p( \mathbb{R}_{x_3};L^1((-R,R)^2) )}}
    \\
    &\quad
    \leq{}
    C
  R^2  
    2^{k_R}
    \sp{
    \n{\psi f}_{L^p(B_R)}
    +
    \sum_{j \geq 1}
    2^{\sigma j}
    \n{
    \Deltav_{j} (\psi f)
    }_{L^p( (-R,R)^2 \times \mathbb{R}_{x_3})} },
\end{align}
and we also see that 
\begin{align}
    &
    \n{
    \sum_{k > k_R, j\leq j_R} \Deltahh_{k} \Deltav_{j} (\psi f)
    }_{L^p( B_R  ) } \\
    &\quad 
    \leq
    C
 R  
    2^{j_R}
    \sp{
    \n{\psi f}_{L^p(B_R)}
    +
    \sum_{k \geq 1} 
    2^{sk}
    \n{\Deltahh_k (\psi f) }_{L^p(\mathbb{R}^2 \times (-R,R))}
    }.
\end{align}
Hence, choosing $k_R$ and $j_R$ so small that 
\begin{align}
    C R^3  2^{2 k_R} 2^{j_R}
    +
    C
    R^{2}  
    2^{k_R}
    +
    C
    R  
    2^{j_R}
    \leq \frac{1}{2},
\end{align}
we have 
\begin{align}
    \n{\psi f}_{\ifB_p^{s,\sigma}}
    \leq{}
    C_R
    \n{\psi f}_{\fB_p^{s,\sigma}}
    &
    +
    C
    \sum_{j \geq 1}
    2^{\sigma j}
    \n{\Deltav_j (\psi f)}_{L^p((-R,R)^2 \times \mathbb{R})}
    \\
    &
    +
    C
    \sum_{k \geq 1} 
    2^{sk}
    \n{\Deltahh_k (\psi f) }_{L^p(\mathbb{R}^2 \times (-R,R))}. 
\end{align}
Invoking again $s>0$, we know that $\dB_{p,1}^s(K)=B_{p,1}^s(K)$ for all compact set $K \subset \mathbb{R}^2$; see \cite{Bah-Che-Dan-11}*{Proposition 2.93}. Hence, we have 
\begin{align}
    \sum_{j \geq 1}
    2^{\sigma j}
    \n{\Deltav_j (\psi f)}_{L^p((-R,R)^2 \times \mathbb{R})}
    & 
    \leq  
    \sum_{j \geq 1}
    \sum_{k \geq 0 } 
    2^{sk} 2^{\sigma j} 
    \n{\iDh_k  \Deltav_j (\psi f)}_{L^p}  
    \\
    & 
    \leq 
    C_R
    \sum_{j \geq 1}
    \sum_{k \in \mathbb{Z}}
    2^{s k}
    2^{\sigma j}
    \n{\Deltahh_k\Deltav_j (\psi f)}_{L^p}
    \leq
    C_R\n{\psi f}_{\fB_p^{s,\sigma}}.
\end{align}
Similarly, it follows from $\sigma>0$ that  
\begin{align}
    \sum_{k \geq 1} 
    2^{sk}
    \n{\Deltahh_k (\psi f) }_{L^p(\mathbb{R}^2 \times (-R,R))}
    \leq
    C_R\n{\psi f}_{\fB_p^{s,\sigma}}.
\end{align}
Thus, we complete the proof.   
\end{proof}

Finally, we investigate the compactness property between the anisotropic inhomogeneous Besov spaces. 
\begin{lemm}\label{lemm:comp}
    Let $1\leq p \leq \infty$, $s_1 \geq 2/p$, $\sigma_1 \geq 1/p$, $0<s_2<s_1$, $0<\sigma_2<\sigma_1$, and $\psi \in \mathscr{S}(\mathbb{R}^3)$.
    Then, the linear map $\ifB_p^{s_1,\sigma_1}(\mathbb{R}^3) \ni f \mapsto \psi f \in \ifB_p^{s_2,\sigma_2}(\mathbb{R}^3)$ is compact.
\end{lemm}
\begin{rem}
    The assumptions $s_1 \geq 2/p$, $\sigma_1 \geq 1/p$ are not actually necessary.
    In the general case, it suffices to smooth out the functions and use the diagonal process. We omit the detail for this direction since the stronger assumptions are enough for our purpose.
\end{rem}
\begin{proof}
The proof adapts from \cite{Bah-Che-Dan-11}*{Theorem 2.94}. 
Let $\{f_n\}_{n=1}^{\infty}$ be a bounded sequence of $\ifB_p^{s_1,\sigma_1}(\mathbb{R}^3)$. It follows from Lemma \ref{lemm:Fatou} that there exist a subsequence $\{f_{n_m} \}_{m=1}^{\infty}$ and a $f \in \ifB_p^{s_1,\sigma_1}(\mathbb{R}^3)$ such that $f_{n_m}  \to f$ in $\mathscr{S}'(\mathbb{R}^3)$. Hence, the proof of Lemma \ref{lemm:comp} reduces to showing that if $\{f_n\}_{n=1}^{\infty}$ is a bounded sequence of $\ifB_p^{s_1,\sigma_1}(\mathbb{R}^3)$ which converges to $0$ in $\mathscr{S}'(\mathbb{R}^3)$, then $\n{\psi f_n}_{\ifB_p^{s_2,\sigma_2}}$ tends to $0$. 
It holds that 
\begin{align}
    &\n{\psi f_n}_{\ifB_p^{s_2,\sigma_2}}
    \\
    &\quad = \sum_{k,j \in \mathbb{N}\cup \{ 0\} } 
    2^{s_2 k} 2^{\sigma_2 j} \n{ \iDh_k \iDv_j (\psi f_n) }_{L^p} 
    \\
    &\quad
    = \sum_{k \leq k_0} \sum_{j \leq j_0}  2^{s_2 k} 2^{\sigma_2 j}\n{ \iDh_k \iDv_j (\psi f_n) }_{L^p}
    + 
    \sum_{k \leq k_0} \sum_{j > j_0}  2^{s_2 k} 2^{\sigma_2 j} \n{ \iDh_k \iDv_j (\psi f_n) }_{L^p}  
    \\
    & \qquad 
    +   
    \sum_{k >k_0} \sum_{j \leq  j_0}  2^{s_2 k} 2^{\sigma_2 j} \n{ \iDh_k \iDv_j (\psi f_n) }_{L^p}  
    + 
    \sum_{k > k_0} \sum_{j > j_0}  2^{s_2 k} 2^{\sigma_2 j} \n{ \iDh_k \iDv_j (\psi f_n) }_{L^p} .
\end{align}
By Lemma \ref{lemm:prod-inhomo},
we infer that $\{\psi f_n\}_{n=1}^{\infty}$ is bounded in $\ifB_p^{s_1,\sigma_1}(\mathbb{R}^3)$. Then, we have 
\begin{align}
    &
    \sum_{k > k_0} \sum_{j > j_0}  2^{s_2 k} 2^{\sigma_2 j} \n{ \iDh_k \iDv_j (\psi f_n) }_{L^p} 
    \\
    &\quad =
    \sum_{k > k_0} \sum_{j > j_0} 
    2^{-(s_1-s_2)k}
    2^{s_1 k} 
    2^{-(\sigma_1-\sigma_2)j}
    2^{\sigma_1 j} \n{ \iDh_k \iDv_j (\psi f_n) }_{L^p} 
    \\
    &\quad \leq 
    2^{-(s_1-s_2) k_0}
    2^{-(\sigma_1-\sigma_2)j_0}
    \sp{
    \sup_{n \geq 1} \n{\psi f_n}_{\ifB_p^{s_1,\sigma_1}} 
    }\\
    &\quad \leq 
    C 2^{-(s_1-s_2) k_0}
    2^{-(\sigma_1-\sigma_2)j_0}.
\end{align}
Based on the continuous embeddings $\ifB_p^{s_1,\sigma_1}(\mathbb{R}^3)\hookrightarrow \ifB_p^{s_2,\sigma_1}(\mathbb{R}^3)$ and $\ifB_p^{s_1,\sigma_1}(\mathbb{R}^3) \hookrightarrow \ifB_p^{s_1,\sigma_2}(\mathbb{R}^3)$, we infer that 
\begin{align}
\begin{aligned}
    \sum_{k \leq k_0} \sum_{j > j_0}  2^{s_2 k} 2^{\sigma_2 j} \n{ \iDh_k \iDv_j (\psi f_n) }_{L^p}
    &= 
    \sum_{k \leq k_0} \sum_{j > j_0}  2^{s_2 k} 2^{-(\sigma_1-\sigma_2)j}
    2^{\sigma_1 j} \n{ \iDh_k \iDv_j (\psi f_n) }_{L^p} \\
    & \leq 
    2^{-(\sigma_1-\sigma_2)j_0}
    \sp{
   \sup_{n \geq 1} \n{\psi f_n}_{\ifB_p^{s_2,\sigma_1}} 
   }
   \leq C  2^{-(\sigma_1-\sigma_2)j_0},
   \end{aligned}  \\
   \begin{aligned}
       \sum_{k >k_0} \sum_{j \leq  j_0}  2^{s_2 k} 2^{\sigma_2 j} \n{ \iDh_k \iDv_j (\psi f_n) }_{L^p}
       &= 
       \sum_{k >k_0} \sum_{j \leq  j_0}  2^{-(s_1-s_2)k}
       2^{s_1 k} 2^{\sigma_2 j} \n{ \iDh_k \iDv_j (\psi f_n) }_{L^p}
       \\
       & 
       \leq 
       2^{-(s_1-s_2)k_0} 
        \sp{
   \sup_{n \geq 1} \n{\psi f_n}_{\ifB_p^{s_1,\sigma_2}} 
   }
   \leq C   2^{-(s_1-s_2)k_0} .
   \end{aligned}
\end{align}
Hence, we conclude that the proof of Lemma \ref{lemm:comp} reduces to 
\begin{align}\label{eq:comp}
    \lim_{n \to \infty} 
    \n{ \iDh_k \iDv_j (\psi f_n)}_{L^p}=0 \quad 
    \text{for all } k,j \geq 0.  
\end{align}
Indeed, Bernstein's inequality gives
\begin{align}
    \n{ \iDh_k \iDv_j (\psi f_n)}_{L^p}
    \leq C 2^{2 \f{k}{p'}} 2^{\f{j}{p'}}
     \n{ \iDh_k \iDv_j (\psi f_n)}_{L^1}.
\end{align}
In the sequel we will focus on the proof of \eqref{eq:comp} for $p=1$. We only show the case for $k,j\geq 1$ since the other cases are treated similarly. We define the translation operator $\tau_a:f \mapsto f(\cdot -a)$ and for simplicity set $\varphi_{\rm h}:=\mathscr{F}^{-1}_{\R^2}[\phi(|\xih|)]$, $\varphi_3:=\mathscr{F}^{-1}_{\R}[\phi(|\xi_3|)]$, where $\phi:[0,\infty) \to \R$ is the function used for the definition of the Littlewood--Paley decomposition. 
Observe that 
\begin{align}
    \iDh_k \iDv_j (\psi f_n) (x)
    &=
    2^{2k} 2^{j} \int_{\R^3} \varphi_{\rm h}( 2^k (\xh-y_{\rm h})  )
    \varphi_3( 2^j  (x_3-y_3)  )  \psi(y) f_n(y) dy \\
    & =
     2^{2k} 2^{j} 
     \lr{
     f_n, \tau_{-\xh} \varphi_{\rm h} (2^k \cdot)  \tau_{-x_3} 
     \varphi_3 (2^j \cdot) 
     \psi 
     }.
\end{align}
Recalling that $\{f_n\}_{n=1}^{\infty}$ converges to $0$ in $\mathscr{S}'(\mathbb{R}^3)$, we infer from the above equality that $\iDh_k \iDv_j (\psi f_n)$ tends to $0$ pointwise as $n\to \infty$. Furthermore, we see that 
\begin{align}
  &  \abso{ \iDh_k \iDv_j (\psi f_n) (x)}
  \\
  & \quad 
    \leq 
    C  2^{2k} 2^{j} 
  \sp{
  \sup_{n \geq 1} \n{ f_n}_{\ifB_p^{s_1,\sigma_1}}
  } 
  \n{
  \tau_{-\xh} \varphi_{\rm h} (2^k \cdot)  \tau_{-x_3} 
     \varphi_3 (2^j \cdot) 
     \psi 
  }_{\ifB_{p',\infty}^{-s_1,-\sigma_1}}. 
\end{align}
Here, we remark that the anisotropic Besov space $\ifB_{p',\infty}^{-s_1,-\sigma_1}(\R^3)$ is defined as Definition \ref{def:Besov} with the sequence-norm $\ell^1(\Z^2)$ replaced by $\ell^{\infty}(\Z^2)$; see \cite{FL-24} for more details. Due to Lebesgue's dominated convergence theorem, showing \eqref{eq:comp} for $p=1$ reduces to the following claim.
\begin{claim}
    Let $s,\sigma\in \R$, $1\leq p,q \leq \infty$, $g\in \mathscr{S}(\mathbb{R}^2_{\xh})$, $h\in \mathscr{S}(\mathbb{R}_{x_3})$, and $v \in \mathscr{S}(\mathbb{R}^3)$. Then, the function $\mathbb{R}^3 \ni (\xh,x_3) \mapsto \n{ (\tau_{\xh} g) (\tau_{x_3} h)v}_{\ifB_{p,q}^{s,\sigma} } \in \mathbb{R}$
    belongs to $L^1(\R^3)$.
\end{claim}
We notice that the proof of the above claim follows in the same manner as \cite{Bah-Che-Dan-11}*{Lemma 2.95}, and the details are thus omitted. This completes the proof of Lemma \ref{lemm:comp}. 
\end{proof}

\section{Global analytic a priori estimates}\label{sec:glo-est}
In this section, we provide the global space-time analytic a priori estimates of the solutions to \eqref{eq:ANS}, which plays the central role in the proof of Theorem \ref{thm:anal}. 
\subsection{Notations on analytic quantities}
To begin with, we prepare some notations used in the sequel.
For $p$ and $\theta$ satisfying \eqref{p-theta},
we set two subsets of $\mathbb{R}^2$ as
\begin{align}\label{eq:Lambda}
    &
    \Lambda^{\rm h}_{p,\theta}:=
    \left\{  
    \sp{ \f{2}{p}-1,\f{1}{p} }, 
    \sp{ \f{2}{p}-1+\theta,\f{1}{p}-\theta }, 
    \sp{ \f{2}{p}-2+\theta,\f{1}{p}-\theta }
    \right\},
    \\
    &
    \Lambda^{\rm v}_{p,\theta}:=
    \left\{  
    \sp{ \f{2}{p}-1,\f{1}{p} }, 
    \sp{ \f{2}{p}-2+\theta,\f{1}{p}-\theta }
    \right\}.
\end{align}  
For $M \in \mathbb{N}\cup \{0\}$, $s,\sigma\in \mathbb{R}$, $1\leq p \leq \infty$, $\lambda>0$, a non-negative and non-increasing $C^1$-function $g=g(t)$ on $[0,T)$, and for vector fields $v_0=v_0(x):\mathbb{R}^3 \to \mathbb{R}^3$ and $v=v(t,x):[0,\infty) \times \mathbb{R}^3 \to \mathbb{R}^3$, we define
\begin{align}
    &
    v^{\alpha,\beta}_{{\rm h};r;g}  (t)
    :=
    \frac{  
    r^{\alpha+|\betah|}  t^{\alpha+\f{|\betah|}{2}}  
    \rho_0^{\beta_3} 
    e^{-\beta_3 g(t)}
    }{\alpha!\beta!}
    \partial_{t}^{\alpha}  
    \partial_{x}^{\beta} 
    v_{\rm h}(t)
    ,
    \\
    &
    v^{\alpha,\betah}_{{\rm h};r}  (t)
    :=
    \frac{  r^{\alpha+|\betah|}  t^{\alpha+\f{|\betah|}{2}}   }{\alpha!\betah!}
    \partial_{t}^{\alpha}  
    \partial_{\xh}^{\betah} 
    v_{\rm h}(t),
    \\
    &
    v^{\alpha,\betah}_{3;r}(t)
    :=
    \frac{r^{\alpha+|\betah|} t^{\alpha+\f{|\betah|}{2}} }{\alpha!\betah!}
    \partial_{t}^{\alpha}  
    \partial_{\xh}^{\betah}  
    v_3(t),
\end{align}
and
\begin{align}
    v_{0,{\rm h}}^{\beta_3}:=\frac{ 
    \rho_0^{\beta_3} 
    }
    {
    \beta_3!
    }
    \partial_{x_3}^{\beta_3}v_{0,{\rm h}}.
\end{align} 
Moreover, for a scalar function $q: (0,\infty) \times \mathbb{R}^3 \to \mathbb{R}$, we set
\begin{align}
    &q^{\alpha,\beta}_{r;g}(t)
     :=
    \frac{
    r^{\alpha+|\betah|}t^{\alpha+\f{|\betah|}{2}}
    \rho_0^{\beta_3} 
    e^{-\beta_3 g(t)}
    }{\alpha!\beta!}
    \partial_{t}^{\alpha}  
    \partial_x^{\beta} q(t) 
    ,
    \\
    &q^{\alpha,\betah}_{r}(t)
     :=
    \frac{r^{\alpha+|\betah|}t^{\alpha+\f{|\betah|}{2}}  }{\alpha!\betah!}
    \partial_{t}^{\alpha}  
    \partial_{\xh}^{\betah} q(t).
\end{align}
Using these notations, we further introduce the following norms:
\begin{align}
    &
    \n{ v_{0,\rm h}}_{\widetilde{\mathcal{D}}_p^{s,\sigma} }
    :=
    \sum_{\beta_3 =0}^{\infty}
    \n{v^{\beta_3}_{0,\rm h}}_{\fB^{s,\sigma}_p}
    =
    \sum_{\beta_3 =0}^{\infty}
    \frac{\rho_0^{\beta_3}}{\beta_3!}
    \n{\partial_{x_3}^{\beta_3} v_{0,\rm h}}_{\fB^{s,\sigma}_p},\\
    &
    \begin{aligned}
    \n{v_{\rm h}}_{\mathcal{H}^{s,\sigma;M}_{p;r;g}(0,T)}
    :={}
    &
    \sum_{
    \substack{0 \leq \alpha \leq M \\ \beta \in (\mathbb{N}\cup \{ 0 \})^3}
    }
    \n{v^{\alpha,\beta}_{{\rm h};r;g}}_{{L^{\infty}}(0,T;\fB^{s,\sigma}_p) \cap L^1(0,T;\fB^{2+s,\sigma}_p)}
    \\
    &
    +
    \sum_{
    \substack{0 \leq \alpha \leq M \\ \beta \in (\mathbb{N}\cup \{ 0 \})^3}
    }
    \beta_3
    \n{ v^{\alpha,\beta}_{{\rm h};r;g}}
    _{L^1_{g'}(0,T;\fB^{s,\sigma}_p)}
    , 
    \end{aligned}
    \\
    &
    \n{v_3}_{\mathcal{V}^{s,\sigma;M}_{p;r}(0,T)}
    :=
    \sum_{
    \substack{0 \leq \alpha \leq M \\ \betah \in (\mathbb{N}\cup \{ 0 \})^2}
    }  
    \n{v^{\alpha,\betah}_{3;r}}_{{L^{\infty}}(0,T;\fB^{s,\sigma}_p)\cap L^1(0,T;\fB^{2+s,\sigma}_p)} 
\end{align} 
Moreover, we set 
\begin{align}
& 
\n{v_{\rm h}}_{\mathcal{H}^M_{p,\theta;r;g}(0,T)}
:= 
\sum_{(s,\sigma)\in \Lambda^{\rm h}_{p,\theta}}
\n{v_{\rm h}}_{\mathcal{H}^{s,\sigma;M}_{p;r;g}(0,T)}, \\
& 
\n{v_3}_{\mathcal{V}^{M}_{p,\theta;r}(0,T)}
:= 
\sum_{(s,\sigma)\in \Lambda^{\rm v}_{p,\theta}}
\n{v_3}_{\mathcal{V}^{s,\sigma;M}_{p;r}(0,T)} . 
\end{align}
Note that there holds
\begin{align}
    &
    \n{v_{0,\rm h}}_{\mathcal{D}_{p,\theta}^{\rm h}}
    = 
    \sum_{(s,\sigma)\in \Lambda^{\rm h}_{p,\theta}}
    \n{v_{0,\rm h}}_{\widetilde{\mathcal{D}}_{p}^{s,\sigma} }
    =
    \sum_{\beta_3 =0}^{\infty}
    \frac{\rho_0^{\beta_3}}{\beta_3!}
    \n{\partial_{x_3}^{\beta_3}v_{0,\rm h}}_{D_{p,\theta}^{\rm h}},\\
    &
    \n{v_{0,3}}_{D_{p,\theta}^{\rm v}}
    = 
    \sum_{(s,\sigma)\in \Lambda^{\rm v}_{p,\theta}}
    \n{v_{0,3}}_{\fB_{p}^{s,\sigma} },\\
    & 
    \n{v_{\rm h}}_{\mathcal{H}^{M}_{p,\theta;r;\log 2}(0,\infty)}
    =
    \sum_{
    \substack{0 \leq \alpha \leq M \\ \beta \in (\mathbb{N}\cup \{ 0 \})^3}
    }
    \frac{
    r^{\alpha+|\betah|}(\rho_0/2)^{\beta_3}  
    }
    {
    \alpha!\beta!
    }
    \n{
    t^{\alpha+\frac{|\beta_{\rm h}|}{2}}
    \partial_{t}^{\alpha}\partial_{x}^{\beta}
    u_{\rm h} 
    }_{S^{\rm h}_{p,\theta}},
    \\
    & 
    \n{v_3}_{\mathcal{V}^{M}_{p,\theta;r}(0,\infty)}
    =
    \sum_{
    \substack{0 \leq \alpha \leq M \\ \betah \in (\mathbb{N}\cup \{ 0 \})^2}
    }
    \frac{r^{\alpha+|\betah|}}{\alpha!\betah!}
    \n{
    t^{\alpha+\frac{|\beta_{\rm h}|}{2}}\partial_{t}^{\alpha}\partial_{\xh}^{\betah}
    u_3
    }_{S^{\rm v}_{p,\theta}}.
\end{align}
\subsection{Analytic estimates for the linear solutions}
Here we focus on the linearized system given as follows:
\begin{align}\label{eq:L}
    \begin{cases}
        \partial_t u - \Deltah u = F, \quad & t>0,x \in \R^3,\\
        u(0,x)=u_0(x), \quad & x \in \R^3,
    \end{cases}
\end{align}
where $u=u(t,x):[0,\infty) \times \R^3 \to \R^3$ is unknown velocity field and $F=F(t,x):[0,\infty) \times \R^3 \to \R^3$ and $u_0=u_0(x):\R^3 \to \R^3$ is the given external force and initial data, respectively.
\begin{prop}\label{prop:anal-lin}
    There exists an absolute positive constant $C$ such that for any $1 \leq p \leq \infty$, $s,\sigma \in \R$, $0<r\leq 1$, $T>0$, 
    and 
    a non-negative and non-decreasing function $g \in C^1([0,T))$ with $g(0)=0$,
    the solution $u$ to \eqref{eq:ANS} with the initial data $u_0=(u_{0,{\rm h}},u_{0,3}) \in \widetilde{\mathcal{D}}_p^{s,\sigma}(\R^3) \times \fB_p^{s,\sigma}(\R^3)$ satisfies
    \begin{align}
        &
        \begin{aligned}
        \n{u_{\rm h}}_{\mathcal{H}_{p;r;g}^{s,\sigma;M}(0,T)}
        \leq 
        C
        \n{u_{0,{\rm h}}}_{\widetilde{\mathcal{D}}_p^{s,\sigma}}
        &+
        Cr
        \n{u_{\rm h}}_{\mathcal{H}_{p;r;g}^{s,\sigma;M}(0,T)}
        \\
        &+
        C
        \sum_{
        \substack{0 \leq \alpha \leq M \\ \beta \in (\mathbb{N}\cup \{ 0 \})^3}
        }
        \n{F^{\alpha,\beta}_{{\rm h};r;g}}_{L^1(0,T;\fB^{s,\sigma}_p)},
        \end{aligned}
        \\
        &
        \begin{aligned}
        \n{u_3}_{\mathcal{V}_{p;r}^{s,\sigma;M}(0,T)}
        \leq 
        C
        \n{u_{0,3}}_{\fB_p^{s,\sigma}}
        &
        +
        Cr
        \n{u_3}_{\mathcal{V}_{p;r}^{s,\sigma;M}(0,T)}
        \\
        &
        +
        C
        \sum_{
        \substack{0 \leq \alpha \leq M \\ \betah \in (\mathbb{N}\cup \{ 0 \})^2}
        }
        \n{F^{\alpha,\betah}_{3;r}}_{L^1(0,T;\fB^{s,\sigma}_p)}.
        \end{aligned}
    \end{align}
\end{prop}
To show Proposition \ref{prop:anal-lin}, we prepare the following lemma.
\begin{lemm}\label{lemm:max-heat}
Let $1\leq p\leq \infty$, $s,\sigma\in \R$, $0<T\leq \infty$, $\beta \geq 0$. Consider the anisotropic heat equation with a friction term
\begin{align}
        \begin{cases}
            \partial_t v + \beta g'(t)v - \Deltah v = G, & t>0,x \in \mathbb{R}^3,\\
            v(0,x) = v_0(x), & x \in \mathbb{R}^3,
        \end{cases}
\end{align}
where $v_0\in \fB^{s,\sigma}_{p}(\R^3)$, a non-negative and non-decreasing function $g \in C^1([0,T))$ with $g(0)=0$, and $G \in L^1(0,T;\fB^{s,\sigma }_{p}(\R^3))$ are given. 
Then, there exists an absolute positive constant $C$ such that 
\begin{align}
    \| v \|_{ 
    L^{\infty}(0,T;\fB^{s ,\sigma}_{p})
    \cap 
    L^1(0,T;\fB_p^{s+2,\sigma})
    }
    +
    \beta
    \n{v}_{{L^1_{g'}}(0,T;\fB_p^{s,\sigma})}
    \leq
    C
    \n{v_0}_{\fB_p^{s,\sigma}}
    +
    C
    \n{G}_{L^1(0,T;\fB_p^{s,\sigma})}.
\end{align}
\end{lemm}
\begin{proof}
We easily see that the solution formula of $v$ is given by 
\begin{align}\label{eq:sol-u}
v(t)= e^{ -\beta  g(t)  } e^{t \Deltah} v_0
+\int_0^t  e^{ (t-\tau) \Deltah} 
e^{ -\beta (g(t)-g(\tau))  } G(\tau) d\tau. 
\end{align}
Recalling that $g'$ is non-negative, \eqref{eq:sol-u} and the same arguments as \cite{Bah-Che-Dan-11}*{Lemma 2.4} imply that for any $1\leq r \leq \infty$ 
\begin{align}\label{eq:sol-u-1}
    \begin{split}
    \n{ \Deltahh_k \Deltav_j v(t) }_{ L^r (0,T; L^p) } 
    & \leq C 
    \n{ \Deltahh_k \Deltav_j  e^{t \Deltah} v_0 }_{ L^r (0,T; L^p) } \\
    & \quad 
    +  
    C \n{  
    \int_0^t \n{\Deltahh_k \Deltav_je^{ (t-\tau) \Deltah} G(\tau)}_{L^p} d\tau
    }_{ L^r (0,T) } \\
    & \leq C 
    \n{ e^{-c2^{2k}t}{\Deltahh_k \Deltav_j  v_0}_L^p }_{ L^r (0,T) } \\
    & \quad 
    +  
    C \n{  
    \int_0^t e^{-c2^{2k}(t-\tau)}\n{\Deltahh_k \Deltav_j G(\tau)}_{L^p} d\tau
    }_{ L^r (0,T) } \\
    & \leq  C 2^{-\f{2}{r} k } \sp{\n{  \Deltahh_k \Deltav_j v_0 }_{L^p}+\n{ \Deltahh_k \Deltav_j G  }_{ L^1 (0,T; L^p) }}.  
    \end{split}
\end{align}
Taking summation of \eqref{eq:sol-u-1} over $j,k\in \Z$ with the weight $2^{sk} 2^{\sigma j}$ and employing the above inequalities, we obtain by choosing $r=1,\infty$ that 
\begin{align}
     \| v \|_{ 
    {L^{\infty}}(0,T;\fB^{s ,\sigma}_{p})    
    \cap L^1(0,T;\fB_p^{s+2,\sigma})}
    \leq
    C
    \n{v_0}_{\fB_p^{s,\sigma}}
    +
    C
    \n{G}_{L^1(0,T;\fB_p^{s,\sigma})}.
\end{align}
To proceed, we notice that \eqref{eq:sol-u} also gives
\begin{align}
\n{ \Deltahh_k \Deltav_j v (t)}_{ L^p } 
\leq C e^{ - \beta g(t)  }  \n{ \Deltahh_k \Deltav_j v_0 }_{L^p} 
+
C
\int_0^t e^{ -\beta (g(t)-g(\tau))  } 
\n{  \Deltahh_k \Deltav_j  G(\tau)  }_{L^p}  d\tau. 
\end{align}
Hence, a straightforward calculation yields
\begin{align}
&
\beta \int_0^T g'(t) \n{ \Deltahh_k \Deltav_j v (t)}_{ L^p } dt
\\
&\quad
 \leq 
C \beta \int_0^T g'(t) e^{ -\beta (g(t)-g(\tau))  }  \n{ \Deltahh_k \Deltav_j v_0 }_{L^p} dt \\
& \qquad 
+ C \beta \int_0^T \int_0^t g'(t)   e^{ -(g(t)-g(\tau))  }
\n{  \Deltahh_k \Deltav_j  G(\tau)  }_{L^p}  d\tau dt \\
&\quad =
C \int_0^T  \frac{d}{dt}\sp{-e^{ -(g(t)-g(\tau))  }} dt \n{ \Deltahh_k \Deltav_j v_0 }_{L^p}  \\
& \qquad 
+ C \int_0^T \int_\tau^T \frac{d}{dt}\sp{-e^{ -(g(t)-g(\tau))  }} dt
\n{  \Deltahh_k \Deltav_j  G(\tau)  }_{L^p}  d\tau  \\
&\quad \leq 
C \n{ \Deltahh_k \Deltav_j v_0 }_{L^p}  
+
C \int_0^T \n{  \Deltahh_k \Deltav_j  G(\tau)  }_{L^p}  d\tau.
\end{align}
Then, we sum up over $j,k\in \Z$ with the weight $2^{sk} 2^{\sigma j}$ to arrive at 
\begin{align}
    \beta 
    \n{v}_{{L^1_{g'}}(0,T;\fB_p^{s,\sigma})}
    \leq
    C
    \n{v_0}_{\fB_p^{s,\sigma}}
    +
    C
    \n{G}_{L^1(0,T;\fB_p^{s,\sigma})},
\end{align}
which completes the proof. 
\end{proof}
Now, making use of the above lemma, we show Proposition \ref{prop:anal-lin}.
\begin{proof}[Proof of Proposition \ref{prop:anal-lin}]
We only focus on the estimate of $u_{\rm h}$
as the estimate for $u_3$ follows from the similar simpler argument as for $u_{\rm h}$.
A direct calculation shows that
\begin{align}
  \p_t \p_{t,x} ^{\alpha,\beta} \uh 
  &
  - \Deltah \p_{t,x} ^{\alpha,\beta} \uh
  + \p_{t,x} ^{\alpha,\beta} \sp{  \sp{\uh\cdot \nablah} \uh }=\p_{t,x} ^{\alpha,\beta} F_{\rm h}. 
\end{align}
Then, we see that
\begin{align}
    \p_t   u^{\alpha,\beta}_{{\rm h};r;g}
    +  
    \lambda \beta_3  g'(t)  u^{\alpha,\beta}_{{\rm h};r;g}
    -
    \Deltah   u^{\alpha,\beta}_{{\rm h};r;g} 
    = 
    \sp{ \alpha +\f{|\betah|}{2} } 
    t^{-1} u^{\alpha,\beta}_{{\rm h};r;g}
    +
    F_{{\rm h};r;g}^{\alpha,\beta}.
\end{align}
For any $(s,\sigma)\in \Lambda_{p,\theta}^{\rm h}$ and $p,\theta$ satisfying \eqref{p-theta}, it follows from Lemma \ref{lemm:max-heat} that 
\begin{align}\label{non-except:1}
    &
    \n{
    u^{\alpha,\beta}_{{\rm h};r;g} }_{L^{\infty}(0,T;\fB^{s,\sigma}_p) \cap L^1(0,T;\fB^{2+s,\sigma}_p)}  
    +
    \beta_3
    \n{
    u^{\alpha,\beta}_{{\rm h};r;g}}_{L^1_{g'}(0,T;\fB^{s,\sigma}_p)} 
    \\
    &\quad
    \leq  
    C \n{ 
    u^{\beta_3}_{0,\rm h}
    }_{\fB^{s,\sigma}_p}  
    + 
    C 
    \n{ 
    F_{{\rm h};r;g}^{\alpha,\beta}
    }_{L^1(0,T;\fB^{s,\sigma}_p)}
    + 
    C
    \sp{ \alpha +\f{|\betah|}{2} }
    \n{
    t^{-1} u^{\alpha,\beta}_{{\rm h};r;g}
    }_{L^1(0,T;\fB^{s,\sigma}_p)}.
\end{align}
Summing up \eqref{non-except:1} both sides for all $0\leq \alpha \leq M$ and $\beta \in (\mathbb{N}\cup \{ 0 \})^3$ with $(\alpha,|\betah|)\neq(0,1)$, we deduce  
\begin{align}
    &
    \sum_{
    \substack{0 \leq \alpha \leq M \\ \beta \in (\mathbb{N}\cup \{ 0 \})^3 \\
    (\alpha, |\betah|) \neq (0,1)
    }
    }  
    \sp{
    \n{
    u^{\alpha,\beta}_{{\rm h};r;g} }_{L^{\infty}(0,T;\fB^{s,\sigma}_p) \cap L^1(0,T;\fB^{2+s,\sigma}_p)}  
     +
    \beta_3
    \n{
    u^{\alpha,\beta}_{{\rm h};r;g}}_{L^1_{\lambda  g'}(0,T;\fB^{s,\sigma}_p)}
    }\\
    & \quad
    \leq  
    C \n{ 
    u_{0,\rm h}
    }_{\widetilde{\mathcal{D}}^{s,\sigma}_{p,\theta}}  
    + 
    C 
    \sum_{
    \substack{0 \leq \alpha \leq M \\ \beta \in (\mathbb{N}\cup \{ 0 \})^3 
    }
    }
    \n{ 
    F_{{\rm h};r;g}^{\alpha,\beta}
    }_{L^1(0,T;\fB^{s,\sigma}_p)} \label{eq:v-h-2}
    \\
    & \qquad 
    +
    C
    \sum_{
    \substack{0 \leq \alpha \leq M \\ \beta \in (\mathbb{N}\cup \{ 0 \})^3 \\
    (\alpha, |\betah|) \neq (0,1)
    }
    }
    \sp{ \alpha +\f{|\betah|}{2} }
    \n{
    t^{-1} u^{\alpha,\beta}_{{\rm h};r;g}
    }_{L^1(0,T;\fB^{s,\sigma}_p)}.
\end{align}    
For the last term of \eqref{eq:v-h-2}, it holds
\begin{align}
    &
    \sum_{
    \substack{0 \leq \alpha \leq M \\ \beta \in (\mathbb{N} \cup \{ 0\})^3 \\
    (\alpha, |\betah|) \neq (0,1)
    }   
    }
    \sp{ \alpha +\f{|\betah|}{2} }
    \n{
    t^{-1} u^{\alpha,\beta}_{{\rm h};r;g}
    }_{L^1(0,T;\fB^{s,\sigma}_p)}  \\
    & 
    \quad 
    =
    \sum_{\substack{1 \leq \alpha \leq M \\ \beta \in (\mathbb{N} \cup \{ 0\})^3 
    }}
    \alpha 
    \n{
    t^{-1} u^{\alpha,\beta}_{{\rm h};r;g}
    }_{L^1(0,T;\fB^{s,\sigma}_p)}  \\
    &
    \qquad
    +
    \sum_{\substack{0 \leq \alpha \leq M \\ \beta \in (\mathbb{N} \cup \{ 0\})^3, |\betah| \geq 1 \\
    (\alpha, |\betah|) \neq (0,1)
    }}
    \f{|\betah|}{2} 
    \n{
    t^{-1} u^{\alpha,\beta}_{{\rm h};r;g} 
    }_{L^1(0,T;\fB^{s,\sigma}_p)}  \\
    &
    \quad 
    = 
    \sum_{\substack{1 \leq \alpha \leq M \\ \beta \in (\mathbb{N} \cup \{ 0\})^3 
    }}
    \frac{
    r^{\alpha+|\betah|}  
      \rho_0^{\beta_3} 
    }
    {
    (\alpha-1)!\beta!
    }
    \n{
    t^{\alpha-1+\f{|\betah|}{2}  } 
    \p_t 
    \sp{ 
    \partial_{t}^{\alpha-1}  
    \partial_{x}^{\beta} 
    u_{\rm h} 
    }
    e^{ -  \beta_3 g(t)  } 
    }_{L^1(0,T;\fB^{s,\sigma}_p)} \label{eq:t^-1} \\
    & \qquad 
    +
    \sum_{\substack{0 \leq \alpha \leq M \\ \beta \in (\mathbb{N} \cup \{ 0\})^3, \\ |\betah| \geq 2 
    }}
    \f{|\betah|}{2}
    \frac{
    r^{\alpha+|\betah|}  
    \rho_0^{\beta_3}  
    }
    {
    \alpha!\beta!
    } 
    \n{
    t^{\alpha+ \f{|\betah|-2}{2}} 
    \p_{\xh}^{ \gamma_{\rm h} }
    \sp{
    \partial_{t}^{\alpha}  
    \partial_{\xh}^{ (\betah -\gamma_{\rm h},\beta_3)}
    u_{\rm h} 
    }
    e^{ - \beta_3 g(t)  }   
    }_{L^1(0,T;\fB^{s,\sigma}_p)} \\
    & \qquad 
    +
    \frac{1}{2}
    \sum_{\substack{  1 \leq \alpha \leq M \\ \beta \in (\mathbb{N} \cup \{ 0\})^3, |\betah| = 1 
    }}
    \n{
    t^{-1} u^{\alpha,\beta}_{{\rm h};r;g}
    }_{L^1(0,T;\fB^{s,\sigma}_p)} \\
    &\quad =: \mathcal{U}_{1}^{s,\sigma;M} + \mathcal{U}_{2}^{s,\sigma;M} + \mathcal{U}_{3}^{s,\sigma;M}.
\end{align}
For the estimate of $\mathcal{U}_{1}^{s,\sigma;M}$, a straightforward calculation via the equation of $\uh$ shows that 
\begin{align}
    \frac{
    r^{\alpha+|\betah|} 
      \rho_0^{\beta_3}     
    }
    {
    (\alpha-1)!\beta!
    } 
    t^{\alpha-1+\f{|\betah|}{2}  }
    \p_t 
    \sp{ 
    \partial_{t}^{\alpha-1}  
    \partial_{x}^{\beta}
    u_{\rm h} 
    }
    e^{ - \beta_3 g(t)  } 
    = 
    r 
    \Deltah u_{{\rm h}; r; g'}^{\alpha-1,\beta}
    +
    r 
    F^{\alpha-1,\beta}_{{\rm h};r;g}. 
\end{align}
Hence, 
\begin{align}
    \mathcal{U}_{1}^{s,\sigma;M}
    & 
    =
    \sum_{\substack{1 \leq \alpha \leq M \\ \beta \in (\mathbb{N} \cup \{ 0\})^3}}
    \frac{
    r^{\alpha+|\betah|}
     \rho_0^{\beta_3}     
    }
    {(\alpha-1)!\beta!}
    \n{
    t^{\alpha-1+\f{|\betah|}{2}  }
    \p_t 
    \sp{ 
    \partial_{t}^{\alpha-1}  
    \partial_{x}^{\beta}  
    u_{{\rm h}} 
    }
    e^{ -  \beta_3 g(t)  } 
    }_{L^1(0,T;\fB^{s,\sigma}_p)} \\
    & 
    \leq 
    C r
    \sum_{\substack{1 \leq \alpha \leq M \\ \beta \in (\mathbb{N} \cup \{ 0\})^3}}
    \n{
    u_{{\rm h};r;g}^{\alpha-1,\beta}
    }_{L^1(0,T;\fB^{2+s,\sigma}_p)}
    +
    r\sum_{\substack{1 \leq \alpha \leq M \\ \beta \in (\mathbb{N} \cup \{ 0\})^3}}
    \n{ F^{\alpha-1,\beta}_{{\rm h};r;g} 
    }_{L^1(0,T;\fB^{s,\sigma}_p)} \\
    & 
    \leq 
    C r
    \n{u_{\rm h}}_{\mathcal{H}^{s,\sigma;M}_{p;r;g}(0,T)}
    +
    r\sum_{\substack{1 \leq \alpha \leq M \\ \beta \in (\mathbb{N} \cup \{ 0\})^3}}
    \n{ F^{\alpha-1,\beta}_{{\rm h};r;g} 
    }_{L^1(0,T;\fB^{s,\sigma}_p)} . 
\end{align}
Next, we consider the estimate for $\mathcal{U}_2^{s,\sigma}$.
Here, for the case of $|\betah| \geq 2$, we choose $\gamma_{\rm h}=\gamma_{\rm h}(\betah) \in (\mathbb{N}\cup \{ 0 \})^2$ so that
\begin{align}
    \gamma_{\rm h}
    :=
    \begin{cases}
        (0,2) \quad & {\rm if}\quad \beta_1=0, \\
        (1,1) \quad & {\rm if}\quad \beta_1\beta_2 \neq 0, \\
        (2,0) \quad & {\rm if}\quad \beta_2=0.
    \end{cases}
\end{align}
We easily see that $|\gamma_{\rm h} |=2$, $\betah - \gamma_{\rm h} \geq 0$, and $|\betah|/\betah! \leq C/(\betah-\gamma_{\rm h})!$.
We then have
\begin{align}
    \mathcal{U}_{2}^{s,\sigma;M}
    \leq
    C
    r
    \sum_{\substack{0 \leq \alpha \leq M \\ \beta \in (\mathbb{N} \cup \{ 0\})^3, |\betah| \geq 2}}
    \n{
    u_{{\rm h};r;g}^{ \alpha, (\betah-\gamma_{\rm h},\beta_3)    }
    }_{L^1(0,T;\fB^{s+2,\sigma}_p)} 
    \leq 
    C r \n{u_{\rm h}}_{\mathcal{H}^{s,\sigma;M}_{p;r;g}(0,T)}.
\end{align}
Observe that $\mathcal{U}_{3}^{s,\sigma;M}$ is estimated in the same way as $\mathcal{U}_{1}^{s,\sigma;M}$. Hence, we get
\begin{align}
\mathcal{U}_{3}^{s,\sigma;M}
\leq 
C r 
\n{\uh}_{\mathcal{H}^{s,\sigma;M}_{p;r;g}}
+
r\sum_{\substack{1 \leq \alpha \leq M \\ \beta \in (\mathbb{N} \cup \{ 0\})^3}}
\n{ F^{\alpha-1,\beta}_{{\rm h};r;g} 
}_{L^1(0,T;\fB^{s,\sigma}_p)} . 
\end{align}
Thus, combining the estimates of $\mathcal{U}_{\ell}^{s,\sigma;M}$ ($\ell=1,2,3$), we have
\begin{align} 
&
\sum_{
\substack{0 \leq \alpha \leq M \\ \beta \in (\mathbb{N} \cup \{ 0\})^3 \\
(\alpha, |\betah|) \neq (0,1)
}   
}
\sp{ \alpha +\f{|\betah|}{2} }
\n{
t^{-1} u^{\alpha,\beta}_{{\rm h};r;g}
}_{L^1(0,T;\fB^{s,\sigma}_p)} 
\\
&\qquad
\leq 
C r 
\n{\uh}_{\mathcal{H}^{s,\sigma;M}_{p;r;g}}
+
r\sum_{\substack{1 \leq \alpha \leq M \\ \beta \in (\mathbb{N} \cup \{ 0\})^3}}
\n{ F^{\alpha-1,\beta}_{{\rm h};r;g} 
}_{L^1(0,T;\fB^{s,\sigma}_p)} . \label{non-except:2}
\end{align}
Hence, combining \eqref{non-except:1}, \eqref{non-except:2}, we have 
\begin{align}
    &
    \sum_{
    \substack{0 \leq \alpha \leq M \\ \beta \in (\mathbb{N}\cup \{ 0 \})^3 \\
    (\alpha, |\betah|) \neq (0,1)
    }
    }  
    \sp{
    \n{
    u^{\alpha,\beta}_{{\rm h};r;g} }_{L^{\infty}(0,T;\fB^{s,\sigma}_p) \cap L^1(0,T;\fB^{2+s,\sigma}_p)}  
    +
    \beta_3
    \n{
    u^{\alpha,\beta}_{{\rm h};r;g}}_{L^1_{g'}(0,T;\fB^{s,\sigma}_p)}
    }\\
    & \quad
    \leq  
    C \n{ 
    u_{0,\rm h}
    }_{\widetilde{\mathcal{D}}^{s,\sigma}_{p,\theta}}  
    + 
    C r 
    \n{\uh}_{\mathcal{H}^{s,\sigma;M}_{p;r;g}}
    +
    C
    r\sum_{\substack{0 \leq \alpha \leq M \\ \beta \in (\mathbb{N} \cup \{ 0\})^3}}
    \n{ F^{\alpha,\beta}_{{\rm h};r;g} 
    }_{L^1(0,T;\fB^{s,\sigma}_p)} . \label{non-except:3}
\end{align}

Next, we consider the left-hand side of \eqref{non-except:1} when $(\alpha,|\betah|)=(0,1)$. 
We recall the integral equation of $\uh$: 
\begin{align}
   \uh(t)= e^{t \Deltah} u_{0,\rm h}-
   \int_0^t e^{(t-\tau) \Deltah} 
    F_{\rm h}(\tau) d\tau.
\end{align}
By setting  
\begin{align}
    u^{\beta_3}_{{\rm h};g}(t)
    := 
    \f{  
    \rho_0^{\beta_3}     
    }
    {\beta_3!} \p_{x_3}^{\beta_3} \uh(t) 
    e^{ -  {  \lambda} \beta_3 g(t)  },  
    \quad
    F_{{\rm h};g}^{\beta_3}(\tau)
    :=
    \f{
    \rho_0^{\beta_3}     
    }
    {\beta_3!}  
    \p_{x_3}^{\beta_3} F_{\rm h}(\tau)
    e^{ -  {  \lambda} \beta_3 g(\tau)  }   ,
\end{align}
we see that $u^{\beta_3}_{{\rm h};g}$ solves the integral equation 
\begin{align}
    u^{\beta_3}_{{\rm h};g}(t)
    =
    e^{ - \beta_3 g (t)  } 
    e^{t \Deltah} 
    u_{0,\rm h}^{\beta_3} 
    -
    \int_0^t  
    e^{ -   \beta_3 \sp{ g(t)-g(\tau) } }  e^{(t-\tau) \Deltah}  F_{{\rm h};g}^{\beta_3}(\tau) d\tau . 
\end{align}
It follows that 
\begin{align}
    \n{ \Deltahh_k\Deltav_j  u^{\beta_3}_{{\rm h};g}(t)  }_{L^p}
    & \leq C 
    e^{ - \beta_3 g(t)  } 
    e^{-c2^{2k} t }  
    \n{  \Deltahh_k\Deltav_j  u_{0,\rm h}^{\beta_3}          }_{L^p} \\
    & \quad 
    + 
    C 
    \int_0^t   e^{-c2^{2k} (t-\tau) }  
    e^{ -  {  \lambda} \beta_3  \sp{ g(t)-g(\tau)} } 
    \n{  \Deltahh_k\Deltav_j   F_{{\rm h};g}^{\beta_3}(\tau)  }_{L^p}  d\tau.
\end{align}
Thus, for any $\betah \in (\mathbb{N}\cup \{ 0 \})^2$ with $|\betah|=1$, the above inequality implies 
\begin{align}
& t^{\f{1}{2}}   
\n{ \Deltahh_k\Deltav_j  { \p_{\xh}^{\betah} u^{\beta_3}_{{\rm h};g} } (t)  }_{L^p} 
\leq C  t^{\f{1}{2}} 2^k \n{ \Deltahh_k\Deltav_j  u^{\beta_3}_{{\rm h};g}(t)  }_{L^p} \\
& \quad 
\leq 
C  t^{\f{1}{2}} 2^k  e^{-c2^{2k} t }  
e^{ -   \beta_3 g (t)  }    
\n{  \Deltahh_k\Deltav_j  u_{0,\rm h}^{\beta_3}          }_{L^p} \\
& \qquad 
+ 
C  t^{\f{1}{2}} 2^k  e^{-c2^{2k} t }  
\int_0^{\f{t}{2}} 
e^{ - \beta_3  \sp{ g(t)-g(\tau) }  }  
\n{  \Deltahh_k\Deltav_j   F_{{\rm h};g}^{\beta_3}(\tau)  }_{L^p}  d\tau 
\\
& \qquad 
+ 
C  \int^t_{\f{t}{2}} e^{-c2^{2k} (t-\tau) }  
e^{ - \beta_3   \sp{ g(t)-g(\tau)}  }    
\tau^{\f{1}{2}} 2^k \n{  \Deltahh_k\Deltav_j   F_{{\rm h};g}^{\beta_3}(\tau)  }_{L^p}  d\tau \label{eq:R2}\\ 
& \quad 
\leq 
C  e^{-c2^{2k} t } \n{  \Deltahh_k\Deltav_j  u_{0,\rm h}^{\beta_3}          }_{L^p} 
+ C  e^{-c2^{2k} t } \int_0^t \n{  \Deltahh_k\Deltav_j   F_{{\rm h};g}^{\beta_3}(\tau)  }_{L^p}  d\tau \\  
& \qquad 
+ 
C \int_0^t e^{-c2^{2k} (t-\tau) } \tau^{\f{1}{2}} 2^k \n{  \Deltahh_k\Deltav_j  F_{{\rm h};g}^{\beta_3}(\tau)  }_{L^p}  d\tau.   
\end{align}
For any $(s,\sigma)\in \Lambda_{p,\theta}^{\rm h}$ and $p,\theta$ satisfying \eqref{p-theta} and $q\in \{ 1,\infty\}$, we take $L^q(0,T)$-norm of the above inequality both sides and sum over $j,k\in \mathbb{Z}$ with the weight $2^{(s+\f{2}{q}) k} 2^{\sigma j}$ to conclude that 
\begin{align}\label{eq:R2-1}
    & \n{
    t^{\f{1}{2}}  \p_{\xh}^{\betah} u^{\beta_3}_{{\rm h};g}
    }_{L^{\infty}(0,T;\fB^{s,\sigma}_p) \cap L^1(0,T;\fB^{ {    2+s} ,\sigma}_p)}  \\  
    & \quad 
    \leq 
    C 
    \sp{
    \n{u^{\beta_3}_{0,\rm h} }_{\fB^{s,\sigma}_p}
    + \n{F_{{\rm h};g}^{\beta_3}}_{L^1(0,T;\fB^{   s ,\sigma}_p)} 
    + \n{  t^{\f{1}{2}} \nablah  F_{{\rm h};g}^{\beta_3} }_{L^1(0,T;\fB^{   s ,\sigma}_p)}
    }. 
\end{align}
Moreover, we notice that \eqref{eq:R2} also implies
\begin{align}
    & t^{\f{1}{2}}   
    \n{ \Deltahh_k\Deltav_j  \sp{ \p_{\xh}^{\betah} u^{\beta_3}_{{\rm h};g} } (t)  }_{L^p}  \\ 
    & \quad 
    \leq 
    C  e^{ -   \beta_3 g(t)  }  
    \n{  \Deltahh_k\Deltav_j { u^{\beta_3}_{0,\rm h}  }        }_{L^p}
    + 
    C \int_0^{\f{t}{2}}  
    e^{ -  \beta_3 ( g(t)-g(\tau) )  }    
    \n{  \Deltahh_k\Deltav_j   F_{{\rm h};g}^{\beta_3}(\tau)  }_{L^p}  d\tau \\
    & \qquad 
    + 
    C  \int^t_{\f{t}{2}}   e^{ -  \beta_3 ( g(t)-g(\tau) )  }
    \tau^{\f{1}{2}} 2^k \n{  \Deltahh_k\Deltav_j   F_{{\rm h};g}^{\beta_3}(\tau)  }_{L^p}  d\tau \\
    & \quad 
    \leq 
    C  e^{ - \beta_3 g(t)  }
    \n{  \Deltahh_k\Deltav_j  u_{0,\rm h}^{\beta_3}         }_{L^p} \\
    & \qquad 
    + C \int_0^t 
    e^{ -  \beta_3 ( g(t)-g(\tau) )  }    
    \sp{  \n{  \Deltahh_k\Deltav_j   F_{{\rm h};g}^{\beta_3}(\tau)  }_{L^p} + \tau^{\f{1}{2}} 2^k \n{  \Deltahh_k\Deltav_j   F_{{\rm h};g}^{\beta_3}(\tau)  }_{L^p}  }
    d\tau .
    \end{align}
Thus, we see that
\begin{small}  
\begin{align}
    &\beta_3 
    \int_0^T g'(t) t^{\f{1}{2}}   
    \n{ \Deltahh_k\Deltav_j  \sp{ \p_{\xh}^{\betah} u^{\beta_3}_{{\rm h};g} } (t)  }_{L^p} dt \\
    & \quad 
    \leq 
    C \lambda \beta_3  \int_0^T g'(t) 
    e^{ - \beta_3 g(t)  } dt 
    \n{  \Deltahh_k\Deltav_j u_{0,\rm h}^{\beta_3}        }_{L^p} \\
    & \qquad   
    + 
    C 
    \int_0^T  \int_0^t \lambda \beta_3  g'(t) 
    e^{ -  \beta_3 ( g(t)-g(\tau) )  }    %
    \\
    &\qquad \qquad \qquad \qquad
    \sp{  \n{  \Deltahh_k\Deltav_j   F_{{\rm h};g}^{\beta_3}(\tau)  }_{L^p} + \tau^{\f{1}{2}} 2^k \n{  \Deltahh_k\Deltav_j   F_{{\rm h};g}^{\beta_3}(\tau)  }_{L^p}  }
    d\tau  dt  \\
    & \quad 
    \leq C \n{  \Deltahh_k\Deltav_j  u_{0,\rm h}^{\beta_3}          }_{L^p} 
    + C  
    \int_0^T \sp{  \n{  \Deltahh_k\Deltav_j   
    F_{{\rm h};g}^{\beta_3}(\tau)  }_{L^p} 
    + 
    \tau^{\f{1}{2}} 2^k \n{  \Deltahh_k\Deltav_j   F_{{\rm h};g}^{\beta_3}(\tau)  }_{L^p}  }
    d\tau,
\end{align}
\end{small}  
which implies
\begin{align}\label{eq:R3}
\begin{split}
&\beta_3  \n{
t^{\f{1}{2}}  \p_{\xh}^{\betah} u^{\beta_3}_{{\rm h};g}
}_{L^1_{ g' }(0,T;\fB^{s,\sigma}_p)} \\
& \quad 
\leq C  
\sp{
\n{u_{0,\rm h}^{\beta_3} }_{\fB^{s,\sigma}_p}
+ \n{F_{{\rm h};g}^{\beta_3}}_{L^1(0,T;\fB^{  s ,\sigma}_p)} 
+ \n{  t^{\f{1}{2}} \nablah F_{{\rm h};g}^{\beta_3} }_{L^1(0,T;\fB^{   s ,\sigma}_p)}
}.
\end{split}
\end{align}
Combining \eqref{eq:R2-1} and \eqref{eq:R3} and recalling that $u_{ {\rm h};r; g }^{\alpha, \beta} = r t^{\f{1}{2}}  \p_{\xh}^{\betah} u^{\beta_3}_{{\rm h};g} $ when $(\alpha,|\betah|)=(0,1)$, we obtain 
\begin{align}
    & \n{
    r t^{\f{1}{2}}  \p_{\xh}^{\betah} u^{\beta_3}_{{\rm h};g}
    }_{L^{\infty}(0,T;\fB^{s,\sigma}_p)\cap L^1(0,T;\fB^{ {    2+s} ,\sigma}_p)} 
    +
    \beta_3  \n{ r 
    t^{\f{1}{2}}  \p_{\xh}^{\betah} u^{\beta_3}_{{\rm h};g}
    }_{L^1_{ {  \lambda  g' } }(0,T;\fB^{s,\sigma}_p)} \\
    & \quad 
    \leq C  
    \sp{ r
    \n{u_{0,\rm h}^{\beta_3} }_{\fB^{s,\sigma}_p}
    +r \n{ F_{{\rm h};g}^{\beta_3}}_{L^1(0,T;\fB^{   s ,\sigma}_p)} 
    + r \n{  t^{\f{1}{2}} \nablah F_{{\rm h};g}^{\beta_3} }_{L^1(0,T;\fB^{   s ,\sigma}_p)}
    }\\ 
    & \quad 
    \leq  
    C  
    \sp{
    \n{u_{0,\rm h}^{\beta_3} }_{\fB^{s,\sigma}_p}
    + \n{F_{{\rm h};g}^{\beta_3}}_{L^1(0,T;\fB^{   s,\sigma}_p)} 
    +  \n{  rt^{\f{1}{2}} \nablah  F_{{\rm h};g}^{\beta_3} }_{L^1(0,T;\fB^{   s ,\sigma}_p)}
    }.
\end{align}
Summing up the above estimate over $|\betah|=1$ and $\beta_3 \in \mathbb{N} \cup \{0 \}$, we have
\begin{align}
    & 
    \sum_{\substack{|\betah|=1, \\ \beta_3 \in \mathbb{N} \cup \{0 \}}}
    \sp 
    {\n{
    r t^{\f{1}{2}}  \p_{\xh}^{\betah} u^{\beta_3}_{{\rm h};g}
    }_{L^{\infty}(0,T;\fB^{s,\sigma}_p)\cap L^1(0,T;\fB^{ {    2+s} ,\sigma}_p)} 
    +
    \beta_3  
    \n{ r 
    t^{\f{1}{2}}  \p_{\xh}^{\betah} u^{\beta_3}_{{\rm h};g}
    }_{L^1_{ {   g' } }(0,T;\fB^{s,\sigma}_p)}} \\
    & \quad 
    \leq C r
    \n{u_{0,\rm h} }_{\widetilde{\mathcal{D}}^{s,\sigma}_p}
    +
    Cr 
    \sum_{
    \substack{0 \leq \alpha \leq M \\ \beta \in (\mathbb{N}\cup \{ 0 \})^3 
    }
    }
    \n{ 
    F_{{\rm h};r;g}^{\alpha,\beta}
    }_{L^1(0,T;\fB^{s,\sigma}_p)}. \label{except}
\end{align}
Combining \eqref{non-except:3} and \eqref{except}, we complete the proof.
\end{proof}
\subsection{Analytic estimates for the nonlinear solutions}
Let us consider the analytic global a priori estimate of the solutions to \eqref{eq:ANS}, which plays a central role of our analysis.
\begin{prop}\label{prop}
Let $M \in \mathbb{N}$.
Assume the existence of an space-time analytic solution $u$ to the Cauchy problem \eqref{eq:ANS} on some time interval $[0,T)$, and a solution $f_M \in C^1([0,\infty))$ to the following ordinary differential equation: 
\begin{align}\label{ODE}
    \begin{dcases}
    \begin{aligned}
    f_M'(t) 
    = {}&
    \sum_{
    \substack{0 \leq \alpha \leq M \\ \beta \in (\mathbb{N}\cup \{ 0 \})^3}
    }
    \frac{r^{\alpha + |\beta_{\rm h}|}t^{\alpha + \frac{|\beta_{\rm h}|}{2}}}{\alpha!\beta!}
    \n{
    \partial_t^{\alpha}
    \partial_{x}^{\beta}
    u_{\rm h} (t)
    }_{\fB^{\frac{2}{p}+1,\frac{1}{p}}_p}
    e^{-\lambda \beta_3 f_M(t)} 
    \\
    &+
    \sum_{
    \substack{0 \leq \alpha \leq M \\ \beta_{\rm h} \in (\mathbb{N}\cup \{ 0 \})^2}
    }
    \frac{r^{\alpha + |\beta_{\rm h}|}t^{\alpha + \frac{|\beta_{\rm h}|}{2}}}{\alpha!\beta_{\rm h}!}
    \n{
    \partial_t^{\alpha}
    \partial_{x_{\rm h}}^{\beta_{\rm h}}
    u_3(t)}_{\fB^{\frac{2}{p},\frac{1}{p}}_p},
    \end{aligned}
    \\
    f_M(0) = 0.
    \end{dcases}
\end{align}
Then, for any $p$ and $\theta$ satisfying \eqref{p-theta}, there exists a constant $C_0=C_0(p,\theta) \geq 1$, independent of $u$, $T$, and $M$, such that 
\begin{align}
    &\n{u_{\rm h}}_{\mathcal{H}^M_{p,\theta;r;\lambda f_M}(0,T)}
    \leq 
    C_0 \n{ u_{0,\rm h}}_{\mathcal{D}_{p,\theta}^{\rm h} }
    + 
    \sp{\f{C_0}{\lambda}+C_0r} \n{u_{\rm h}}_{\mathcal{H}^M_{p,\theta;r;\lambda f_M}(0,T)}
    \\
    & \quad 
    +  
    \frac{C_0}{\lambda}
    \exp
    \lp{
    \lambda
    \sup_{0 \leq t <T}
    f_M(t)
    }  
    \n{u_{\rm h}}_{\mathcal{H}^M_{p,\theta;r;\lambda f_M}(0,T)}
    +
    C_0 
    \sp{\n{u_{\rm h}}_{\mathcal{H}_{p,\theta;r;\lambda f_M}^M(0,T)}}^2 \\
    & \quad 
    +  
    C_0  
    \int_0^T 
    \sum_{
    \substack{0 \leq \alpha \leq M \\ \betah \in (\mathbb{N}\cup \{ 0 \})^2}
    }
    \frac{r^{\alpha+|\betah|} t^{\alpha+\f{|\betah|}{2}} }{\alpha!\betah!}
    \n{
    \partial_{t}^{\alpha}  
    \partial_{\xh}^{\betah}  
    u_3(t)
    }_{\fB^{\f{2}{p}+1,\f{1}{p}}_p}
    \n{u_{\rm h}}_{\mathcal{H}^M_{p,\theta;r;\lambda f_M}(0,t)}  dt 
\end{align}
and 
\begin{align}
    &
    \n{u_3}_{\mathcal{V}^M_{p,\theta;r}(0,T)}
    \leq{}
    C_0   
    \n{u_{0,3}}_{D^{\rm v}_{p,\theta}} 
    +
    C_0r
    \n{u_3}_{\mathcal{V}^M_{p,\theta;r}(0,T)}
    \\
    & \quad 
    +
    C_0 
    \n{u_{\rm h}}_{\mathcal{H}^M_{p,\theta;r;\lambda f_M}(0,T)} 
    \n{u_3}_{\mathcal{V}^M_{p,\theta;r}(0,T)}
    +
    C_0 
    \sp{  \n{u_{\rm h}}_{\mathcal{H}^M_{p,\theta;r;\lambda f_M}(0,T)}}^2 
\end{align}
for all $\lambda > 0$ and $0 < r \leq 1$.
Moreover, it holds
\begin{align}
    \sup_{0 \leq t < T} f_M(t)
    \leq
    C_0
    \n{u_3}_{\mathcal{V}^M_{p,\theta}(0,T)}^{1-\theta}
    \n{u_{\rm h}}_{\mathcal{H}^M_{p,\theta;r;\lambda f_M}(0,T)}^{\theta}
    +
    \n{u_{\rm h}}_{\mathcal{H}^M_{p,\theta;r;\lambda f_M}(0,T)}.
\end{align}
\end{prop}

\begin{proof}
We divide the proof into three parts: the first one concerns the estimate for the function $f_M(t)$ and the second and third one focus on the analytic estimates for the horizontal and vertical components. 

\noindent 
{\it Step 1. The estimate for $f_M(t)$.} 
We see by the interpolation inequality that
\begin{align}
    & 
    \sum_{
    \substack{0 \leq \alpha \leq M \\ \betah \in (\mathbb{N}\cup \{ 0 \})^2}
    }
    \n{
    u^{\alpha,\betah}_{3;r} 
    }_{L^1(0,T;\fB^{\f{2}{p},\frac{1}{p}}_p)} \\
    & \quad
    \leq 
    \sum_{
    \substack{0 \leq \alpha \leq M \\ \betah \in (\mathbb{N}\cup \{ 0 \})^2}
    }
    \n{
    u^{\alpha,\betah}_{3;r} }^{1-\theta}
    _{L^1(0,T;\fB^{\f{2}{p}+\theta,\frac{1}{p}-\theta}_p)} 
    \n{
    u^{\alpha,\betah}_{3;r} }^{\theta}
    _{L^1(0,T;\fB^{-1+\f{2}{p}+\theta,\frac{1}{p}+1-\theta}_p)} 
    \\
    & \quad
    \leq 
    \sum_{
    \substack{0 \leq \alpha \leq M \\ \betah \in (\mathbb{N}\cup \{ 0 \})^2}
    }
    \n{
    u^{\alpha,\betah}_{3;r} }^{1-\theta}
    _{L^1(0,T;\fB^{\f{2}{p}+\theta,\frac{1}{p}-\theta}_p)} 
    \n{
    \p_{x_3} u^{\alpha,\betah}_{3;r} }^{\theta}
    _{L^1(0,T;\fB^{-1+\f{2}{p}+\theta,\frac{1}{p}-\theta}_p)} 
    \\
    & \quad
    =
    \sum_{
    \substack{0 \leq \alpha \leq M \\ \betah \in (\mathbb{N}\cup \{ 0 \})^2}
    }
    \n{
    u^{\alpha,\betah}_{3;r} }^{1-\theta}
    _{L^1(0,T;\fB^{\f{2}{p}+\theta,\frac{1}{p}-\theta}_p)} 
    \n{
    \divh u^{\alpha,\betah}_{{\rm h};r} }^{\theta}
    _{L^1(0,T;\fB^{-1+\f{2}{p}+\theta,\frac{1}{p}-\theta}_p)} 
    \\ 
    & \quad 
    \leq 
    C 
    \sum_{
    \substack{0 \leq \alpha \leq M \\ \betah \in (\mathbb{N}\cup \{ 0 \})^2}
    }
     \n{
    u^{\alpha,\betah}_{3;r} }^{1-\theta}
    _{L^1(0,T;\fB^{\f{2}{p}+\theta,\frac{1}{p}-\theta}_p)} 
    \n{
    u^{\alpha,\betah}_{{\rm h};r} }^{\theta}
    _{L^1(0,T;\fB^{\f{2}{p}+\theta,\frac{1}{p}-\theta}_p)} \\
    & \quad 
    \leq C 
    \n{u_3}_{\mathcal{V}^M_{p;r}(0,T)}^{1-\theta} 
    \n{u_{\rm h}}_{\mathcal{H}^M_{p;r;\lambda f_M}(0,T)}^{\theta}.
\end{align}
Using this,
we have
\begin{align}
    f_M(t)   
    ={}&
    \sum_{
    \substack{0 \leq \alpha \leq M \\ \betah \in (\mathbb{N}\cup \{ 0 \})^2}
    }
    \n{
    u^{\alpha,\betah}_{3;r} 
    }_{L^1(0,t;\fB^{\f{2}{p},\frac{1}{p}}_p)}
    +
     \sum_{
    \substack{0 \leq \alpha \leq M \\ \beta \in (\mathbb{N}\cup \{ 0 \})^3}
    }
    \n{
    u^{\alpha,\beta}_{{\rm h};r;\lambda f_M}
    }_{L^1(0,t;\fB^{\f{2}{p}+1,\frac{1}{p}}_p)} \\ 
    \leq{}&
  \sum_{
    \substack{0 \leq \alpha \leq M \\ \betah \in (\mathbb{N}\cup \{ 0 \})^2}
    }
    \n{
    u^{\alpha,\betah}_{3;r} 
    }_{L^1(0,T;\fB^{\f{2}{p},\frac{1}{p}}_p)}    
    +
    \n{u_{\rm h}}_{ \mathcal{H}^{\f{2}{p}-1,\frac{1}{p};M}_{p;r;\lambda f_M}(0,T) } \\
    \leq{}&
    C
    \n{u_3}_{\mathcal{V}^M_{p;r}(0,T)}^{1-\theta} 
    \n{u_{\rm h}}_{\mathcal{H}^M_{p;r;\lambda f_M}(0,T)}^{\theta}
    +
    \n{u_{\rm h}}_{ \mathcal{H}^{\f{2}{p}-1,\frac{1}{p};M}_{p;r;\lambda f_M}(0,T) } .
\end{align}
{\it Step 2. Analytic estimate for $u_{\rm h}$.} 
A direct calculation shows that\footnote{In what follows, we interpret that $\sum_{(\alpha,\beta) \in \varnothing} A_{\alpha,\beta} = 0$.}
\begin{align}
  \p_{t,x} ^{\alpha,\beta}
  \lp{(u \cdot \nabla)u_{\rm h} + \nablah P}
  ={}&
  \p_{t,x} ^{\alpha,\beta} \sp{  \sp{\uh\cdot \nablah} \uh }
  +
  u_3 \p_{x_3} \p_{t,x} ^{\alpha,\beta} \uh \\
  & 
  + 
  \sum_{
  \substack{ \alpha=\alpha'+\alpha'' \\ \betah=\betah'+\betah'' \\ 
  (\alpha',\betah')\neq (0,0)}
  }
  \f{ \alpha! }{ \alpha'! \alpha''!  } \f{ \betah! }{ \betah'! \betah''!  }
  \sp{
  \p_{t,\xh}^{\alpha',\betah' } u_3 
  } \p_{x_3}  
  \sp{
  \p_{t,\xh}^{\alpha'',\betah'' } \p_{x_3}^{\beta_3} \uh
  } \\
  &
  +
  \sum_{
  \substack{ \alpha=\alpha'+\alpha'' \\  \beta=\beta'+\beta'' \\ 
  \beta_3' \neq 0}
  }
  \f{ \alpha! }{ \alpha'! \alpha''!  } \f{ \beta! }{ \beta'! \beta''!  }
  \sp{
  - \divh \p_{t,x}^{ \alpha',\beta'-e_3 } \uh 
  }
  \p_{t,x}^{ \alpha'',\beta''+e_3 } \uh \\
  &
  +
  \nablah \p_{t,x} ^{\alpha,\beta} P ,
\end{align}
which implies that
\begin{align}
    \lp{(u \cdot \nabla)u + \nabla P}_{{\rm h};\lambda f_M}^{\alpha,\beta}
    ={}&
    \sum_{
    \substack{ \alpha=\alpha'+\alpha'' \\  \beta=\beta'+\beta''}
    } 
    \sp{
    u^{\alpha',\beta'}_{{\rm h};r;\lambda f_M}\cdot \nablah
    } 
    u^{\alpha'',\beta''}_{{\rm h};r;\lambda f_M}
    \\
    &
    + 
    u_{3} \p_{x_3} u^{\alpha,\beta}_{{\rm h};r;\lambda f_M} 
    +
    \sum_{
    \substack{ \alpha=\alpha'+\alpha'' \\ \betah=\betah'+\betah'' \\ 
    (\alpha',\betah')\neq (0,0)}
    }
    u^{\alpha',\betah'}_{3;r}
    \p_{x_3}   u^{\alpha'',\betah'',\beta_3 }_{{\rm h};r;\lambda f_M}
    \\
    & 
    +
    \sum_{
    \substack{ \alpha=\alpha'+\alpha'' \\  \beta=\beta'+\beta'' \\ 
    \beta_3' \neq 0}
    } 
    \f{ \beta_3''+1  }{ \beta_3'  }    
    \sp{
    - \divh   u^{\alpha',\beta'-e_3}_{{\rm h};r;\lambda f_M}
    } 
    u^{\alpha'',\beta''+e_3}_{{\rm h};r;\lambda f_M}
    +\nablah P^{\alpha,\beta}_{r;\lambda f_M}.
\end{align}
Thus,
it follows from Proposition \ref{prop:anal-lin} that for any $(s,\sigma)\in \Lambda_{p,\theta}^{\rm h}$,
\begin{align}
    \n{u_{\rm h}}_{\mathcal{H}^{s,\sigma;M}_{p;r;\lambda f_M}(0,T)}
    \leq {}&
    C \n{ 
    u_{0,\rm h}
    }_{D^{s,\sigma}_{p,\theta}} 
    +   
    C r 
    \n{u_{\rm h}}_{\mathcal{H}^{s,\sigma;M}_{p;r;\lambda f_M}(0,T)}
    \\
    &
    +
    C
    \n{\lp{(u \cdot \nabla)u + \nabla P}_{{\rm h};\lambda f_M}^{\alpha,\beta}}_{L^1(0,T;\fB_p^{s,\sigma})}
    \\
    \leq {}&
    C \n{ 
    u_{0,\rm h}
    }_{D^{s,\sigma}_{p,\theta}} 
    +   
    C r 
    \n{u_{\rm h}}_{\mathcal{H}^{s,\sigma;M}_{p;r;\lambda f_M}(0,T)}
    +
    C
    \sum_{m=1}^5
    \mathcal{I}^{s,\sigma;M}_{m},
\end{align}
where we have set
\begin{align}
    \mathcal{I}^{s,\sigma;M}_{1}
    &:=
    \sum_{
    \substack{0 \leq \alpha \leq M \\ \beta \in (\mathbb{N}\cup \{ 0 \})^3}
    }
    \sum_{
    \substack{ \alpha=\alpha'+\alpha'' \\ \beta=\beta'+\beta'' \\ 
    \beta_3' \neq 0}
    } 
    \f{ \beta_3''+1  }{ \beta_3'  } 
    \n{
    \sp{\divh   u^{\alpha',\beta'-e_3}_{{\rm h};r;\lambda f_M} }  
    u^{\alpha'',\beta''+e_3}_{{\rm h};r;\lambda f_M}
    }_{L^1(0,T;\fB^{s,\sigma}_p)}  \\
    \mathcal{I}^{s,\sigma;M}_2
    &:=
    \sum_{
    \substack{0 \leq \alpha \leq M \\ \beta \in (\mathbb{N}\cup \{ 0 \})^3}
    }
    \n{
    u_{3} \p_{x_3} u^{\alpha,\beta}_{{\rm h};r;\lambda f_M}
    }_{L^1(0,T;\fB^{s,\sigma}_p)} 
    \\
    \mathcal{I}^{s,\sigma;M}_{3}
    &:=
    \sum_{
    \substack{0 \leq \alpha \leq M \\ \beta \in (\mathbb{N}\cup \{ 0 \})^3}
    }
    \sum_{
    \substack{ \alpha=\alpha'+\alpha'' \\ \beta=\beta'+\beta''}
    }
    \n{
    \sp{
    u^{\alpha',\beta'}_{{\rm h};r;\lambda f_M}\cdot \nablah
    }
    u^{\alpha'',\beta''}_{{\rm h};r;\lambda f_M}
    }_{L^1(0,T;\fB^{s,\sigma}_p)}  \\
    \mathcal{I}^{s,\sigma;M}_{4}
    &:=
    \sum_{
    \substack{0 \leq \alpha \leq M \\ \beta \in (\mathbb{N}\cup \{ 0 \})^3}
    }
    \sum_{
    \substack{ \alpha=\alpha'+\alpha'' \\
    \betah=\betah'+\betah'' \\ 
    (\alpha',\betah')\neq (0,0)}
    } 
    \n{
    u^{\alpha',\betah'}_{3;r} 
    \p_{x_3}   u^{\alpha'',(\betah'',\beta_3) }_{{\rm h};r;\lambda f_M}
    }_{L^1(0,T;\fB^{s,\sigma}_p)}   
    \\
    \mathcal{I}^{s,\sigma;M}_{5}
    &:=
    \sum_{
    \substack{0 \leq \alpha \leq M \\ \beta \in (\mathbb{N}\cup \{ 0 \})^3}
    }
    \n{ 
    \nablah P^{\alpha,\beta}_{\lambda f_M} 
    }_{L^1(0,T;\fB^{s,\sigma}_p)}.
\end{align}
We remark that $\mathcal{I}_1^{s,\sigma;M}$ and $\mathcal{I}_5^{s,\sigma;M}$ are the two most harmful terms in the nonlinear estimates.

For the estimate of $\mathcal{I}_1^{s,\sigma;M}$, it holds
\begin{align}
&
\mathcal{I}_1^{s,\sigma;M}
=
\sum_{
    \substack{0 \leq \alpha \leq M \\ \betah \in (\mathbb{N}\cup \{ 0 \})^2\\ \beta_3 \geq 1} 
    }
\sum_{
  \substack{ \alpha=\alpha'+\alpha'' \\ \beta=\beta'+\beta'' \\ 
  \beta_3' \geq 1}  
  } 
  \f{ \beta_3''+1  }{ \beta_3'  } 
\n{
\sp{\divh   u^{\alpha',\beta'-e_3}_{{\rm h};r;\lambda f_M} }  
u^{\alpha'',\beta''+e_3}_{{\rm h};r;\lambda f_M}
}_{L^1(0,T;\fB^{s,\sigma}_p)}  
\\
& \quad 
\leq C
\sum_{
    \substack{0 \leq \alpha \leq M \\ \betah \in (\mathbb{N}\cup \{ 0 \})^2\\ \beta_3 \geq 1}
    }
    \sum_{
  \substack{ \alpha=\alpha'+\alpha'' \\ \beta=\beta'+\beta'' \\ 
  \beta_3' \geq 1}  
  } 
  \f{ \beta_3''+1  }{ \beta_3'  } 
  \int_0^T 
  \n{
  u^{\alpha',\beta'-e_3}_{{\rm h};r;\lambda f_M}
  }_{ \fB^{\f{2}{p}+1,\f{1}{p}}_p }   
  \n{
  u^{\alpha'',\beta''+e_3}_{{\rm h};r;\lambda f_M}
  }_{ \fB^{s,\sigma}_p }
dt \\
& \quad 
=C
\sum_{
    \substack{0 \leq \alpha \leq M \\ \betah \in (\mathbb{N}\cup \{ 0 \})^2\\ \beta_3 \geq 1}
    }
    \sum_{
  \substack{ \alpha=\alpha'+\alpha'' \\ \beta=\beta'+\beta'' \\ 
  \beta_3'' \geq 1}  
  } 
  \f{ \beta_3''  }{ \beta_3' +1 } 
  \int_0^T 
  \n{
  u^{\alpha',\beta'}_{{\rm h};r;\lambda f_M}
  }_{ \fB^{\f{2}{p}+1,\f{1}{p}}_p } 
  \n{
  u^{\alpha'',\beta''}_{{\rm h};r;\lambda f_M}
  }_{ \fB^{s,\sigma}_p }  
dt \label{eq:B-1+1}\\
& \quad 
\leq C
 \sum_{ \substack{ 0 \leq \alpha'' \leq M \\ \betah'' \in (\mathbb{N}\cup \{ 0 \})^2, \beta''_3 \geq 1  }  } 
\beta''_3 
  \int_0^T 
  \sum_{ \substack{ 0 \leq \alpha' \leq M   \\   \beta' \in (\mathbb{N}\cup \{ 0 \})^3
}}
\n{
  u^{\alpha',\beta'}_{{\rm h};r;\lambda f_M}
  }_{ \fB^{\f{2}{p}+1,\f{1}{p}}_p } 
  \n{
  u^{\alpha'',\beta''}_{{\rm h};r;\lambda f_M}
  }_{ \fB^{s,\sigma}_p }
dt \\
& \quad 
\leq \frac{C}{\lambda}
 \sum_{ \substack{ 0 \leq \alpha'' \leq M \\ \betah'' \in (\mathbb{N}\cup \{ 0 \})^2, \beta''_3 \geq 1  }  } 
\beta''_3 
  \int_0^T 
  \lambda f_M'(t)
  \n{
  u^{\alpha'',\beta''}_{{\rm h};r;\lambda f_M}
  }_{ \fB^{s,\sigma}_p }
dt \\
& \quad 
= \frac{C}{\lambda}
\sum_{ \substack{0 \leq \alpha'' \leq M \\  \betah'' \in (\mathbb{N}\cup \{ 0 \})^2, \beta''_3 \geq 1  }  } 
\beta''_3 
\n{
u^{\alpha'',\beta''}_{{\rm h};r;\lambda f_M}
}_{L^1_{ {  \lambda  f_M' } }(0,T;\fB^{s,\sigma}_p)} 
\\
& \quad 
\leq \f{C}{\lambda} 
 \n{u_{\rm h}}_{\mathcal{H}^{s,\sigma;M}_{p;r;\lambda f_M}(0,T)}. 
\end{align}
For the estimate of $\mathcal{I}_2^{s,\sigma;M}$, we obtain from Lemma \ref{lemm:prod} and the interpolation inequality that 
\begin{align}
& 
\mathcal{I}_2^{s,\sigma;M}
=
\sum_{
    \substack{0 \leq \alpha \leq M \\ \beta \in (\mathbb{N}\cup \{ 0 \})^3}
    } 
  \sum_{
  \substack{ \alpha=\alpha'+\alpha'' \\ \beta=\beta'+\beta''}
  }
\n{
 \sp{
  u^{\alpha',\beta'}_{{\rm h};r;\lambda f_M}\cdot \nablah
  }
   u^{\alpha'',\beta''}_{{\rm h};r;\lambda f_M}
}_{L^1(0,T;\fB^{s,\sigma}_p)}  \\
& \quad 
\leq 
\sum_{
    \substack{0 \leq \alpha' \leq M \\ \beta' \in (\mathbb{N}\cup \{ 0 \})^3}
    } 
\sum_{
    \substack{0 \leq \alpha'' \leq M \\ \beta'' \in (\mathbb{N}\cup \{ 0 \})^3}
    } 
 \n{
 \sp{
  u^{\alpha',\beta'}_{{\rm h};r;\lambda f_M}\cdot \nablah
  }
   u^{\alpha'',\beta''}_{{\rm h};r;\lambda f_M}
}_{L^1(0,T;\fB^{s,\sigma}_p)}  \label{eq:B-00}\\
& \quad 
\leq 
C 
\sum_{
    \substack{0 \leq \alpha' \leq M\\ \beta' \in (\mathbb{N}\cup \{ 0 \})^3}
    } 
\sum_{
    \substack{0 \leq \alpha'' \leq M \\ \beta'' \in (\mathbb{N}\cup \{ 0 \})^3}
    } 
\n{
u^{\alpha',\beta'}_{{\rm h};r;\lambda f_M}
}_{L^2(0,T;\fB^{\f{2}{p},\f{1}{p}}_p)} 
\n{
u^{\alpha'',\beta''}_{{\rm h};r;\lambda f_M}
}_{L^2(0,T;\fB^{s+1,\sigma}_p)}  \\   
& \quad 
\leq 
C 
\n{u_{\rm h}}_{\mathcal{H}^{s,\sigma;M}_{p;r;\lambda f_M}(0,T)}
\n{u_{\rm h}}_{\mathcal{H}^{\f{2}{p}-1,\f{1}{p};M}_{p;r;\lambda f_M}(0,T)} .   
\end{align}
For the estimate of $\mathcal{I}^{s,\sigma;M}_3$, we see by 
\begin{align}
    \partial_{x_3} u^{\alpha,\beta}_{{\rm h};r;\lambda f_M}
    =
    (\beta_3+1)u^{\alpha,\beta+e_3}_{{\rm h};r;\lambda f_M}
    e^{\lambda f_M(t)}
\end{align}
that 
\begin{align}
    &
    \n{
    u_3
    \partial_{x_3}
    u^{\alpha,\beta}_{{\rm h};r;\lambda f_M}
    }_{L^1(0,T;\fB^{s,\sigma}_p)}
    \leq 
    C
    \int_0^T
    \n{u_3(t)}_{\fB^{\f{2}{p},\frac{1}{p}}_p}
    \n{\partial_{x_3}u^{\alpha,\beta}_{{\rm h};r;\lambda f_M}(t)}_{\fB^{s,\sigma}_p}   
    dt \\
    &\quad
    \leq
    C
    e^{\lambda f_M(T)}
    (\beta_3+1)
    \int_0^T
    \n{u_3(t)}_{\fB^{\f{2}{p},\frac{1}{p}}_p}
    \n{u^{\alpha,\beta+e_3}_{{\rm h};r;\lambda f_M}(t)}_{\fB^{s,\sigma}_p}   
    dt \\
    &\quad
    \leq 
    C
    \exp
    \lp{
    \lambda \sup_{0 \leq t < T} f_M(t)
    }(\beta_3+1)
    \int_0^T 
    \n{u_3(t)}_{\fB^{\f{2}{p},\frac{1}{p}}_p}
    \n{  u^{\alpha,\beta+ e_3}_{{\rm h};r;\lambda f_M}(t)}_{\fB^{s,\sigma}_p}   
    dt\\
    &\quad
    \leq 
    \frac{C}{\lambda}
    \exp
    \lp{
    \lambda \sup_{0 \leq t < T} f_M(t)
    }(\beta_3+1)
    \n{
    u^{\alpha,\beta+ e_3}_{{\rm h};r;\lambda f_M}
    }_{L^1_{ {    \lambda f_M' } }(0,T;\fB^{s,\sigma}_p)}  .  
\end{align}
Consequently, we have 
\begin{align}
    \mathcal{I}_3^{s,\sigma;M}
    &=
    \sum_{
    \substack{0 \leq \alpha \leq M \\ \beta \in (\mathbb{N}\cup \{ 0 \})^3}
    } 
    \n{
    u_3
    \partial_{x_3}
    u^{\alpha,\beta}_{{\rm h};r;\lambda f_M}
    }_{L^1(0,T;\fB^{s,\sigma}_p)} 
    \\
    & 
    \leq 
    \frac{C}{\lambda}
    \exp
    \lp{
    \lambda \sup_{0 \leq t < T} f_M(t)
    }
    \n{u_{\rm h}}_{\mathcal{H}^{s,\sigma}_{p;r;\lambda f_M}(0,T)} .   
\end{align}
In the same spirit, it follows that 
\begin{align}
    & \sum_{
    \substack{ \alpha=\alpha'+\alpha'' \\
    \betah=\betah'+\betah'' \\ 
    (\alpha',\betah')\neq (0,0)}
    } 
    \n{
    u^{\alpha',\betah'}_{3;r} 
    \p_{x_3}   u^{\alpha'',(\betah'',\beta_3) }_{{\rm h};r;\lambda f_M}
    }_{L^1(0,T;\fB^{s,\sigma}_p)}   \\
    & \quad 
    \leq C 
    \sum_{
    \substack{ \alpha=\alpha'+\alpha'' \\
    \betah=\betah'+\betah'' \\ 
    (\alpha',\betah')\neq (0,0)}
    } 
    \int_0^T 
    \n{
    u^{\alpha',\betah'}_{3;r} 
    }_{\fB^{\f{2}{p},\frac{1}{p}}_p}  
    \n{
    \p_{x_3}   u^{\alpha'',(\betah'',\beta_3) }_{{\rm h};r;\lambda f_M}
    }_{\fB^{s,\sigma}_p}  dt \\  
    &\quad 
    \leq 
    C
    \exp
    \lp{
    \lambda \sup_{0 \leq t < T} f_M(t)
    }
    \sum_{
    \substack{ \alpha=\alpha'+\alpha'' \\
    \betah=\betah'+\betah'' \\ 
    (\alpha',\betah')\neq (0,0)}
    }
    (\beta_3+1)
    \int_0^T 
    \n{
    u^{\alpha',\betah'}_{3;r} 
    }_{\fB^{\f{2}{p},\frac{1}{p}}_p}  
    \n{
    u^{\alpha'',(\betah'',\beta_3+1) }_{{\rm h};r;\lambda f_M}
    }_{\fB^{s,\sigma}_p}  dt  \\  
    &\quad
    \leq 
    \frac{C}{\lambda}
    \exp
    \lp{
    \lambda \sup_{0 \leq t < T} f_M(t)
    }
    \sum_{
    \substack{ \alpha=\alpha'+\alpha'' \\
    \betah=\betah'+\betah'' \\ 
    (\alpha',\betah')\neq (0,0)}
    }
    (\beta_3+1)
    \n{
    u^{\alpha'',(\betah'',\beta_3+1)  }_{{\rm h};r;\lambda f_M}
    }_{L^1_{ {    \lambda f_M' } }(0,T;\fB^{s,\sigma}_p)}  , 
\end{align}
which implies
\begin{align}
    \mathcal{I}_4^{s,\sigma;M}
    &
    =
    \sum_{
    \substack{0 \leq \alpha \leq M \\ \beta \in (\mathbb{N}\cup \{ 0 \})^3}
    }  
    \sum_{
    \substack{ \alpha=\alpha'+\alpha'' \\
    \betah=\betah'+\betah'' \\ 
    (\alpha',\betah')\neq (0,0)}
    } 
    \n{
    u^{\alpha',\betah'}_{3;r} 
    \p_{x_3}   u^{\alpha'',(\betah'',\beta_3) }_{{\rm h};r;\lambda f_M}
    }_{L^1(0,T;\fB^{s,\sigma}_p)}   \\
    & 
    \leq 
    \frac{C}{\lambda}
    \exp
    \lp{
    \lambda \sup_{0 \leq t < T} f_M(t)
    }
    \n{u_{\rm h}}_{\mathcal{H}^{s,\sigma}_{p;r;\lambda f_M}(0,T)} .
\end{align}
Finally, we consider the estimate of $\mathcal{I}_5^{s,\sigma;M}$.
Observe that the pressure may be decomposed as 
\begin{align}
    P
    &=
    \sum_{\ell,m=1}^2
    \partial_{x_{\ell}}\partial_{x_m}(-\Delta)^{-1}
    (u_{\ell}u_m)
    +
    2
    \sum_{m=1}^2
    \partial_{x_3}\partial_{x_m}(-\Delta)^{-1}
    (u_3u_m)
    +
    \partial_{x_3}^2(-\Delta)^{-1}
    (u_3^2)  \nonumber  \\
    &=
    \sum_{\ell,m=1}^2
    \partial_{x_{\ell}}\partial_{x_m}(-\Delta)^{-1}
    (u_{\ell}u_m)
    +
    2
    \partial_{x_3}(-\Delta)^{-1}
    (u_{\rm h}\cdot \nablah u_3)
     \label{eq:decom-p} \\
    &=:
    P_1 + P_2.   \nonumber 
\end{align}
Now we set 
\begin{align}
    P^{\alpha,\beta}_{i;r}
    :=
    \frac{r^{\alpha+|\betah|}t^{\alpha+\f{|\betah|}{2}}\rho_0^{\beta_3}e^{- \lambda \beta_3 f_M(t)  }}{\alpha!\beta!}
    {\partial_{t}^{\alpha}  
    \partial_{x}^{\beta}  P_i  
    },
    \qquad i=1,2. 
\end{align}
For the estimate of $P^{\alpha,\beta}_{1;r}$, it holds 
\begin{align}
    \nablah P^{\alpha,\beta}_{1;r}
    =
    2 
    \sum_{
  \substack{ \alpha=\alpha'+\alpha'' \\ \beta=\beta'+\beta''}
  } 
  \sum_{\ell,m=1}^2
    \partial_{x_{\ell}}\partial_{x_m}(-\Delta)^{-1}
   \sp{  u^{\alpha',\beta'}_{\ell;r;\lambda f_M}
    u^{\alpha'',\beta''}_{m;r;\lambda f_M}
    }.
\end{align}
From the above relation and Lemma \ref{lemm:prod}, we infer for any $(s,\sigma)\in \Lambda_{p,\theta}^{\rm h}$ and $p,\theta$ satisfying \eqref{p-theta} that  
\begin{align}
&
\sum_{
    \substack{0 \leq \alpha \leq M \\ \beta \in (\mathbb{N}\cup \{ 0 \})^3}
    }     
    \n{
    \nablah P^{\alpha,\beta}_{1;r}
    }_{L^1(0,T;\fB^{s,\sigma}_p)} \\
& \quad 
\leq 
C 
\sum_{
    \substack{0 \leq \alpha' \leq M \\ \beta' \in (\mathbb{N}\cup \{ 0 \})^3}
    } 
\sum_{
    \substack{0 \leq \alpha''\leq M \\ \beta'' \in (\mathbb{N}\cup \{ 0 \})^3}
    } 
\n{
u^{\alpha',\beta'}_{{\rm h};r;\lambda f_M}
}_{L^2(0,T;\fB^{\f{2}{p},\f{1}{p}}_p)} 
\n{
u^{\alpha'',\beta''}_{{\rm h};r;\lambda f_M}
}_{L^2(0,T;\fB^{s,\sigma}_p)}  \\
& \quad 
\leq 
C 
\n{u_{\rm h}}_{\mathcal{H}^{s,\sigma;M}_{p;r;\lambda f_M}(0,T)}
\n{u_{\rm h}}_{\mathcal{H}^{\f{2}{p}-1,\f{1}{p};M}_{p;r;\lambda f_M}(0,T)} . 
\end{align}
Next, we focus on the estimate of $ \nablah P^{\alpha,\beta}_{2;r}$. Using 
\begin{align}
    v_{\rm h}\cdot \nablah \partial_{x_3}w_3
    =
    -
    v_{\rm h}\cdot \nablah \divh w_{\rm h}
    =
    -
    \divh (v_{\rm h}\divh w_{\rm h})
    +
    \divh v_{\rm} \divh w_{\rm h}
\end{align}
for divergence-free vector fields $v$ and $w$,
we have 
\begin{align}
    \nablah P^{\alpha,\beta}_{2;r} 
    &=
    2 
    \sum_{
    \substack{ \alpha=\alpha'+\alpha'' \\ \beta=\beta'+\beta''}
    } 
    \sum_{
    \substack{ \beta_3=\beta_3'+\beta_3'' \\ \beta_3'' \geq 1}
    }   
    \f{r^{\alpha+|\betah|}t^{ \alpha+\f{|\betah|}{2}  }\rho_0^{\beta_3}e^{- \lambda \beta_3 f_M(t)  }}{ \alpha'! \alpha''!   \beta'! \beta''!   }
    \\
    & \qquad\qquad \qquad\qquad\qquad
    \nablah \p_{x_3}(-\Delta)^{-1} \sp{{
    \p_{t,x}^{\alpha',\beta' }
    \uh
    \cdot \nablah} 
    {
    \p_{t,x}^{\alpha'',\beta''  }
    u_3
    }}  \\
    & \quad 
    + 2 \f{r^{\alpha+|\betah|}t^{ \alpha+\f{|\betah|}{2}  }\rho_0^{\beta_3}e^{- \lambda \beta_3 f_M(t)  }}{ \alpha! \beta!   } 
    \nablah \p_{x_3}(-\Delta)^{-1}
    \sp{
    \p_{t,x}^{\alpha,\beta} \uh 
    \cdot \nablah u_3 }  \\ 
    & \quad 
    + 
    2
    \sum_{
    \substack{ \alpha=\alpha'+\alpha'' \\ \betah=\betah'+\betah'' \\
    (\alpha'',\betah'')\neq (0,0)
    } } 
    \f{r^{\alpha+|\betah|}t^{ \alpha+\f{|\betah|}{2}  }\rho_0^{\beta_3}e^{- \lambda \beta_3 f_M(t)  }}{ \alpha'! \alpha''!   \betah'! \betah''! \beta_3! }
    \\
    & \qquad\qquad \qquad\qquad
    \nablah \p_{x_3}(-\Delta)^{-1}  
    \sp{{
    \p_{t,\xh}^{\alpha,\betah' }
    \p_{x_3}^{\beta_3}  
    \uh
    \cdot \nablah} 
    {
    \p_{t,\xh}^{\alpha'',\betah'' }
    u_3
    }}
    \\ 
    &  =  
    -2   
    \sum_{
    \substack{ \alpha=\alpha'+\alpha'' \\ \betah=\betah'+\betah''}
    } 
    \sum_{
    \substack{ \beta_3=\beta_3'+\beta_3'' \\ \beta_3'' \geq 1}
    }   
    \f{r^{\alpha+|\betah|}t^{ \alpha+\f{|\betah|}{2}  }\rho_0^{\beta_3}e^{- \lambda \beta_3 f_M(t)  }}{ \alpha'! \alpha''!   \beta'! \beta''!   }
    \\
    & 
    \qquad\qquad \qquad\qquad\qquad
    \nablah \p_{x_3}(-\Delta)^{-1}
    \sp{
    {
    \p_{t,x}^{\alpha',\beta' }
    \uh
    } 
    {
    \p_{t,x}^{\alpha'',\beta''-e_3 } 
    \divh 
    \uh  
    }}
    \\
    & \quad  
    +  
    2    
    \sum_{
    \substack{ \alpha=\alpha'+\alpha'' \\ \betah=\betah'+\betah''}
    } 
    \sum_{
    \substack{ \beta_3=\beta_3'+\beta_3'' \\ \beta_3'' \geq 1}
    }   
    \f{r^{\alpha+|\betah|}t^{ \alpha+\f{|\betah|}{2}  }\rho_0^{\beta_3}e^{- \lambda \beta_3 f_M(t)  }}{ \alpha'! \alpha''!   \beta'! \beta''!   }
    \\
    &
    \qquad\qquad \qquad\qquad\qquad
    \nablah \p_{x_3}(-\Delta)^{-1}
    \sp{
    {
    \p_{t,x}^{\alpha',\beta' }
    \divh  \uh
    } 
    {
    \p_{t,x}^{\alpha'',\beta''-e_3} 
    \divh  \uh  
    }}
    \\
    & \quad      
    + 
    2     
    \nablah \p_{x_3}(-\Delta)^{-1}
    \sp{
    u^{\alpha,\beta}_{{\rm h};r;\lambda f_M} 
    \cdot \nablah u_3 
    }  \\ 
    & \quad 
    + 
    2
    \sum_{
    \substack{ \alpha=\alpha'+\alpha'' \\ \betah=\betah'+\betah'' \\
    (\alpha'',\betah'')\neq (0,0)
    } } 
    \nablah \p_{x_3}(-\Delta)^{-1} 
    \sp{
    u^{\alpha,(\betah',\beta_3)}_{{\rm h};r;\lambda f_M}
    \cdot \nablah
    u_{3;r}^{\alpha'',\betah'' }
    }
    \\ 
    & =:
    \nablah P_{2,1;r}^{\alpha,\beta}  
    +
    \nablah P_{2,2;r}^{\alpha,\beta}  
    +
    \nablah P_{2,3;r}^{\alpha,\beta}  
    +
    \nablah P_{2,4;r}^{\alpha,\beta}. 
\end{align}
We begin with the estimate of $P_{2,1}^{\alpha,\beta} $. 
It follows from 
\begin{align}
    &
    \nablah \p_{x_3} \divh (-\Delta)^{-1} 
    \sp{
    \p_{t,x}^{\alpha',\beta' }
    \uh
    \p_{t,x}^{\alpha'',\beta''-e_3 } 
    \divh 
    \uh  
    }
    \\
    &\quad
    =
    \nablah \divh (-\Delta)^{-1} 
    \sp{
    \p_{t,x}^{\alpha',\beta'+e_3 }
    \uh
    \p_{t,x}^{\alpha'',\beta''-e_3 } 
    \divh 
    \uh  
    }
    \\
    &\qquad
    +
    \nablah \divh (-\Delta)^{-1} 
    \sp{
    \p_{t,x}^{\alpha',\beta' }
    \uh
    \p_{t,x}^{\alpha'',\beta'' } 
    \divh 
    \uh  
    }
\end{align}
that
\begin{align} 
    \nablah P_{2,1;r}^{\alpha,\beta}  
    &= 
    -2   
    \sum_{
    \substack{ \alpha=\alpha'+\alpha'' \\ \betah=\betah'+\betah''}
    } 
    \sum_{
    \substack{ \beta_3=\beta_3'+\beta_3'' \\ \beta_3'' \geq 1}
    }   
    \f{\beta_3'+1}{\beta_3''}  \nablah  \divh (-\Delta)^{-1}
    \sp{
    u^{\alpha',\beta'+ e_3}_{{\rm h};r;\lambda f_M}
    \divh  
    u^{\alpha'',\beta''- e_3}_{{\rm h};r;\lambda f_M}} 
    \\
    & \quad 
    -2   
    \sum_{
    \substack{ \alpha=\alpha'+\alpha'' \\ \betah=\betah'+\betah''}
    } 
    \sum_{
    \substack{ \beta_3=\beta_3'+\beta_3'' \\ \beta_3'' \geq 1}
    }   
    \nablah  \divh (-\Delta)^{-1} 
    \sp{
    u^{\alpha',\beta'}_{{\rm h};r;\lambda f_M}
    \divh u^{\alpha'',\beta''}_{{\rm h};r;\lambda f_M}},
\end{align}  
which and the same arguments as in \eqref{eq:B-1+1} and \eqref{eq:B-00} imply that for any $(s,\sigma)\in \Lambda_{p,\theta}^{\rm h}$,
\begin{align}
    &
    \sum_{
    \substack{0 \leq \alpha \leq M \\ \beta \in (\mathbb{N}\cup \{ 0 \})^3, \beta_3 \geq 1}
    }
    \n{ \nablah P_{2,1;r}^{\alpha,\beta}  }_{ L^1 (0,T; \fB^{s,\sigma}_p ) }   \\
    & \qquad 
    \leq 
    C
    \n{u_{\rm h}}_{\mathcal{H}^{s,\sigma;M}_{p;r;\lambda f_M}(0,T)}
    \n{u_{\rm h}}_{\mathcal{H}^{\f{2}{p}-1,\frac{1}{p};M}_{p;r;\lambda f_M}(0,T)}
    +
    \frac{C}{\lambda}
    \n{u_{\rm h}}_{\mathcal{H}^{s,\sigma;M}_{p;r;\lambda f_M}(0,T)}.
\end{align} 
Next, we consider the estimate of $\nablah P_{2,3;r}^{\alpha,\beta}$. By Lemma \ref{lemm:prod}, it holds that 
\begin{align}
    \n{
    \nablah P_{2,3;r}^{\alpha,\beta}  
    }_{L^1(0,T;\fB^{s,\sigma}_p)} 
    &\leq
    C
    \int_0^T
    \n{
    u^{\alpha,\beta}_{{\rm h};r;\lambda f_M} (t) \cdot \nablah u_3(t)
    }_{\fB^{s,\sigma}_p}
    dt \\
    & 
    \leq 
    C \int_0^T
    \n{ u_3(t) }_{\fB^{\f{2}{p}+1,\f{1}{p}}_p}  
    \n{  u^{\alpha,\beta}_{{\rm h};r;\lambda f_M}(t)}_{\fB^{s,\sigma}_p}
    dt, 
\end{align}
which
gives rise to  
\begin{align}
    \sum_{
    \substack{0 \leq \alpha \leq M \\ \beta \in (\mathbb{N}\cup \{ 0 \})^3}
    }
    \n{
    \nablah P_{2,3;r}^{\alpha,\beta}  
    }_{L^1(0,T;\fB^{s,\sigma}_p)} 
    \leq {}
    C 
    \int_0^T
    \n{ u_3 (t)}_{\fB^{\f{2}{p}+1,\f{1}{p}}_p}  
    \n{u_{\rm h}}_{\mathcal{H}^{s,\sigma;M}_{p;r;\lambda f_M}(0,t)}
    dt.  
\end{align}
For the estimate of $\nablah P_{2,4;r}^{\alpha,\beta}$, it follows from Lemma \ref{lemm:prod} that
\begin{align}
    & \n{ 
    \nablah P_{2,4;r}^{\alpha,\beta}
    }_{L^1(0,T;\fB^{s,\sigma}_p)}  \\
    & \quad 
    \leq C
    \sum_{
    \substack{ \alpha=\alpha'+\alpha'' \\ \betah=\betah'+\betah'' \\
    (\alpha'',\betah'')\neq (0,0)
    } } 
    \int_0^T
    \n{
    u^{\alpha',(\betah',\beta_3)}_{{\rm h};r;\lambda f_M}(t)
    }_{\fB^{s,\sigma}_p} 
    \n{
    u_{3;r}^{ \alpha'',\betah''  }(t)
    }_{\fB^{\f{2}{p}+1,\f{1}{p}}_p} dt  \\
    & \quad 
    \leq C
    \sum_{
    \substack{ \alpha=\alpha'+\alpha'' \\ \betah=\betah'+\betah'' 
    } } 
    \int_0^T
    \n{
    u^{\alpha',(\betah',\beta_3)}_{{\rm h};r;\lambda f_M}(t)
    }_{\fB^{s,\sigma}_p} 
    \n{
    u_{3;r}^{ \alpha'',\betah''  }(t)
    }_{\fB^{\f{2}{p}+1,\f{1}{p}}_p} dt , 
\end{align}
which yields  
\begin{align}
    & \sum_{
    \substack{0 \leq \alpha \leq M \\ \beta \in (\mathbb{N}\cup \{ 0 \})^3}
    }
    \n{ 
    \nablah P_{2,4;r}^{\alpha,\beta}
    }_{L^1(0,T;\fB^{s,\sigma}_p)} \\
    & \quad 
    \leq C
    \sum_{
    \substack{0 \leq \alpha' \leq M \\ \beta' \in (\mathbb{N}\cup \{ 0 \})^3}}
    \sum_{
    \substack{0 \leq \alpha'' \leq M \\ \betah'' \in (\mathbb{N}\cup \{ 0 \})^2}}
    \int_0^T
    \n{
    u^{\alpha',\beta'}_{{\rm h};r;\lambda f_M}(t)
    }_{\fB^{s,\sigma}_p} 
    \n{
    u_{3;r}^{ \alpha'',\betah''  }(t)
    }_{\fB^{\f{2}{p}+1,\f{1}{p}}_p} dt \\
    & \quad 
    \leq 
    C 
    \int_0^T  
    \sum_{
    \substack{0 \leq \alpha'' \leq M \\ \betah'' \in (\mathbb{N}\cup \{ 0 \})^2}
    }
    \n{
    u_{3;r}^{ \alpha'',\betah'' }(t)
    }_{\fB^{\f{2}{p}+1,\f{1}{p}}_p}
    \n{u_{\rm h}}_{\mathcal{H}^{s,\sigma;M}_{p;r;\lambda f_M}(0,t)}  dt.
\end{align}

It remains to consider the estimate of $\nablah P_{2,2;r}^{\alpha,\beta}$ which is more involved. To this end, we see from Lemma \ref{lemm:prod} 
that  
\begin{align}
    &\n{\nablah P_{2,2;r}^{\alpha,\beta} }_{   \fB^{\f{2}{p}-1,\f{1}{p}}_p  }\\
    &
    \leq C    
    \sum_{
    \substack{ \alpha=\alpha'+\alpha'' \\ \betah=\betah'+\betah''}
    } 
    \sum_{
    \substack{ \beta_3=\beta_3'+\beta_3'' \\ \beta_3'' \geq 1}
    }   
    \f{r^{\alpha+|\betah|}t^{ \alpha+\f{|\betah|}{2}  }\rho_0^{\beta_3}e^{- \lambda \beta_3 f_M(t)  }  }{ \alpha'! \alpha''!   \beta'! \beta''!   }
    \\
    & \qquad 
    \sum_{j,k\in \mathbb{Z}} \f{ 2^k 2^j  }{ 2^{2k}+2^{2j} } 
    2^{(\f{2}{p}-1)k} 2^{\f{1}{p}j} 
    \n{
    \Deltahh_k\Deltav_j  \sp{
    {
    \p_{t,x}^{\alpha',\beta' }
    \divh  \uh
    } 
    {
    \p_{t,x}^{\alpha'',\beta''-e_3 }  
    \divh  \uh  
    } }   
    }_{L^p} \\ 
    &= C    
    \sum_{
    \substack{ \alpha=\alpha'+\alpha'' \\ \betah=\betah'+\betah''}
    } 
    \sum_{
    \substack{ \beta_3=\beta_3'+\beta_3'' \\ \beta_3'' \geq 1}
    }   
    \f{r^{\alpha+|\betah|}t^{ \alpha+\f{|\betah|}{2}  }\rho_0^{\beta_3}e^{- \lambda \beta_3 f_M(t)  }  }{ \alpha'! \alpha''!   \beta'! \beta''!   }
    \\
    & \quad \quad 
    \sum_{j,k\in \mathbb{Z}} \f{ 2^{(2-\theta)k} 2^{\theta j} }
    { 2^{2k}+2^{2j} } 
    2^{ (\f{2}{p} -2+\theta)  k }   2^{ (1+\f{1}{p}-\theta) j}     
    \n{
    \Deltahh_k\Deltav_j  \sp{
    {
    \p_{t,x}^{\alpha',\beta' }
    \divh  \uh
    } 
    {
    \p_{t,x}^{\alpha'',\beta''-e_3 }  
    \divh  \uh  
    } }   
    }_{L^p}   \\ 
    &\leq C    
    \sum_{
    \substack{ \alpha=\alpha'+\alpha'' \\ \betah=\betah'+\betah''}
    } 
    \sum_{
    \substack{ \beta_3=\beta_3'+\beta_3'' \\ \beta_3'' \geq 1}
    }   
    \f{r^{\alpha+|\betah|}t^{ \alpha+\f{|\betah|}{2}  }\rho_0^{\beta_3}e^{- \lambda \beta_3 f_M(t)  }  }{ \alpha'! \alpha''!   \beta'! \beta''!   }   \\
    & \qquad 
    \n{
    {
    \p_{t,x}^{\alpha',\beta' }
    \divh  \uh
    } 
    {
    \p_{t,x}^{\alpha'',\beta''-e_3 }  
    \divh  \uh  
    } 
    }_{  \fB^{\f{2}{p}-2+\theta,1+\f{1}{p}-\theta}_p  } , 
\end{align}
which and Lemma \ref{lemm:prod} yield
\begin{align}
    \n{\nablah P_{2,2;r}^{\alpha,\beta} }_{   \fB^{\f{2}{p}-1,\f{1}{p}}_p  }
    &\leq C    
    \sum_{
    \substack{ \alpha=\alpha'+\alpha'' \\ \betah=\betah'+\betah''}
    } 
    \sum_{
    \substack{ \beta_3=\beta_3'+\beta_3'' \\ \beta_3'' \geq 1}
    }   
    \f{r^{\alpha+|\betah|}t^{ \alpha+\f{|\betah|}{2}  }\rho_0^{\beta_3}e^{- \lambda \beta_3 f_M(t)  }  }{ \alpha'! \alpha''!   \beta'! \beta''!   }\\
    &\qquad 
    \left(
    \n{
    \p_{t,x}^{\alpha',\beta' }
    \divh  \uh
    }_{    \fB^{\f{2}{p}-2+\theta,\f{1}{p}-\theta}_p    } 
    \n{
    \p_{t,x}^{\alpha'',\beta''-e_3 }   
    \divh  \uh  
    }_{ \fB^{\f{2}{p},1+\f{1}{p}}_p } \right. \\
    & \qquad 
    \left.
    +
    \n{
    \p_{t,x}^{\alpha',\beta' }
    \divh  \uh
    }_{ \fB^{\f{2}{p}-2+\theta, 1+\f{1}{p}-\theta}_p } 
    \n{
    \p_{t,x}^{\alpha'',\beta''-e_3 }  
    \divh  \uh  
    }_{ \fB^{\f{2}{p},\f{1}{p}}_p } \right) \\
    &\leq C    
    \sum_{
    \substack{ \alpha=\alpha'+\alpha'' \\ \betah=\betah'+\betah''}
    } 
    \sum_{
    \substack{ \beta_3=\beta_3'+\beta_3'' \\ \beta_3'' \geq 1}
    }   
    \f{r^{\alpha+|\betah|}t^{ \alpha+\f{|\betah|}{2}  }\rho_0^{\beta_3}e^{- \lambda \beta_3 f_M(t)  }  }{ \alpha'! \alpha''!   \beta'! \beta''!   }\\  
    & \qquad 
    \left(  
    \n{
    \p_{t,x}^{\alpha',\beta' }
    \uh
    }_{ \fB^{\f{2}{p}-1+\theta,  \f{1}{p}-\theta }_p } 
    \n{
    \p_{t,x}^{\alpha'',\beta'' }  
    \uh  
    }_{ \fB^{\f{2}{p}+1,\f{1}{p}}_p } \right. \\
    &\left. \quad \qquad 
    +  
    \n{
    \p_{t,x}^{\alpha'',\beta''-e_3 } 
    \uh  
    }_{ \fB^{\f{2}{p}+1,\f{1}{p}}_p }
    \n{
    \p_{t,x}^{\alpha',\beta' +e_3 }
    \uh
    }_{ \fB^{\f{2}{p}-1+\theta , \f{1}{p}-\theta  }_p } 
    \right) . \label{22r:1}
\end{align}
Hence, it holds
\begin{align}
    &\sum_{
    \substack{0 \leq \alpha \leq M \\ \betah \in (\mathbb{N}\cup \{ 0 \})^2\\ \beta_3 \geq 1}
    }
    \n{
    \nablah
    P_{2,2;r}^{\alpha,\beta}
    }_{ L^1(0,T; \fB^{\f{2}{p}-1,\frac{1}{p}}_p ) } 
    \\
    & 
    \quad 
    \leq 
    C
    \sum_{
    \substack{0 \leq \alpha \leq M \\ \betah \in (\mathbb{N}\cup \{ 0 \})^2 \\ \beta_3 \geq 1}
    }  
    \sum_{
    \substack{ \alpha=\alpha'+\alpha'' \\ \beta=\beta'+\beta''\\ \beta_3'' \geq 1}
    }   
    \n{
    u^{\alpha',\beta'}_{{\rm h};r;\lambda f_M}
    }_{  L^{\infty}(0,T; \fB^{\f{2}{p}-1+\theta,  \f{1}{p}-\theta }_p) } 
    \n{
    u^{\alpha'',\beta''}_{{\rm h};r;\lambda f_M}
    }_{ L^1(0,T;   \fB^{\f{2}{p}+1,\f{1}{p}}_p )   }   
    \\
    &\qquad
    +
    C 
    \sum_{
    \substack{0 \leq \alpha \leq M \\ \betah \in (\mathbb{N}\cup \{ 0 \})^2 \\ \beta_3 \geq 1}
    } 
    \sum_{
    \substack{ \alpha=\alpha'+\alpha'' \\ \beta=\beta'+\beta''\\ \beta_3'' \geq 1}
    }   
    \frac{ \beta_3'+1 }{ \beta_3''  }
    \int_0^T
    \n{
    u^{\alpha'',\beta_3''-e_3}_{{\rm h};r;\lambda f_M}
    }_{ \fB^{\f{2}{p}+1,\f{1}{p}}_p }
    \n{
    u^{\alpha',\beta_3'+e_3}_{{\rm h};r;\lambda f_M}
    }_{ \fB^{\f{2}{p}-1+\theta , \f{1}{p}-\theta  }_p } 
    dt
    \\
    & 
    \quad 
    \leq 
    C
    \n{u_{\rm h}}_{\mathcal{H}^{\f{2}{p}-1+\theta,  \f{1}{p}-\theta}_{p;r;\lambda f_M}(0,T)}
    \n{u_{\rm h}}_{\mathcal{H}^{\frac{2}{p}-1,\f{1}{p}}_{p;r;\lambda f_M}(0,T)}
    +
    \f{C}{\lambda}
    \n{u_{\rm h}}_{\mathcal{H}^{ \f{2}{p}-1+\theta , \f{1}{p}-\theta }_{p;r;\lambda f_M} (0,T)}, 
\end{align}
where, in the last step, we have invoked similar arguments of \eqref{eq:B-1+1}. 
We see that 
\begin{align}
    & \n{\nablah P_{2,2;r}^{\alpha,\beta}  }_{ \fB^{\f{2}{p}-2+\theta, \f{1}{p}-\theta}_p }
    \\
    &\leq C    
    \sum_{
    \substack{ \alpha=\alpha'+\alpha'' \\ \beta=\beta'+\beta'' \\ \beta_3'' \geq 1}
    }   
    \f{r^{\alpha+|\betah|}t^{ \alpha+\f{|\betah|}{2}  }\rho_0^{\beta_3}e^{- \lambda \beta_3 f_M(t)  }}{ \alpha'! \alpha''!   \beta'! \beta''!  }  
    \\
    &  \qquad    
    \sum_{j,k\in \mathbb{Z}} \f{ 2^k 2^j}{ 2^{2k}+2^{2j} } 
    2^{  (\f{2}{p}-2+\theta) k}  2^{ (\f{1}{p}-\theta)j } 
    \n{
    \Deltahh_k\Deltav_j  \sp{
    {
    \p_{t,x}^{\alpha',\beta' }
    \divh  \uh
    } 
    {
    \p_{t,x}^{\alpha'',\beta''-e_3 } 
    \divh  \uh  
    } }   
    }_{L^p}   \\  
    & \leq C  
    \sum_{
    \substack{ \alpha=\alpha'+\alpha'' \\ \beta=\beta'+\beta''\\ \beta_3'' \geq 1}
    }   
    \f{r^{\alpha+|\betah|}t^{ \alpha+\f{|\betah|}{2}  }\rho_0^{\beta_3}e^{- \lambda \beta_3 f_M(t)  }}{ \alpha'! \alpha''!   \beta'! \beta''!   } 
    \n{
    {
    \p_{t,x}^{\alpha',\beta' }
    \divh  \uh
    } 
    {
    \p_{t,x}^{\alpha'',\beta''-e_3 }  
    \divh  \uh  
    } 
    }_{ \fB^{\f{2}{p}-2+\theta, \f{1}{p}-\theta}_p } \\
    & \leq 
    C  \sum_{
    \substack{ \alpha=\alpha'+\alpha'' \\ \beta=\beta'+\beta'' \\ \beta_3'' \geq 1}
    }   
    \f{r^{\alpha+|\betah|}t^{ \alpha+\f{|\betah|}{2}  }\rho_0^{\beta_3}e^{- \lambda \beta_3 f_M(t)  }}{ \alpha'! \alpha''!   \beta'! \beta''!   }
    \\
    & \qquad   \qquad    \qquad   
    \n{
    \p_{t,x}^{\alpha',\beta' }
    \divh  \uh
    }_{ \fB^{  \f{2}{p}-1+\theta, \f{1}{p}-\theta }_p }
    \n{
    \p_{t,x}^{\alpha'',\beta''-e_3 }
    \divh  \uh  
    }_{ \fB^{\f{2}{p}-1,\f{1}{p}}_p }  \\
    & \leq   
    C  \sum_{
    \substack{ \alpha=\alpha'+\alpha'' \\ \beta=\beta'+\beta''\beta_3'' \geq 1}
    }   
    \f{r^{\alpha+|\betah|}t^{ \alpha+\f{|\betah|}{2}  }\rho_0^{\beta_3}e^{- \lambda (\beta_3-1) f_M(t)  }}{ \alpha'! \alpha''!   \betah'! \betah''! \beta_3'! (\beta_3''-1)!  }
    \\
    & \qquad   \qquad    \qquad    \qquad  
    \n{
    \p_{t,x}^{\alpha',\beta' }
    \uh
    }_{ \fB^{\f{2}{p}+\theta, \f{1}{p}-\theta}_p }
    \n{
    \p_{t,x}^{\alpha'',\beta''-e_3 }  
    \uh  
    }_{ \fB^{\f{2}{p},\f{1}{p}}_p } . 
\end{align}
Thus, we have
\begin{align}
    & \sum_{
    \substack{0 \leq \alpha \leq M \\ \betah \in (\mathbb{N}\cup \{ 0 \})^2\\ \beta_3 \geq 1} 
    }  
    \n{
    \nablah
    P_{2,2;r}^{\alpha,\beta}
    }_{ L^1(0,T; \fB^{\f{2}{p}-2+\theta, \f{1}{p}-\theta}_p ) } 
    \\
    & 
    \quad 
    \leq  
    C
    \sum_{
    \substack{0 \leq \alpha \leq M \\ \betah \in (\mathbb{N}\cup \{ 0 \})^2\\ \beta_3 \geq 1}
    }
    \sum_{
    \substack{ \alpha=\alpha'+\alpha'' \\ \beta=\beta'+\beta''\\ \beta_3'' \geq 1}
    }   
    \n{    
    u^{ \alpha',\beta' }_{{\rm h};r;\lambda f_M}
    }_{L^2(0,T;\fB^{\f{2}{p}+\theta, \f{1}{p}-\theta}_p )}
    \n{
    u^{\alpha'',\beta''-e_3}_{{\rm h};r;\lambda f_M}
    }_{L^2(0,T; \fB^{\f{2}{p},\f{1}{p}}_p  )} \\
    & 
    \quad 
    \leq 
    C
    \n{u_{\rm h}}_{\mathcal{H}^{\f{2}{p}-1+\theta,\frac{1}{p}-\theta;M}_{p;r;\lambda f_M}(0,T)}
    \n{u_{\rm h}}_{\mathcal{H}^{\frac{2}{p}-1,\f{1}{p};M}_{p;r;\lambda f_M}(0,T)}. 
\end{align}
Making use of Lemma \ref{lemm:prod},  
we see that
\begin{align}
    &\n{\nablah P_{2,2;r}^{\alpha,\beta} }_{ \fB^{\f{2}{p}-1+\theta ,\f{1}{p}-\theta}_p }
    \\
    &\leq C    
    \sum_{
    \substack{ \alpha=\alpha'+\alpha'' \\ \beta=\beta'+\beta'' \beta_3'' \geq 1}
    }   
    \f{r^{\alpha+|\betah|}t^{ \alpha+\f{|\betah|}{2}  }\rho_0^{\beta_3}e^{- \lambda \beta_3 f_M(t)  }}{ \alpha'! \alpha''!   \beta'! \beta''!   }
    \sum_{j,k\in \mathbb{Z}} \f{ 2^k 2^j}{ 2^{2k}+2^{2j} } 
    2^{ (\f{2}{p}-1+\theta ) k  } 2^{ (\f{1}{p}-\theta)j }        \\
    & \qquad 
    \n{
    \Deltahh_k\Deltav_j  \sp{
    {
    \p_{t,x}^{\alpha',\beta' }
    \divh  \uh
    } 
    {
    \p_{t,x}^{\alpha'',\beta''-e_3 }
    \divh  \uh  
    } }   
    }_{L^p}   \\  
    &\leq C    
    \sum_{
    \substack{ \alpha=\alpha'+\alpha'' \\ \betah=\beta'+\beta''\beta_3'' \geq 1}
    }   
    \f{r^{\alpha+|\betah|}t^{ \alpha+\f{|\betah|}{2}  }\rho_0^{\beta_3}e^{- \lambda \beta_3 f_M(t)  }}{ \alpha'! \alpha''!   \beta'! \beta''!   }
    \sum_{j,k\in \mathbb{Z}} \f{ 2^{2k}}{ 2^{2k}+2^{2j} } 
    2^{(\f{2}{p}-2+\theta )k}   2^{ (1+\f{1}{p}-\theta) j}        \\
    & \qquad 
    \n{
    \Deltahh_k\Deltav_j  \sp{
    {
    \p_{t,x}^{\alpha',\beta' }
    \divh  \uh
    } 
    {
    \p_{t,x}^{\alpha'',\beta''-e_3 } 
    \divh  \uh  
    } }   
    }_{L^p}   \\  
    & \leq C  
    \sum_{
    \substack{ \alpha=\alpha'+\alpha'' \\ \beta=\beta'+\beta'' \beta_3'' \geq 1}
    }   
    \f{r^{\alpha+|\betah|}t^{ \alpha+\f{|\betah|}{2}  }\rho_0^{\beta_3}e^{- \lambda \beta_3 f_M(t)  }}{ \alpha'! \alpha''!   \beta'! \beta''!   }\\
    & \qquad  
    \n{
    {
    \p_{t,x}^{\alpha',\beta' }
    \divh  \uh
    } 
    {
    \p_{t,x}^{\alpha'',\beta''-e_3 }  
    \divh  \uh  
    } 
    }_{ \fB^{\f{2}{p}-2+\theta,1+\f{1}{p}-\theta}_p }. \label{22r:2}
\end{align}
Since the right-hand side in \eqref{22r:2} is completely same as that in \eqref{22r:1},
we may follow the same argument as above to obtain
\begin{align}
    &\sum_{
    \substack{0 \leq \alpha \leq M \\ \betah \in (\mathbb{N}\cup \{ 0 \})^2\\ \beta_3 \geq 1}
    }
    \n{
    \nablah
    P_{2,2;r}^{\alpha,\beta}
    }_{ L^1(0,T; \fB^{\f{2}{p}-1+\theta ,\f{1}{p}-\theta}_p ) } 
    \\
    & 
    \qquad 
    \leq 
    C
    \n{u_{\rm h}}_{\mathcal{H}^{\f{2}{p}-1+\theta,  \f{1}{p}-\theta}_{p;r;\lambda f_M}(0,T)}
    \n{u_{\rm h}}_{\mathcal{H}^{\frac{2}{p}-1,\f{1}{p}}_{p;r;\lambda f_M}(0,T)}
    +
    \f{C}{\lambda}
    \n{u_{\rm h}}_{\mathcal{H}^{ \f{2}{p}-1+\theta , \f{1}{p}-\theta }_{p;r;\lambda f_M} (0,T)}, 
\end{align}


\noindent 
{\it Step 3. Analytic estimates for $u_3$.}
Here, we aim to obtain the analytic estimates for $u_3$ with respect to $t>0$ and $\xh \in \mathbb{R}^2$. Compared with the previous step, the calculation here simplifies a lot due to the absence of vertical derivative. 

\begin{align}
    \partial_{t,\xh}^{\alpha,\betah}[(u\cdot \nabla)u_3 + \partial_{x_3} P]
    ={}
    &
    \sum_{
    \substack{ \alpha=\alpha'+\alpha'' \\ \betah=\betah'+\betah''}
    }
    \f{\alpha!}{\alpha'! \alpha''!} 
    \f{\betah!}{\betah'! \betah''!} 
    { \p_{t,\xh}^{\alpha',\betah'} \uh }
    \cdot \nablah
    { \p_{t,\xh}^{\alpha',\betah'} u_3 }
    \\
    & 
    -
    \sum_{
    \substack{ \alpha=\alpha'+\alpha'' \\ \betah=\betah'+\betah''}
    }
    \f{\alpha!}{\alpha'! \alpha''!} 
    \f{\betah!}{\betah'! \betah''!} 
    { \p_{t,\xh}^{\alpha',\betah'} u_3 }
    { \p_{t,\xh}^{\alpha'',\betah''} \divh \uh }
    \\
    &
    + 
    \p_{x_3} \p_{t,\xh}^{\alpha,\betah} P, 
\end{align}
which implies
\begin{align}
    [(u\cdot \nabla)u + \nabla P]_{3;r}^{\alpha,\betah}
    ={}&
    \sum_{
    \substack{ \alpha=\alpha'+\alpha'' \\ \betah=\betah'+\betah''}
    } 
    u_{{\rm h};r}^{ \alpha',\betah'}
    \cdot \nablah  u_{3;r}^{ \alpha'',\betah''}
    \\
    &
    - 
    \sum_{
    \substack{ \alpha=\alpha'+\alpha'' \\ \betah=\betah'+\betah''}
    } 
    u_{3;r}^{ \alpha',\betah'} 
    \divh u_{{\rm h};r}^{ \alpha'',\betah''}
    + 
    \p_{x_3} P_r^{\alpha,\betah}. 
\end{align}
From Proposition \ref{prop:anal-lin}, it follows that
for any $(s,\sigma)\in \Lambda_{p,\theta}^{\rm v}$,
\begin{align}
    \n{
    u_{3}
    }_{\mathcal{V}_{p,\theta;r}^{s,\sigma;M}(0,T)}
    &\leq  
    C \n{ 
    u_{0,3}
    }_{\fB^{s,\sigma}_p}  
    + 
    Cr
    \n{
    u_{3}  
    }_{\mathcal{V}_{p,\theta;r}^{s,\sigma;M}(0,T)}
    \\
    &\quad 
    +
    C
    \sum_{\substack{0 \leq \alpha \leq M \\ \betah \in (\mathbb{N} \cup \{0\})^2}}
    \n{[(u\cdot \nabla)u + \nabla P]_{3;r}^{\alpha,\betah}}_{L^1(0,T;\fB_p^{s,\sigma})}
    \\
    &
    \leq  
    C \n{ 
    u_{0,3}
    }_{\fB^{s,\sigma}_p}  
    + 
    Cr
    \n{
    u_{3}  
    }_{\mathcal{V}_{p,\theta;r}^{s,\sigma;M}(0,T)}
    +
    C
    \sum_{m=1}^3
    \mathcal{J}^{s,\sigma;M}_{m},
\end{align}
where we set 
\begin{align}
    \mathcal{J}^{s,\sigma;M}_1
    & 
    :=
    \sum_{
    \substack{0 \leq \alpha \leq M \\ \betah \in (\mathbb{N}\cup \{ 0 \})^2
    }
    } 
    \sum_{
    \substack{ \alpha=\alpha'+\alpha'' \\ \betah=\betah'+\betah''}
    } 
    \n{
    u_{{\rm h};r}^{ \alpha',\betah'}
    \cdot \nablah  u_{3;r}^{ \alpha'',\betah''}
    }_{L^1(0,T;\fB^{s,\sigma}_p)} , \\
    \mathcal{J}^{s,\sigma;M}_{2}
    & 
    :=
    \sum_{
    \substack{0 \leq \alpha \leq M \\ \betah \in (\mathbb{N}\cup \{ 0 \})^2
    }
    } 
    \sum_{
    \substack{ \alpha=\alpha'+\alpha'' \\ \betah=\betah'+\betah''}
    } 
    \n{
    u_{3;r}^{ \alpha',\betah'} 
    \divh u_{{\rm h};r}^{ \alpha'',\betah''}
    }_{L^1(0,T;\fB^{s,\sigma}_p)}  , \\
    \mathcal{J}^{s,\sigma;M}_{3}
    & 
    :=
    \sum_{
    \substack{0 \leq \alpha \leq M \\ \betah \in (\mathbb{N}\cup \{ 0 \})^2
    }
    } 
    \n{\p_{x_3} p_r^{\alpha,\betah}}_{L^1(0,T;\fB^{s,\sigma}_p)} .
\end{align}
It follows from Lemma \ref{lemm:prod} that 
\begin{align}
  \mathcal{J}^{s,\sigma;M}_{1}   
  & = \sum_{
    \substack{0 \leq \alpha \leq M \\ \betah \in (\mathbb{N}\cup \{ 0 \})^2}
    }  
    \sum_{
    \substack{ \alpha=\alpha'+\alpha'' \\ \betah=\betah'+\betah''}
    } 
    \n{
    u_{{\rm h};r}^{ \alpha',\betah'}
    \cdot \nablah  u_{3;r}^{ \alpha'',\betah''}
    }_{L^1(0,T;\fB^{s,\sigma}_p)} \\
    &  
    \leq 
    \sum_{
    \substack{0 \leq \alpha' \leq M \\ \betah' \in (\mathbb{N}\cup \{ 0 \})^2}
    }  
    \sum_{
    \substack{0 \leq \alpha''\leq M \\ \betah'' \in (\mathbb{N}\cup \{ 0 \})^2}
    }   
    \n{
    u_{{\rm h};r}^{ \alpha',\betah'}
    \cdot \nablah  u_{3;r}^{ \alpha'',\betah''}
    }_{L^1(0,T;\fB^{s,\sigma}_p)} \label{eq:uh-u3} \\
    & 
    \leq 
    C  
    \sum_{
    \substack{0 \leq \alpha' \leq M \\ \betah' \in (\mathbb{N}\cup \{ 0 \})^2}
    }  
    \sum_{
    \substack{0 \leq \alpha''\leq M \\ \betah'' \in (\mathbb{N}\cup \{ 0 \})^2}
    }   
    \n{ 
    u_{{\rm h};r}^{ \alpha',\betah'}
    }_{L^2(0,T;\fB^{\f{2}{p},\f{1}{p}}_p)} 
    \n{
    u_{3;r}^{ \alpha'',\betah''}
    }_{L^2(0,T;\fB^{s+1,\sigma}_p)} \\
    &  
    \leq 
    C  \n{u_{{\rm h}}}_{\mathcal{H}^{\f{2}{p}-1,\f{1}{p};M}_{p;r;\lambda f_M}(0,T)} 
    \n{u_{3}}_{\mathcal{V}^{s,\sigma;M}_{p,r}(0,T)}.
\end{align}
Analogously, it holds that 
\begin{align}
 \mathcal{J}^{s,\sigma;M}_{2}    &  =
    \sum_{
    \substack{0 \leq \alpha \leq M \\ \betah \in (\mathbb{N}\cup \{ 0 \})^2}
    }  
    \sum_{
    \substack{ \alpha=\alpha'+\alpha'' \\ \betah=\betah'+\betah''}
    } 
    \n{
    u_{3;r}^{ \alpha',\betah'} 
    \divh u_{{\rm h};r}^{ \alpha'',\betah''}
    }_{L^1(0,T;\fB^{s,\sigma}_p)} \\
    & 
    \leq 
    \sum_{
    \substack{0 \leq \alpha' \leq M \\ \betah' \in (\mathbb{N}\cup \{ 0 \})^2}
    }  
    \sum_{
    \substack{0 \leq \alpha''\leq M \\ \betah'' \in (\mathbb{N}\cup \{ 0 \})^2}
    }   
    \n{
    u_{3;r}^{ \alpha',\betah'} 
    \divh u_{{\rm h};r}^{ \alpha'',\betah''}
    }_{L^1(0,T;\fB^{s,\sigma}_p)} \label{eq:uh-uh} \\
    & 
    \leq 
    C 
    \sum_{
    \substack{0 \leq \alpha' \leq M \\ \betah' \in (\mathbb{N}\cup \{ 0 \})^2}
    }  
    \sum_{
    \substack{0 \leq \alpha''\leq M \\ \betah'' \in (\mathbb{N}\cup \{ 0 \})^2}
    }   
    \n{
    u_{3;r}^{ \alpha',\betah'} 
    }_{L^2(0,T;\fB^{\f{2}{p},\f{1}{p}}_p)} 
    \n{
    u_{{\rm h};r}^{ \alpha'',\betah''}
    }_{L^2(0,T;\fB^{s+1,\sigma}_p)} \\
    & 
    \leq 
    C 
    \n{ u_{3} }_{\mathcal{V}^{\f{2}{p}-1,\f{1}{p};M}_p(0,T)} 
    \n{  u_{{\rm h}}  }_{\mathcal{H}^{s,\sigma;M}_{p;r;\lambda f_M}(0,T)} . 
\end{align}
By the decomposition of pressure \eqref{eq:decom-p}, we see
\begin{align}
    \p_{x_3} P_r^{\alpha,\betah}
    & =
    \sum_{
    \substack{ \alpha=\alpha'+\alpha'' \\ \betah=\betah'+\betah''}
    } 
    \sum_{k,\ell=1}^2\partial_{x_k}\partial_{x_3}(-\Delta)^{-1}
    \sp{ \p_{x_{\ell}} u_{k;r}^{ \alpha',\betah'}   }
    u_{\ell;r}^{ \alpha'',\betah''}  \\
    & \quad 
    +
    \sum_{
    \substack{ \alpha=\alpha'+\alpha'' \\ \betah=\betah'+\betah''}
    } 
    \sum_{k,\ell=1}^2\partial_{x_k}\partial_{x_3}(-\Delta)^{-1}
    u_{k;r}^{ \alpha',\betah'}  
    \sp{ \p_{x_{\ell}} u_{\ell;r}^{ \alpha'',\betah''}   } \\
    & \quad 
    + 2 
    \sum_{
    \substack{ \alpha=\alpha'+\alpha'' \\ \betah=\betah'+\betah''}
    } 
    \partial_{x_3}^2(-\Delta)^{-1} 
    \sp{
    u_{{\rm h};r}^{ \alpha',\betah'} 
    \cdot  \nablah u_{3;r}^{ \alpha'',\betah''} 
    } .
\end{align}
From the same arguments as \eqref{eq:uh-u3} and \eqref{eq:uh-uh}, it follows that 
\begin{align}
  \mathcal{J}^{s,\sigma;M}_{3}   & = \sum_{
    \substack{0 \leq \alpha \leq M \\ \betah \in (\mathbb{N}\cup \{ 0 \})^2}
    }    
    \n{\p_{x_3} P_r^{\alpha,\betah}}_{L^1(0,T;\fB^{s,\sigma}_p)}
    \\
    & 
    \leq 
    C 
    \n{u_{{\rm h}}}_{\mathcal{H}^{\f{2}{p}-1,\f{1}{p};M}_{p;r;\lambda f_M}(0,T)} 
    \n{ u_{{\rm h}}  }_{  \mathcal{H}^{s,\sigma;M}_{p;r;\lambda f_M}(0,T)} 
    +
    C
    \n{u_{{\rm h}}}_{\mathcal{H}^{\f{2}{p}-1,\f{1}{p};M}_{p;r;\lambda f_M}(0,T)} 
    \n{u_{3}}_{\mathcal{V}^{s,\sigma;M}_{p,r}(0,T)} . 
\end{align}
Combining the above all nonlinear estimates, we obtian the desired estimate for $u_3$ and complete the proof of Proposition \ref{prop}. 
\end{proof}

\section{Proof of the main result}\label{sec:pf}
Now, we are in a position to present the proof of our main result.
\begin{proof}
We divide the proof into two steps. 
In the first step, we introduce the approximation system and close the global analytic a priori estimates for the approximate solutions by making use of the calculations performed in the previous section.
In the second step, we pass to the limit for the approximate solutions and construct the analytic solution.

\noindent
{\it Step 1. Uniform boundedness of the approximated solutions.} 
Let $\chi \in C_c^{\infty}(\mathbb{R}^n)$ satisfy $0 \leq \chi \leq 1$ and
\begin{align}
    \chi(\xi) = 
    \begin{cases}
        1 & (|\xi| \leq 1), \\
        0 & (|\xi| \geq 2).
    \end{cases}
\end{align}
For $N \in \mathbb{N}$, we define 
\begin{align}
    \mathcal{P}_N f:= \mathscr{F}^{-1}\lp{\chi(N^{-1}\xi)\widehat{f}(\xi)}, \qquad f \in \mathscr{S}'(\mathbb{R}^3).
\end{align}
Note that there holds
\begin{align}
    \n{\mathcal{P}_Nf}_{\fB_p^{s,\sigma}} \leq C^*\n{f}_{\fB_p^{s,\sigma}}
\end{align}
for all $f \in \fB_p^{s,\sigma}(\mathbb{R}^3)$ with the constant $C^*=\n{\mathscr{F}^{-1}[\chi]}_{L^1}$.
In what follows, we define 
\begin{align}
    \lambda:=10C_0^*, \qquad r:= \frac{1}{10C_0^*},
\end{align}
where $C_0^*:=C^*C_0$ and $C_0$ denote the positive constant appearing in Proposition \ref{prop}.
Let us choose initial data $u_0=(u_{0,{\rm h}},u_{0,3})$ so that 
\begin{align}
    &
    3C_0^*\n{u_{0,{\rm h}}}_{\mathcal{D}_{p,\theta}^{\rm h}}
    \exp\lp{\frac{2C_0^*+4 \theta (C_0^*)^2}{\theta}\n{u_{0,3}}_{D^{\rm v}_{p,\theta}}}
    \leq 
    \sp{\frac{\log 2}{2\lambda C_0^*}}^{\frac{1}{\theta}},
    \\
    &
    3(C_0^*)^2\n{u_{0,{\rm h}}}_{\mathcal{D}_{p,\theta}^{\rm h}}
    \exp\lp{4(C_0^*)^2\n{u_{0,3}}_{D^{\rm v}_{p,\theta}}}
    \leq{}
    \frac{\log 2}{2\lambda}.
\end{align}

Let us consider the approximation system 
\begin{align}\label{eq:app}
    \begin{cases}
        \partial_t u^N - \Deltah u^N + \mathcal{P}_N \lp{(u^N \cdot \nabla)u^N} + \nabla P^N = 0, \quad & t>0, x \in \mathbb{R}^3, \\
        \div u^N = 0, & t\geq 0, x \in \mathbb{R}^3, \\
        u^N(0,x) = \mathcal{P}_Nu_0(x), & x \in \mathbb{R}^3,
    \end{cases}
\end{align}
where the pressure $P^N$ is given by 
\begin{align}
    P^N=\mathcal{P}_N\sum_{k,\ell=1}^3 \partial_{x_k}\partial_{x_\ell}(-\Delta)^{-1}\sp{u_k^Nu_{\ell}^N}.
\end{align}
We then, see that there exists a solution $u^N$ in the class
\begin{align}
    \partial_t^{\alpha}
    u^N \in BC([0,T_N');\fB_p^{\frac{2}{p}-1,\frac{1}{p}}(\mathbb{R}^3)) \cap L^1(0,T_N';\fB_p^{\frac{2}{p}+1,\frac{1}{p}}(\mathbb{R}^3)),
    \quad
    \alpha \in \mathbb{N} \cup \{ 0 \},
\end{align}
where $T_N'=T_N'(p,\theta,u_0)$ denotes the maximal existence time.
Since $\supp\widehat{u^N}(t) \subset \Mp{\xi \in \R^3\ ;\ |\xi| \leq 2N}$, the approximated solution $u^N$ is real analytic for $x \in \mathbb{R}^3$ in the sense that 
\begin{align}
    &
    \sum_{\beta \in (\mathbb{N} \cup \{ 0\})^3}
    \frac{R^{|\beta|}}{\beta!}
    \n{\partial_x^{\beta}u^N}_{{L^{\infty}}(0,T_N;\fB_p^{\frac{2}{p}-1,\frac{1}{p}}) \cap L^1(0,T_N';\fB_p^{\frac{2}{p}+1,\frac{1}{p}})}
    \\
    &\quad
    \leq 
    \sum_{\beta \in (\mathbb{N} \cup \{ 0\})^3}
    \frac{(CNR)^{|\beta|}}{\beta!}
    \n{u^N}_{{L^{\infty}}(0,T_N;\fB_p^{\frac{2}{p}-1,\frac{1}{p}}) \cap L^1(0,T_N';\fB_p^{\frac{2}{p}+1,\frac{1}{p}})}
    \\
    &\quad 
    \leq 
    e^{CNR}
    \n{u^N}_{{L^{\infty}}(0,T_N;\fB_p^{\frac{2}{p}-1,\frac{1}{p}}) \cap L^1(0,T_N';\fB_p^{\frac{2}{p}+1,\frac{1}{p}})}
    <
    \infty
\end{align}
for all $R>0$.

By the Peano theorem, there exists a solution $f^N_M$ to the ordinary differential equation
\begin{align}\label{ODEN}
    \begin{dcases}
    \begin{aligned}
    \frac{d}{dt}f^N_M(t) 
    = {}&
    \sum_{
    \substack{0 \leq \alpha \leq M \\ \beta \in (\mathbb{N}\cup \{ 0 \})^3}
    }
    \frac{r^{\alpha + |\beta_{\rm h}|}t^{\alpha + \frac{|\beta_{\rm h}|}{2}}}{\alpha!\beta!}
    \n{
    \partial_t^{\alpha}
    \partial_{x}^{\beta}
    u^N_{\rm h} (t)
    }_{\fB^{\frac{2}{p}+1,\frac{1}{p}}_p}
    e^{-\lambda \beta_3 f^N_M(t)} 
    \\
    &+
    \sum_{
    \substack{0 \leq \alpha \leq M \\ \beta_{\rm h} \in (\mathbb{N}\cup \{ 0 \})^2}
    }
    \frac{
    r^{\alpha + |\beta_{\rm h}|}t^{\alpha + \frac{|\beta_{\rm h}|}{2}}
    \rho_0^{\beta_3} 
    }
    {
    \alpha!\beta_{\rm h}!
    }
    \n{
    \partial_t^{\alpha}
    \partial_{x_{\rm h}}^{\beta_{\rm h}}
    u^N_3(t)}_{\fB^{\frac{2}{p},\frac{1}{p}}_p},
    \end{aligned}
    \\
    f^N_M(0) = 0
    \end{dcases}
\end{align}
on a time interval $[0,T_N)$ with some $0<T_N\leq T_N'$.
Moreover, we see that 
\begin{align}
    \n{\uh^N}_{\mathcal{H}_{p,\theta;r;\lambda f^N_M}^M(0,T_N)}< \infty,
    \qquad
    \n{u^N_3}_{V_{p,\theta;r}^M(0,T_N)}< \infty.
\end{align}
In order to show the global existence and a priori bound for the approximated solutions, we define a time
\begin{align}
    T_N^*
    :=
    \sup
    \Mp{
    T \in (0,T_N)
    \ ;\ 
    \begin{aligned}
    &
    \n{\uh^N}_{\mathcal{H}_{p,\theta;r;\lambda f^N_M}^M(0,T)} 
    \\
    &\quad
    \leq 
    3C_0^* 
    \n{u_{0,{\rm h}}}_{\mathcal{D}_{p,\theta}^{\rm h}}
    \exp\lp{(2C_0^*)^2\n{u_{0,3}}_{D^{\rm v}_{p,\theta}}}
    \end{aligned}
    }
\end{align}
and suppose to contrary that 
\begin{align}
    T_N^* < \infty.
\end{align}
Let $0<T<T_N^*$.
For the estimate of $u^N_3$, we see by Proposition \ref{prop} and the definition of $T_N^*$ that 
\begin{align}
    \n{u^N_3}_{\mathcal{V}^M_{p,\theta;r}(0,T)}
    \leq{}& 
    C_0^*   
    \n{u_{0,3}}_{D^{\rm v}_{p,\theta}} 
    +
    C_0^*r
    \n{u^N_3}_{\mathcal{V}^M_{p,\theta;r}(0,T)}
    \\
    &
    +
    \frac{1}{10} 
    \n{u^N_3}_{\mathcal{V}^M_{p,\theta;r}(0,T)}
    +
    C_0^* 
    \sp{ \frac{1}{10C_0^*}}^2,
\end{align}
which yields
\begin{align}\label{est:u3N}
    \n{u^N_3}_{\mathcal{V}^M_{p,\theta;r}(0,T)}
    \leq{}&
    2
    C_0^*   
    \n{u_{0,3}}_{D^{\rm v}_{p,\theta}} 
    +
    \frac{1}{50C_0^*}.
\end{align}
We remark that the apriori estimate in Proposition \ref{prop} hold for $u^N$ with $C_0$ replaced by $C_0^*$.
Then, we have the uniform estimate of $f^N_M$:
\begin{align}
    \sup_{t \in [0,T]} &f^N_M(t)
    \leq{}
    C_0^*
    \n{u^N_3}_{\mathcal{V}^M_{p,\theta}(0,T)}^{1-\theta}
    \n{\uh^N}_{\mathcal{H}^M_{p,\theta;r;\lambda f^N_M}(0,T)}^{\theta}
    +
    \n{\uh^N}_{\mathcal{H}^M_{p,\theta;r;\lambda f^N_M}(0,T)}
    \\
    \leq{}&
    C_0^*\sp{2C_0^*\n{u_{0,3}}_{D^{\rm v}_{p,\theta}}+\frac{1}{50C_0^*}}^{1-\theta}
    \sp{3C_0^*\n{u_{0,{\rm h}}}_{\mathcal{D}_{p,\theta}^{\rm h}}}^{\theta}
    \exp\lp{4\theta (C_0^*)^2\n{u_{0,3}}_{D^{\rm v}_{p,\theta}}}
    \\
    &
    +
    3C_0^*\n{u_{0,{\rm h}}}_{\mathcal{D}_{p,\theta}^{\rm h}}\exp\lp{4(C_0^*)^2\n{u_{0,3}}_{D^{\rm v}_{p,\theta}}}
    \\
    \leq{}&
    C_0^*
    \sp{3C_0^*\n{u_{0,{\rm h}}}_{\mathcal{D}_{p,\theta}^{\rm h}}}^{\theta}
    \exp\lp{\sp{2C_0^*+4 \theta (C_0^*)^2}\n{u_{0,3}}_{D^{\rm v}_{p,\theta}}}\label{est:f}
    \\
    &
    +
    3(C_0^*)^2\n{u_{0,{\rm h}}}_{\mathcal{D}_{p,\theta}^{\rm h}}\exp\lp{4(C_0^*)^2\n{u_{0,3}}_{D^{\rm v}_{p,\theta}}}
    \\
    \leq{}&
    \frac{\log 2}{\lambda} .
\end{align}
For the estimate of the horizontal velocity field, there holds by Proposition \ref{prop} that
\begin{align}
    &
    \n{\uh^N}_{\mathcal{H}_{p,\theta;r;\lambda f^N_M}^M(0,T)}
    \leq{}
    C_0^*\n{u_{0,{\rm h}}}_{\mathcal{D}_{p,\theta}^{\rm h}}
    \\
    &\quad
    +
    \sp{\frac{1}{10}+\frac{1}{10}}
    \n{\uh^N}_{\mathcal{H}_{p,\theta;r;\lambda f^N_M}^M(0,T)}
    +
    \frac{1}{10}
    \n{\uh^N}_{\mathcal{H}_{p,\theta;r;\lambda f^N_M}^M(0,T)}
    \\
    &\quad
    +
    \frac{1}{10}
    \exp
    \lp{\lambda \frac{\log 2}{\lambda}}
    \n{\uh^N}_{\mathcal{H}_{p,\theta;r;\lambda f^N_M}^M(0,T)}
    \\
    &\quad
    +  
    C_0^*  
    \int_0^T 
    \sum_{
    \substack{0 \leq \alpha \leq M \\ \betah \in (\mathbb{N}\cup \{ 0 \})^2}
    }
    \frac{r^{\alpha+|\betah|} t^{\alpha+\f{|\betah|}{2}} \rho_0^{\beta_3}}{\alpha!\betah!}
    \n{
    \partial_{t}^{\alpha}  
    \partial_{\xh}^{\betah}  
    u^N_3(t)
    }_{\fB^{\f{2}{p}+1,\f{1}{p}}_p}
    \n{\uh^N}_{\mathcal{H}^M_{p,\theta;r;\lambda f^N_M}(0,t)}  dt,
\end{align}
which implies
\begin{align}
    &
    \n{\uh^N}_{\mathcal{H}_{p,\theta;r;\lambda f^N_M}^M(0,T)}
    \leq{}
    2C_0^*\n{u_{0,{\rm h}}}_{\mathcal{D}_{p,\theta}^{\rm h}}
    \\
    &\quad
    +  
    2C_0^*  
    \int_0^T 
    \sum_{
    \substack{0 \leq \alpha \leq M \\ \betah \in (\mathbb{N}\cup \{ 0 \})^2}
    }
    \frac{r^{\alpha+|\betah|} t^{\alpha+\f{|\betah|}{2}} \rho_0^{\beta_3}}{\alpha!\betah!}
    \n{
    \partial_{t}^{\alpha}  
    \partial_{\xh}^{\betah}  
    u^N_3(t)
    }_{\fB^{\f{2}{p}+1,\f{1}{p}}_p}
    \n{\uh^N}_{\mathcal{H}^M_{p,\theta;r;\lambda f^N_M}(0,t)}  dt.
\end{align}
By virtue of the Gronwall lemma, we have
\begin{align}\label{est:uhN}
    \begin{split}
    \n{\uh^N}_{\mathcal{H}_{p,\theta;r;\lambda f^N_M}^M(0,T^*_N)}
    \leq{}&
    2C_0^*\n{u_{0,{\rm h}}}_{\mathcal{D}_{p,\theta}^{\rm h}}
    \exp\lp{2C_0^*\n{u^N_3}_{\mathcal{V}_{p,\theta;r}^M(0,T^*_N)}}
    \\
    \leq{}&
    2e^{\frac{1}{25}}C_0^*\n{u_{0,{\rm h}}}_{\mathcal{D}_{p,\theta}^{\rm h}}
    \exp\lp{(2C_0^*)^2\n{u_{0,3}}_{D^{\rm v}_{p,\theta}}}.
    \end{split}
\end{align}
Since $2e^{\frac{1}{25}} < 3$,
the above estimate contradicts the definition of $T_N^*$ and thus we have $T_N^*=T_N=\infty$.

Hence, combining \eqref{est:u3N}, \eqref{est:f}, and \eqref{est:uhN}, we have 
\begin{align}\label{est:unif}
    \begin{split}
    \n{\uh^N}_{\mathcal{H}_{p,\theta;r;\log 2}^M(0,\infty)}
    \leq{}&
    C\n{u_{0,{\rm h}}}_{\mathcal{D}_{p,\theta}^{\rm h}}
    \exp\lp{C\n{u_{0,3}}_{D^{\rm v}_{p,\theta}}},
    \\
    \n{u^N_3}_{\mathcal{V}^M_{p,\theta;r}(0,\infty)}
    \leq{}&
    C  
    \n{u_{0,3}}_{D^{\rm v}_{p,\theta}} 
    +
    C.
    \end{split}
\end{align}

\noindent
{\it Step 2. Existence of an analytic solution.}
For $m \in \mathbb{N}$, let $\psi_m \in C_c(\mathbb{R}^3)$ satisfy $0 \leq \psi_m \leq 1$, $\supp \psi_m \subset \{ x \in \mathbb{R}^3\ ;\ |x| \leq m+1\}$, and $\psi_m(x) = 1$ for all $x \in \mathbb{R}^3$ with $|x| \leq m$.
From the previous step, it follows that 
\begin{align}
    &
    \n{t^{\frac{1}{2}}u^N}_{{L^{\infty}}(0,\infty;\fB_p^{\frac{2}{p},\frac{1}{p}})}
    +
    \n{t^{\frac{3}{2}}\partial_t u^N}_{L^{\infty}(0,\infty;\fB_p^{\frac{2}{p},\frac{1}{p}})}
    \\
    &\quad
    \leq
    C
    \sum_{|\betah|=1}
    \sp{
    \n{t^{\frac{1}{2}}\partial_{\xh}^{\betah}u^N}_{{L^{\infty}}(0,\infty;\fB_p^{\frac{2}{p}-1,\frac{1}{p}})}
    +
    \n{t^{\frac{3}{2}}\partial_t \partial_{\xh}^{\betah}u^N}_{L^{\infty}(0,\infty;\fB_p^{\frac{2}{p}-1,\frac{1}{p}})}
    }\\
    &\quad
    \leq 
    C
    \sp{
    \n{\uh^N}_{\mathcal{H}_{p,\theta;r;\lambda f^N_M}^M(0,\infty)}
    +
    \n{u_3^N}_{\mathcal{V}_{p,\theta;r}^M(0,\infty)}
    }\\
    &\quad
    \leq 
    C\n{u_{0,{\rm h}}}_{\mathcal{D}_{p,\theta}^{\rm h}}
    \exp\lp{C\n{u_{0,3}}_{D^{\rm v}_{p,\theta}}}
    +
    C\n{u_{0,3}}_{D^{\rm v}_{p,\theta}} 
    + 
    C,
\end{align}
which and the second assertion of Lemma \ref{lemm:embedding} imply 
\begin{align}\label{unif-bdd}
    \sup_{N \in \mathbb{N}}
    \sp{
    \n{t^{\frac{1}{2}}\psi_mu^N}_{{L^{\infty}}(0,\infty;\ifB_p^{\frac{2}{p},\frac{1}{p}})}
    +
    \n{t^{\frac{3}{2}}\partial_t\sp{\psi_m u^N}}_{L^{\infty}(0,\infty;\ifB_p^{\frac{2}{p},\frac{1}{p}})}
    }
    <
    \infty.
\end{align}
Let $0<\varepsilon<1/p$.
Seeing by \eqref{unif-bdd} that $\{\psi_mu^N\}_{N=1}^{\infty}$ is uniform bounded and equicontinuous in $\ifB_p^{\frac{2}{p},\frac{1}{p}}(\mathbb{R}^3)$ and using of Lemma \ref{lemm:comp} via the relation $\psi_m\psi_{m+1}=\psi_m$, we make use of the mixture of the Ascoli--Arzelà theorem and the Cantor diagonal process to extract a subsequence $\{u^{N_k}\}_{k=1}^{\infty}$ such that for any $m \in \mathbb{N}$, we find a $\bar{u}^m \in C((0,\infty);\ifB_p^{\frac{2}{p}-\varepsilon,\frac{1}{p}-\varepsilon}(\mathbb{R}^3))$ satisfying that for every $0<t<t'<\infty$,
\begin{align}
    \lim_{k \to \infty}
    \sup_{t \leq \tau \leq t'}\n{\psi_mu^{N_k}(\tau) - \bar{u}^m(\tau)}_{\ifB_p^{\frac{2}{p}-\varepsilon,\frac{1}{p}-\varepsilon}}=0.
\end{align}
Here, since $\psi_m=\psi_m\psi_{m'}$ yields $\bar{u}^m=\psi_{m}\bar{u}^{m'}$ for all positive integers $m' > m$, the function $\bar{u}^{m'}$ is independent of $m'>m$ on the region $\{ x \in \mathbb{R}^3 \ ;\ |x| \leq m \}$ up to null sets.
Hence, we see that there exists an almost everywhere pointwise limit $u^m(t,x) \to u(t,x)$ $(m \to \infty)$ and also see that for every $m \in \mathbb{N}$ and $0<t<t'<\infty$, 
\begin{align}
    \lim_{k \to \infty}
    \sup_{t \leq \tau \leq t'}\n{\psi_m u^{N_k}(\tau) - \psi_m u(\tau)}_{\ifB_p^{\frac{2}{p}-\varepsilon,\frac{1}{p}-\varepsilon}}=0.
\end{align}
Thus, for every $\mathbb{R}^3$-valued test function $\psi \in C_c^{\infty}(\mathbb{R}^3)$ with $\div \psi = 0$, it holds for each $0<t<t'<\infty$  that 
\begin{align}
    &
    \lr{u^{N_k}(t_*),\psi} \to \lr{u(t_*),\psi}, \quad
    t_*=t,t',
    \\
    &
    \int_t^{t'}\lr{\Deltah u^{N_k}(\tau),\psi}d\tau 
    \to 
    \int_t^{t'}\lr{\Deltah u(\tau),\psi} dt,
    \\
    &
    \int_t^{t'}\lr{(u^{N_k}(\tau)\cdot \nabla)u^{N_k}(\tau),\psi}d\tau 
    \to 
    \int_t^{t'}\lr{(u(\tau)\cdot \nabla)u(\tau),\psi} dt
\end{align}
as $k \to \infty$.
From this, we easily see that $u$ solve \eqref{eq:ANS} in the distributional sense for $t>0$ by passing $N=N_k \to \infty$ in the weak form of the approximation system \eqref{eq:app}.
On the other hand, since 
\begin{align}
    \sup_{k \in \mathbb{N}}
    \sp{
    \n{u^{N_k}}_{{L^{\infty}}(0,\infty;\fB_p^{\frac{2}{p}-1,\frac{1}{p}})}
    +
    \n{u^{N_k}}_{L^1(0,\infty;\fB_p^{\frac{2}{p}+1,\frac{1}{p}})}
    }
    <\infty,
\end{align}
Lemma \ref{lemm:Fatou-time} and the similar argument as in \cite{Bah-Che-Dan-11}*{pp.443--444} imply 
\begin{align}
    u \in {L^{\infty}}(0,\infty;\fB_p^{\frac{2}{p}-1,\frac{1}{p}}(\mathbb{R}^3)) \cap L^1(0,\infty;\fB_p^{\frac{2}{p}+1,\frac{1}{p}}(\mathbb{R}^3)), 
\end{align}
and 
\begin{align}
    \n{u}_{L^{\infty}(0,\infty;\fB_p^{\frac{2}{p}-1,\frac{1}{p}})\cap L^1(0,\infty;\fB_p^{\frac{2}{p}+1,\frac{1}{p}})}
    \leq 
    C
    \liminf_{\ell \to \infty}
    \n{u^{N_{k_\ell}}}_{{L^{\infty}}(0,\infty;\fB_p^{\frac{2}{p}-1,\frac{1}{p}})\cap L^1(0,\infty;\fB_p^{\frac{2}{p}+1,\frac{1}{p}})}
\end{align}
for some subsequence $\{u^{N_{k_\ell}}\}_{\ell=1}^{\infty}$.
We notice that the above fact implies $u(0)=u_0$.
Repeating the same procedure for $\partial_{t,x}^{\alpha,\beta}\uh^N$ and $\partial_{t,\xh}^{\alpha,\betah}u_N$ with $0\leq \alpha \leq M$ and $\beta \in (\mathbb{N} \cup \{0 \})^3$, we see that the solution $u$ to \eqref{eq:ANS} constructed above satisfies
\begin{align}
    &
    \n{u_{\rm h}}_{\mathcal{H}_{p,\theta;r;\log 2}^M(0,\infty)} 
    \leq 
    C
    \n{u_{0,{\rm h}}}_{\mathcal{D}_{p,\theta}^{\rm h}}
    \exp\lp{C\n{u_{0,3}}_{D^{\rm v}_{p,\theta}}}, 
    \\
    &
    \n{u_3}_{\mathcal{V}^M_{p,\theta;r}(0,\infty)}
    \leq{}
    C
    \n{u_{0,3}}_{D^{\rm v}_{p,\theta}} 
    +
    C.
\end{align}
Since $M$ is arbitrary, we complete the construction of the space-time analytic solution by letting $M \to \infty$.
For the uniqueness of solutions with $1 \leq p \leq 2$, our solution belongs to the solution class considered in \cite{Pai05} and thus the uniqueness lies on the result in \cite{Pai05}.
This completes the proof.
\end{proof}

\vspace{4mm}
{\bf{Acknowledgements.}} 
M. Fujii was supported by JSPS KAKENHI, Grant Number JP25K17279. 
Y. Li was supported by National Natural Science Foundation of China (12571228), Natural Science Foundation of Anhui Province (2408085MA018), Natural Science Research Project in Universities of Anhui Province (2024AH050055); he sincerely thanks Professor Yongzhong Sun for patient guidance and encouragement.

\vspace{4mm}
{\bf{Data Availability.}} Data sharing is not applicable to this article as no datasets were generated or analyzed
during the current study.
\vspace{4mm}

{\bf{Conflicts of interest.}} All authors certify that there are no conflicts of interest for this work.




\end{document}